\documentclass[a4paper,dvipsnames]{article}

\usepackage{amsthm,amssymb,amsmath,enumerate,graphicx,epsf,bm,caption,framed,relsize}
\usepackage{authblk}
\usepackage[greek,english]{babel}  
\usepackage[percent]{overpic}
\usepackage{epstopdf}
\usepackage{bbm}

\newcommand{\COLORON}{0}
\newcommand{\NOTESON}{0}
\newcommand{\Debug}{0}
\usepackage[usenames]{color} 
\usepackage{amsthm,amssymb,amsmath,bbm,enumerate,graphicx,epsf,stmaryrd}
\usepackage[bookmarks, colorlinks=false, breaklinks=true]{hyperref} 

\newcommand{\comment}[1]{}
\newcommand{\COMMENT}[1]{}

\definecolor{darkgray}{rgb}{0.3,0.3,0.3}
\newcommand{\defi}[1]{{\color{darkgray}\emph{#1}}}

\comment{
	\begin{lemma}\label{}	
\end{lemma}
\begin{proof}

\end{proof}

\begin{theorem}\label{}
\end{theorem} 
\begin{proof} 	

\end{proof}

}

\newtheorem{proposition}{Proposition}[section]
\newtheorem{definition}[proposition]{Definition}
\newtheorem{theorem}[proposition]{Theorem}
\newtheorem{corollary}[proposition]{Corollary}

\newtheorem{lemma}[proposition]{Lemma}

\newtheorem{examp}[proposition]{Example}

\newcommand{\FIG}{0}

\ifnum \NOTESON = 1 \newcommand{\note}[1]{ 

\hspace*{-30pt}
	{\color{blue}  NOTE: \color{Turquoise}{\small  \tt \begin{minipage}[c]{1.1\textwidth}  #1 \end{minipage} \ignorespacesafterend }} 
	
	}
\else \newcommand{\note}[1]{} \fi

\newcommand{\afsubm}[1]{ \ifnum \Debug = 1 {\mymargin{#1}}
\fi} 

\ifnum \Debug = 1 \newcommand{\sss}{\ensuremath{\color{red} \bowtie \bowtie \bowtie\ }}
\else \newcommand{\sss}{} \fi

\ifnum \FIG = 1 \newcommand{\fig}[1]{Figure ``{#1}''}
\else \newcommand{\fig}[1]{Figure~\ref{#1}} \fi

\ifnum \FIG = 1 
\else  \fi

\ifnum \Debug = 1 \usepackage[notref,notcite]{showkeys}
\fi

\ifnum \COLORON = 0 \renewcommand{\color}[1]{}
\fi

\newcommand{\N}{\ensuremath{\mathbb N}}
\newcommand{\R}{\ensuremath{\mathbb R}}
\newcommand{\C}{\ensuremath{\mathbb C}}
\newcommand{\Z}{\ensuremath{\mathbb Z}}

\newcommand{\cc}{\ensuremath{\mathcal C}}

\newcommand{\ce}{\ensuremath{\mathcal E}}

\newcommand{\cp}{\ensuremath{\mathcal P}}

\newcommand{\cs}{\ensuremath{\mathcal S}}
\newcommand{\cgr}{\ensuremath{\mathcal R}}

\newcommand{\oo}{\ensuremath{\omega}}
\newcommand{\OO}{\ensuremath{\Omega}}

\newcommand{\eps}{\ensuremath{\epsilon}}

\newcommand{\sm}{\backslash}

\makeatletter
\DeclareRobustCommand{\cev}[1]{%
  \mathpalette\do@cev{#1}%
}
\newcommand{\do@cev}[2]{%
  \fix@cev{#1}{+}%
  \reflectbox{$\m@th#1\vec{\reflectbox{$\fix@cev{#1}{-}\m@th#1#2\fix@cev{#1}{+}$}}$}%
  \fix@cev{#1}{-}%
}
\newcommand{\fix@cev}[2]{%
  \ifx#1\displaystyle
    \mkern#23mu
  \else
    \ifx#1\textstyle
      \mkern#23mu
    \else
      \ifx#1\scriptstyle
        \mkern#22mu
      \else
        \mkern#22mu
      \fi
    \fi
  \fi
}

\makeatother

\newcommand{\nin}{\ensuremath{{n\in\N}}}

\newcommand{\pth}[2]{\ensuremath{#1}\text{--}\ensuremath{#2}~path}

\newcommand{\g}{\ensuremath{G\ }}
\newcommand{\G}{\ensuremath{G}}

\newcommand{\are}{\vec{e}}

\newcommand{\Ex}{\mathbb E}
\renewcommand{\Pr}{\mathbb{P}}

\newcommand{\Lr}[1]{Lemma~\ref{#1}}
\newcommand{\Lrs}[1]{Lemmas~\ref{#1}}
\newcommand{\Tr}[1]{Theorem~\ref{#1}}

\newcommand{\Sr}[1]{Section~\ref{#1}}

\newcommand{\Prr}[1]{Pro\-position~\ref{#1}}

\newcommand{\Cr}[1]{Corollary~\ref{#1}}

\newcommand{\Dr}[1]{De\-fi\-nition~\ref{#1}}

\newcommand{\Cg}{Cayley graph}

\renewcommand{\iff}{if and only if}
\newcommand{\fe}{for every}
\newcommand{\Fe}{For every}

\newcommand{\st}{such that}

\newcommand{\ti}{there is}

\newcommand{\obda}{without loss of generality}

\newcommand{\wrt}{with respect to}

\newcommand{\labequ}[2]{ \begin{equation} \label{#1} #2 \end{equation} } 
\newcommand{\labsplitequ}[2]{ \begin{equation} \begin{split}
 \label{#1} #2 \end{split} \end{equation} } 
\newcommand{\labtequ}[2]{
 \begin{equation} \label{#1} 	\begin{minipage}[c]{0.9\textwidth}  #2 \end{minipage} \ignorespacesafterend \end{equation} }

\newcommand{\mymargin}[1]{
 \ifnum \Debug = 1
  \marginpar{%
    \begin{minipage}{\marginparwidth}\small%
      \begin{flushleft}%
        {\color{blue}#1}%
      \end{flushleft}%
   \end{minipage}%
  }%
 \fi
}%

\newcommand{\mySection}[2]{}

\usepackage{xcolor}

\definecolor{myred}{rgb}{0.90196078,0,0}
\definecolor{mygreen}{rgb}{0,0.90196078,0}
\definecolor{myblue}{rgb}{0,0,0.90196078}

\usepackage{mathtools}

\DeclarePairedDelimiter\abs{\lvert}{\rvert}

\RequirePackage{latexsym} \RequirePackage{amsthm}
\RequirePackage{amsmath}
\usepackage{hyperref}

\usepackage{enumerate}
\RequirePackage{amssymb} \RequirePackage{makeidx}
\usepackage{dsfont}
\usepackage{mathrsfs}

\newtheorem{remark}[proposition]{Remark}

\newcommand{\myremark}[1]{\ifnum \Debug = 1 \tiny #1 \fi}
\newcommand{\ABP}{Aizenman-Newman-Barsky property}
\newcommand{\nnm}{nearest-neighbour model}
\newcommand{\lrm}{long-range model}
\newcommand{\Nnm}{Nearest-neighbour model}
\newcommand{\Lrm}{Long-range model}

\newcommand{\pint}{\cp-interface}
\newcommand{\mpint}{multi-\cp-interface}

\newcommand{\scv}{interface} 
\newcommand{\smc}{multi-interface} 
\newcommand{\ssp}{separating strip}
\newcommand{\sms}{separating multi-strip}

\newcommand{\sboxt}{\mathsmaller {\mathsmaller \boxtimes}}
\newcommand{\Ls}{\mathbb{L}^d_\sboxt}
\newcommand{\CS}{\mathcal S}

\newcommand{\Zs}{\ensuremath{\Z^{2}}}
\newcommand{\Zsd}{\ensuremath{\Z^{2*}}}
\newcommand{\vx}{\ensuremath{\bm{x}}}
\newcommand{\sgn}[1]{\ensuremath{(-1)^{#1}}}
\newcommand{\Zn}{\mathcal{Z}}

\newcommand{\ar}[1]{\vec{#1}}
\newcommand{\ard}[2]{\ensuremath{\vec{#1^{#2}}}}
\newcommand{\ra}[1]{\cev{#1}}
\newcommand{\dar}[1]{\overset\leftrightarrow{#1}}
\newcommand{\pb}{proper bipartition}
\newcommand{\qtl}{planar quasi-transitive lattice}
\newcommand{\Wthm}{Weierstrass' \Tr{thmWei}}

\usepackage{authblk}

\title{Analyticity results in Bernoulli Percolation}

\author[1]{Agelos Georgakopoulos}
\author[2]{Christoforos Panagiotis}
\affil[1,2]{{Mathematics Institute}\\
	{University of Warwick}\\
	{CV4 7AL, UK}\thanks{Supported by the European Research Council (ERC) under the European Union's Horizon 2020 research and innovation programme (grant agreement No 639046).}\\}
	
\begin{document}
\maketitle

\date{}

\newcommand{\MS}{\mathcal {MS}}
\newcommand{\HRT}{Hardy--Ramanujan formula}
\newcommand{\RL}{ring \Lr{ring}}
\newcommand{\DD}{\ensuremath{\mathbb D}}
\newcommand{\pcs}{\ensuremath{\dot{p}_c}}
\newcommand{\per}{\ensuremath{\partial_{int} C(o)}}
\newcommand{\bou}{\ensuremath{\partial_{ext} C(o)}}
\newcommand{\IEP}{Inclusion-Exclusion Principle}

\begin{abstract}
We prove that for Bernoulli percolation on $\mathbb{Z}^d$, $d\geq 2$, the percolation density is an analytic function of the parameter in the supercritical interval. For this we introduce some techniques that have further implications. In particular, we prove that the susceptibility is analytic in the subcritical interval for all transitive short- or long-range models, and that $p_c^{bond} <1/2$ for certain families of triangulations for which Benjamini \& Schramm conjectured that $p_c^{site} \leq 1/2$. 
\end{abstract}


\section{Introduction}

We prove that for Bernoulli percolation on the cubic lattice $\mathbb{Z}^d$, $d\geq 2$, the probability $\theta_o(p)$ that the origin $o$ is in an infinity cluster is an analytic function of the parameter $p$ in the supercritical interval $p\in (p_c,1]$. This answers a question of Kesten \cite{Ke81}. Our techniques imply analiticity results for other functions and setups.

\subsection{Background and motivation}
In Bernoulli bond percolation, each edge of a connected, locally finite graph $G$ is either deleted (vacant) or retained (occupied)  independently at random with retention probability $p\in [0,1]$ to obtain a random subgraph $\omega$ of $G$. Connected components of $\omega$ are referred to as \defi{clusters}. Percolation theory is primarily concerned with the structure of retained clusters in $\omega$, how this structure changes as the parameter $p$ is varied, and in particular, the {\em phase-transitions} where a small change in $p$ imposes a dramatic change of this structure.
A principal quantity of interest is  the \defi{percolation density} $\theta=\theta_o(p)$, defined as the probability that the cluster $C(o)$ of a vertex $o$ is infinite. The   \defi{percolation threshold} is
$$p_c:= \inf_{} \{p\leq 1 \mid \theta_o(p)>0\},$$
and for quasi-transitive graphs it is strictly between zero and one \cite{DCGRSY}. 

The phase-transition at $p_c$ is the most fundamental and well-known result of percolation theory. It immediately raises the question of whether there are phase-transitions at other values of $p$ ---and how to define them. The obvious approach to hunting for  phase-transitions is to consider functions, such as $\theta_o(p)$, describing the macroscopic behaviour of the model, and study their smoothness in the interval $p\in (0,1)$. 

In this paper we prove that several functions studied in percolation theory are analytic functions of the parameter $p$. 
We consider Bernoulli bond percolation on a variety of graphs, as well as general \lrm s (defined in \Sr{LRM}) preserved by a transitive group action. 

\medskip
Perhaps the first occurrence of questions of smoothness in percolation theory dates back to the work of Sykes \& Essam \cite{SykesEssam}. Trying to compute the value of $p_c$ for bond percolation on the square lattice $\Z^2$, Sykes \& Essam obtained that the free energy (aka.\ mean number of clusters per vertex) $\kappa(p):=\Ex_p(|C_o|^{-1})$ satisfies the functional equation $\kappa(p)=\kappa(1-p)+\phi(p)$ for some polynomial $\phi(p)$. Under the assumption of smoothness of $\kappa$ for every value of the parameter $p$ other than $p_c$, at which it is conjectured that $\kappa$ has a singularity, they obtained that $p_c=1/2$ due to the symmetry of the functional equation around $1/2$. Their work generated considerable interest, and a lot of the early work in percolation was focused on the smoothness of functions like $\kappa$ and the \defi{susceptibility} $\chi$ ---i.e.\ the expected number of vertices in the cluster of a fixed vertex $o$--- that describe the macroscopic behaviour of its clusters. Kunz \& Souillard \cite{KunSou} proved that $\kappa$ is analytic for small enough $p$ for percolation on $\Z^d, d>1$. Grimmett \cite{GriDif} proved that $\kappa$ is $C^\infty$ for $d=2$. A breakthrough was made by Kesten \cite{Ke81}, who proved that $\kappa$ and $\chi$ are analytic for $p\in [0,p_c)$ for all $d$.\footnote{The threshold $p_T$ in Kesten's original formulation was later shown to coincide with $p_c$ by Aizenman \& Barsky  \cite{AizBar}.} (Despite all the efforts, the argument of Sykes \& Essam has never been made rigorous, and all proofs of the fact that $p_c=1/2$ when $d=2$ use different methods, see e.g.\ \cite{KestenCritical,BoRioShor}.)

Except for the special case of $\kappa$ on $\Z^2$ (and other planar lattices), smoothness results are harder to obtain in the supercritical interval $(p_c,1]$, partly because the cluster size distribution $P_n:=\Pr_p(|C(o)| = n)$ has an exponential tail below $p_c$ (\Sr{sec ABP}) but not above $p_c$ \cite{AiDeSoLow}. Still, it is known that $\theta, \kappa$, and the `truncation' $\chi^f$ of $\chi$ are infinitely differentiable for $p\in (p_c,1]$ on $\Z^d$ (see \cite{ChChNeBer} or \cite[\S 8.7]{Grimmett} and references therein). It is a well-known open problem, dating back to \cite{Ke81} at least, and appearing in several textbooks (\cite[Problem 6]{KestenBook},\cite{GrimmettDisordered,Grimmett}), whether $\theta$ is analytic for $p\in (p_c,1]$. Partial progress was made by Braga et al. \ \cite{BrPrSaSco,BrPrSa}, who showed that $\theta$ is analytic for $p$ close enough to $1$. In this paper we fully answer this question in the affirmative (\Tr{analytic}). We also answer the corresponding questions, asked by Michelen et al. \cite{MiPeRoQue}, for Galton-Watson trees (\Tr{thm GW}), and by G\"unter et al. \cite{PenroseBooleanModel} for the Boolean model in $\R^2$ (\Tr{Boolean}).

\medskip
Part of the interest for this question comes form Griffiths' \cite{Griffiths} discovery of models, constructed by applying the Ising model on 2-dimensional percolation clusters, in which the free energy is infinitely differentiable but not analytic. This phenomenon is since called a \defi{Griffiths singularity}, see \cite{EntGri} for an overview and further references. We remark that the importance of Griffiths' discovery is that such a behaviour is possible for functions arising in statistical mechanics. For arbitrary functions this is not a surprise, as `most' $C_\infty$ functions on $[0,1]$ are nowhere analytic \cite{CatDif}.

The study of the analytical properties of the free energy is a common theme in several models of Statistical Mechanics. Perhaps the most famous such example is Onsager's exact calculation of the free energy of the square-lattice Ising model \cite{Onsager}. A corollary of this calculation is the computation of the critical temperature, as well as the analyticity of the free energy for all temperatures other than the critical one. See also \cite{KLM} for an alternative proof of the latter result. The analytical properties of the free energy have also been studied for the $q$-Potts model, which generalizes the Ising model. For this model, the analyticity of the free energy has been proved for $d=2$ and all supercritical temperatures when $q$ is large enough \cite{Complete}. 

\subsection{Summary of results}

We now summarise the main results of this paper. More details are provided in the following section. Let
\begin{align} \label{def pC}
\begin{split}
p_\C:= &\inf_{} \{p\leq 1 \mid \theta_o(p) \text{ is analytic in } (p,1].\}
\end{split}
\end{align}
\begin{enumerate}
\item \label{ichi} For every quasi-transitive graph, and every quasi-transitive (1-parameter) \lrm, the susceptibility $\chi(p)$ is analytic in the subcritical interval $[0,p_c)$.
\item \label{ithpl} For every $d\geq 2$, the percolation density $\theta(p)$ on $\mathbb{Z}^d$ is analytic in the supercritical interval $(p_c,1]$ (in other words, $p_\C=p_c$). So is the $n$-point function $\tau$ and its truncation $\tau^f$. The corresponding results are proved for quasi-transitive lattices in $\R^2$, and continuum percolation in $\R^2$ as well.
\item \label{ithtr} For almost every Galton-Watson tree defined by any progeny distribution, we have $p_\C=p_c$. On the other hand, we display a tree $T$ for which $\theta$ is nowhere analytic on $(p_c,1]$.
\item \label{ithfp} For every finitely presented, 1-ended \Cg, we have $p_\C<1$. 
Moreover, for some \Cg\ of every finitely presented, 1-ended group, we have $p_\C\leq 1-p_c$ for both site and bond percolation. 
\item \label{ithna} For every non-amenable graph with bounded degrees, we have $p_\C<1$. 
It is possible for $\theta$ to be analytic at the uniqueness threshold $p_u$.  
\item  \label{itriang} For certain families of triangulations for which Benjamini \& Schramm \cite{BeSchrPer} and Benjamini \cite{Benjamini} conjectured that $p_c^{site} \leq 1/2$, we prove $p_c^{bond}\leq p_\C <1/2$.
\end{enumerate}

\subsection{Proof ideas}
Kesten's method for the analyticity of $\chi$ (or $\kappa$) below $p_c$ \cite{Ke81} (see also \cite[\S 6.4]{Grimmett}) involves extending $p$ and $\chi$ to the complex plane, and applying the standard complex analytic machinery of Weierstrass to the series $\chi(p):=\sum_{n\in \N} n P_n(p)$. This uses the fact that $P_n(p)$ can be expressed as a polynomial by considering all possible clusters of size $n$, and can hence be extended to $\C$. To show that this series converges to an analytic function $\chi(z)$ on an open neighbourhood of $[p_c,1)$, one needs upper bounds for $|P_n(z)|$ inside appropriate domains in order to apply the Weierstrass M-test (see Appendix~\ref{sec App CA}). These bounds are obtained combining the well-known fact due to Aizenman \& Barsky \cite{AizBar} that $P_n(z)$ decays exponentially in $n$ for real $z$, with elementary complex-analytic calculations. Kesten's calculations involved the numbers of certain `lattice animals', but we observe (\Tr{chi NN}) that this is not necessary and his proof can be simplified. 
An immediate benefit of this simplification is that the proof extends beyond  $\Z^d$, to bond and site percolation on any quasi-transitive graph. The only ingredients needed are the appropriate exponential decay statement and elementary complex analysis. Moreover, with a bit more work the proof can be extended to \lrm s: the functions $P_n(p)$ are no longer polynomials, but we show (\Tr{Pm entire}) how they can be extended into entire functions, i.e.\ complex-analytic functions defined for all $p\in \C$. This summarises the proof of \ref{ichi}, which is given in detail in \Sr{secChi}. 
One application of \ref{ichi} of particular interest to us is to a \lrm\ studied in \cite{IGperc}, which was the original motivation of our work.

The technique we just sketched is used in our results \ref{ithpl}--\ref{ithna} as well, but additional ingredients are needed. For \ref{ithpl}, we write $\theta(p)= 1- \sum_n P_n(p)$ by the definitions, but as $P_n$ decays slower than exponentially for $p>p_c$ \cite{KunSou,Grimmett}, the above machinery cannot be applied to this series. Therefore, instead of working with the size of $C(o)$, we work with the `perimeter' of its boundary. To make this more precise, consider  first the $2$-dimensional case and define the \defi{\scv} of the cluster $C(o)$ to be the pair $(\per,\bou)$, where \per\ denotes the set of edges of $C(o)$ bounding its unbounded face, and \bou\ denotes the set of vacant edges  incident with \per\ lying in the unbounded face of $C(o)$ (\fig{figscv}). We say that such a pair of edge sets $I=(\per,\bou)$ \defi{occurs} in some percolation instance, if it is the \scv\ of some cluster, in which case all edges in \per\ are occupied and all edges in \bou\ are vacant. For any plausible such $I$, the probability $P_I(p):= \Pr_p(\text{$I$ occurs})$ is just $p^{|\per|} (1-p)^{|\bou|}$ by the definitions, which is a polynomial we can extend to $\C$ hoping to apply our complex-analytic machinery. Moreover, these $P_I$ exhibit the kind of exponential decay we need: \bou\ gives rise to a connected subgraph of the dual lattice, and we can combine a well-known coupling between supercritical bond percolation on a lattice and subcritical bond percolation on its dual (see \Tr{pc star}) with the aforementioned exponential decay of $P_n$ applied to the dual.

\begin{figure}[htbp]
   \centering
   \noindent
\begin{overpic}[width=.6\linewidth]{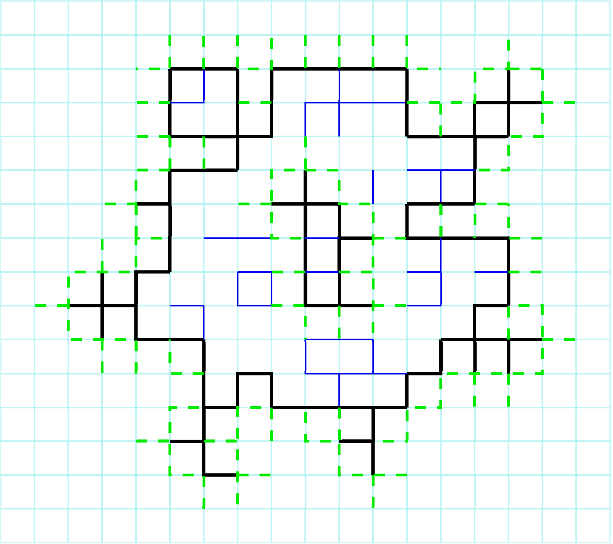}
\put(51,40){$o$}
\end{overpic}
\caption{\small An example of two \scv s of percolation clusters, one nested inside the other, with the inner one being the \scv\ of  $C(o)$. We depict $\per$ with bold lines,  and $\bou$ with dashed lines (but we also depict the other \scv\ similarly). Other occupied edges are depicted in plain lines (blue, if colour is shown).} \label{figscv}
\end{figure}


Still, further challenges arise when trying to express $\theta$ in terms of the functions $P_I$, because knowing that a certain \scv\ $I$ occurs does not imply that it is part of the cluster $C(o)$: there could be other \scv s nested inside $I$, as exemplified in \fig{figscv}. We overcome this difficulty using the \IEP, to express $\theta$ as 
\labtequ{iepintro}{$\theta(p)= 1- \sum_{I \in \MS} (-1)^{c(I)+1} P_I$,}
where $\MS$ is the set of finite disjoint unions of \scv s, and $c(I)$ counts the number of \scv s in $I$. The problem now becomes whether the probability for such an $I \in \MS$ with $n$ edges in total decays exponentially in $n$. All we know so far is that the probability to have an \scv\ containing a fixed vertex $x$ decays exponentially, which seems to be of little use given that there are many ways to partition $n$ into smaller integers $n_1, \ldots n_k$, and construct an  $I \in \MS$ out of $k$ \scv s of lengths $n_i$, each rooted at one of many candidate vertices $x_i$. But there is a way to bring all these possibilities under control, and establish the desired exponential decay, by a certain combination of the following ingredients: \\
a) 
the fact that the number of partitions of an integer $n$ grows subexponentially in $n$ (\Sr{sec parts});\\
b) some combinatorial arguments that restrict the possible vertices $x_i$ at which the \scv s meet the horizontal axis, and\\
c) using the BK inequality (\Tr{BKthm}) to argue that for each choice of a partition of $n$, and vertices $x_1, \ldots x_k$, the probability of occurrence of an $I \in \MS$ complying with this data decays as fast as if we had a single \scv\ of size $n$ (which we already know to decay exponentially). 

This summarises the proof of \ref{ithpl} in 2-dimensions, which is given in detail in \Sr{sec th pl}. Our method applies to \emph{site} percolation  (\Cr{triangular}) as well, with the only difference that we now need to work with the matching graph rather than the dual. 

\medskip
We remark that formula \eqref{iepintro} can be thought of as a refinement of the well-known Peierls argument (see e.g.\ \cite[p.~16]{Grimmett}), where instead of an inequality we now have an equality. The price to pay is that the structures arising ---of the form (\per,\bou) instead of just \bou--- are harder to enumerate, and the benefit is that the events we consider are mutually exclusive, hence the equality. We found this technique very useful in this paper and expect it to be useful elsewhere.

\medskip
The only use of planarity in the proof of \ref{ithpl} we just sketched was the duality argument needed for the exponential tail of the size of an \scv. It is easy to imagine generalising \scv s to higher dimensions, although coming up with a precise definition that uniqely associates an interface with any percolation cluster requires some thought. In \Sr{sec fp} we offer such a definition that applies to all graphs, not just lattices in $\R^d$. We show that once we fix an 1-ended graph \G, a basis of its cycle space (for $G=\Z^d$ the family of all squares is a natural choice) and a percolation instance 
$\omega$, every finite cluster $C$ of $\omega$ uniquely defines an `\scv' $I=(\per,\bou)$ with \per\ a connected subgraph of $C$, and \bou\ containing 
the minimal cut separating $C$ from infinity. This refines the argument of Timar \cite{TimCut} used to simplify the proof of the theorem of Babson \& Benjamini that $p_c<1$ for every $1$-ended finitely presented \Cg. When \g is such a \Cg, we show that our \scv s exhibit an exponential tail by repeating the arguments (a)-(c) from above, and reach \ref{ithfp}. This is one of the hardest results of this paper (\Tr{thm pc fp}), mainly due to the `deterministic' \Tr{uniq pint}. It also applies to site percolation (\Cr{thm pc fp s}). Moreover, we show that if we `triangulate' our \Cg\ by adding more generators, then we can achieve $p_\C\leq 1-p_c$ 
for both site and bond percolation (\Tr{thm pc triang}).

With the aforementioned generalization of the notion of interfaces, our method still yields the analyticity of $\theta$ on $\mathbb{Z}^d$, $d\geq 3$ for the values of $p$ close to $1$, but not in the whole supercritical interval. The main obstacle is that for values of $p$ in the interval $(p_c,1-p_c)$, the distribution of the size of the interface of $C_o$ has only a stretched exponential tail, which follows from the work of Kesten and Zhang \cite{KeZhaPro}. 

In the same paper, Kesten and Zhang introduced some variants of the standard boundary of $C_o$ that are obtained by dividing the lattice $\Z^d$ into large boxes, and proved that these variants satisfy the desired exponential tail on the whole supercritical interval.\footnote{The threshold $p_c(H^d)$ in Kesten's and Zhang's original formulation was proved later to coincide with $p_c(\Z^d)$ by Grimmett and Marstrand \cite{GriMar}.} It is natural to try to apply our method to those variants, however, it turns out that their occurrence does not prevent the origin from being connected to infinity. Instead, we expand these variants into larger objects that we call \textit{separating components}. In Section \ref{theta} (\Lr{Co finite}) we prove that whenever a separating component $S$ occurs, we can find inside $S$ and its boundary $\partial_\sboxt S$ an edge cut  $\partial^b \mathcal{S}_o$ separating the origin from infinity. Conversely, some separating component occurs whenever $C_o$ is finite (\Lr{an S occurs}). Thus we can express $\theta$ in terms of the occurrence of separating components (see \eqref{cases} in \Sr{theta sum S}). 
In contrast to the behaviour of the boundary of $C_o$ which has only a stretched exponential tail on the interval $(p_c,1-p_c]$, this $\partial^b \mathcal{S}_o$ has an exponential tail in the whole supercritical interval. We plug this exponential decay into our general tool (Corollary~\ref{cor general}) to obtain the analyticity of $\theta$ above $p_c$ in \Sr{expand}. In Section \ref{section tau} we use similar arguments to prove the analyticity of the $k$-point function $\tau$ and its truncation $\tau^f$, as well as of $\chi^f$ and $\kappa$.

Typically, $\partial^b \mathcal{S}_o$ has size of smaller magnitude than the boundary of $C_o$, and it is obtained from the latter by `smoothening' some of its parts with `fractal' structure. As a corollary, we re-obtain, in \Sr{Pete}, a result of Pete \cite{Pete} about the exponential decay of the probability that $C_o$ is finite but sends a lot of closed edges to the infinite component.

\medskip
A well-known theorem of Benjamini \& Schramm \cite{BeSchrPer} states that $p_c(G)\leq \frac1{1+h_E(G)}$ where $h_E(G)$ is the Cheeger constant of an arbitrary graph \G, and so $p_c<1$ for every non-amenable graph. We show in \Sr{sec na}, where we recall the relevant definitions, that the same bound applies to $p_\C$ (\Tr{thm nonam}). This has the interesting consequence that $\theta$ does not witness the phase transition at the uniqueness threshold $p_u$: using the results of \cite{PakSmir,ThomRemark}, we deduce that $\theta$ is analytic at $p_u$ for some Cayley graph of every non-amenable group. Another consequence of $p_c(G)\leq \frac1{1+h_E(G)}$ is that $p_\C =p_c$ for the infinite $d$-regular tree (\Cr{cor tree}). Shortly after the first draft \cite{analyticity} of the current paper was released, Hermon \& Hutchcroft \cite{HerHutSup} proved that $\theta$ is analytic in the whole supercritical interval for every non-amenable transitive graph, by establishing that the cluster size distribution $P_n$ has an exponential tail in the whole supercritical interval. 

The analyticity of the annealed analogue of $\theta(p)$ for Galton--Watson trees for $p>p_c$ was recently established by Michelen, Pemantle \& Rosenberg, who asked whether analyticity holds in the quenched regime as well \cite[Question 3]{MiPeRoQue}. In Section \ref{sec trees} (\Tr{thm GW}) we answer this question in the affirmative \eqref{ithtr}. On the other hand, we display a tree $T$ for which $\theta$ is nowhere analytic on $(p_c,1]$ (\Sr{nowhere}).  Following a discussion at an \href{https://percolation.ethz.ch/slack/}{on-line platform}, Tom Hutchcroft constructed a unimodular random graph for which $\theta$ is nowhere analytic on $(p_c,1]$.

\medskip
Most of this paper is concerned with analyticity results, but some of the methods developed can be applied to provide bounds on $p_c$ as well. We display this in \Sr{sec triang}, where we prove that $p_c^{bond} <1/2$ for certain families of triangulations for which Benjamini \& Schramm \cite{BeSchrPer}, Benjamini \cite{Benjamini}, and Angel, Benjamini \& Horesh \cite{AnBeHor} conjectured that $p_c^{site} \leq 1/2$ (\ref{itriang}). Using similar ideas, the second author and J.~Haslegrave \cite{HasPanSit} prove that $p_c^{site} < 1/2$ for all triangulations of a disc with minimum degree at least $7$, answering another question of \cite{BeSchrPer}.

\medskip
After proving some functions to be analytic, additional fun, and hopefully results, can be had by studying their complex extensions. As already mentioned, we proved that the functions $P_n$ admit entire extensions (trivially for \nnm s), and are therefore uniquely determined by their Maclaurin coefficients. As most observables of percolation theory, e.g.\ $\chi$ and $\theta$, are uniquely determined by the sequence $\{P_n\}_{n\in \N}$, it makes sense to study those coefficients. We do so in \Sr{sec alt}, where we show that their signs alternate with $n$, and do not depend on the model (\Tr{alterf}). 

We use this fact in \Sr{sec neg}, where we show how one can make sense of a negative percolation threshold $p_c^-\in \R_{<0}$. As it happened in the history of $p_c$, more than one candidate definitions are possible. We could show that some of them coincide (\Tr{coincide}), but there are still more questions than results on this topic.

\medskip
Most of the essence of our proofs lies in combinatorial arguments. We have made an effort to make this paper accessible to the non-expert, except for this introduction that uses terminology that is defined later. The complex analysis we use is at undergraduate level, involving only some classics we recall in Appendix~\ref{sec App CA} and elementary manipulations. Hardly any background in probability theory is assumed, but some familiarity with the basics of percolation theory as in \cite{Grimmett} will be helpful.

\section{The setup} \label{sec setup}

We recall some standard definitions of percolation theory in order to fix our notation. For more details the reader can consult e.g.\ \cite{Grimmett,LyonsBook}. For a higher level overview of percolation theory we recommend the recent survey \cite{HugoSurvey}.

\subsection{\Nnm s} \label{sec nnm}
Let $G = (V,E)$ be a locally finite countably infinite graph, and let $\Omega:= \{0,1\}^E$ be the set of \defi{percolation instances} on \G. We say that an edge $e$ is \defi{vacant} (respectively, \defi{occupied}) in a  percolation instance $\oo\in \Omega$, if $\oo(e)=0$ (resp.\ $\oo(e)=1$).

By (Bernoulli, bond) \defi{percolation} on $G$ with parameter $p\in [0,1]$ we mean the random subgraph of $G$ obtained by keeping each edge with probability $p$ and deleting it with probability $1-p$, with these decisions being independent of each other.

More formally, we endow $\Omega$  with the $\sigma$-algebra $\mathcal{F}$ generated by the cylinder sets $C_e:= \{\oo \in \Omega, \oo(e)=\eps\}_{e\in E, \eps\in \{0,1\}}$, and the probability measure defined as the product measure $\Pr_p:=\Pi_{e\in E} \mu_e$, where $p\in [0,1]$ is our \defi{percolation parameter} and $\mu_e$ is the Bernoulli measure on $\{0,1\}$ determined by $\mu_e(1)=p$.

The \defi{percolation threshold} $p_c(G)$ is defined by 
$$p_c(G):= \sup\{p \mid \Pr_p(|C(o)|=\infty)=0\},$$
where the \defi{cluster} $C(o)$ of $o\in V$ is the component of $o$ in the subgraph of \g spanned by the occupied edges. It is easy to show that $p_c(G)$ does not depend on the choice of $o$.

To define \defi{site percolation} we repeat the same definitions, except that we now let $\Omega:= \{0,1\}^V$, and let  $C(o)$ be the component of $o$ in the subgraph of \g induced by the occupied vertices. 

In this paper the graph $G$ is \textit{a priori} arbitrary. Some of our results will need assumptions on \g like vertex-transitivity or planarity, but these will be explicitly stated as needed.

\subsection{\Lrm s} \label{sec lrm}

Long range percolation is a generalisation of Bernoulli bond percolation where different edges become occupied with different probabilities, and each vertex can have infinitely many incident edges that can become occupied. In fact, the graph is often taken to be the complete graph on countably many vertices, and so its edges play a rather trivial role. Therefore, it is simpler to define our model with a set rather than a graph as follows.

Let $V$ be a countably infinite set (the \defi{vertices}), and let $E=V^2$ be the set of pairs of its elements (the \defi{edges}). We will typically write $xy$ istead of $\{x,y\}$ to denote an element of $E$. Let $\mu: E \to \R_{\geq 0}$ be a function satisfying $\sum_{y\in V} \mu(xy) =1$ \fe\ $x\in V$ (in some occasions we allow more general $\mu$, satisfying just $\sum_{y\in V} \mu(xy) <\infty$). The data $V,\mu$ define a random graph on $V$ similarly to the previous definition, except that we now make each edge $xy$ vacant with probability $e^{-\mu(xy)t}$, with our parameter $t$ now ranging in $[0,\infty)$. The corresponding probability measure on $\Omega= \{0,1\}^E$ is denoted by $\Pr_t$ (We like thinking of $t$ as time, with the each edge $xy$ becoming occupied if vacant at a tick of a Poisson clock with rate  $\mu(xy)$.)

Analogously to $p_c$, one defines 
$$t_c=t_c(V,\mu):= \sup\{t \mid \Pr_t(|C(o)|=\infty)=0\},$$
which again does not depend on the choice of $o\in V$. 

We say that such a percolation model, defined by $V$ and $\mu$, is \defi{transitive}, if there is a group acting transitively on $V$ that preserves $\mu$. In other words, if \fe\ $x,y\in V$ there is a bijection $\pi: V\to V$ \st\ $\mu(\pi(z)\pi(w)) = \mu(zw)$.

\medskip
Long range percolation is a less standard topic that is not typically found in textbooks, and the term often refers to the special case where the group  acting transitively is $\Z$, for example in order to come up with a model in which $\theta$ is discontinous at $t_c$ \cite{AizNewDis}. In the generality we work with it has been considered in e.g.\ \cite{AizNewTre,IGperc}.

\section{Definitions and preliminaries}

\subsection{Graph theoretic definitions}\label{sec GT}

Let $G = (V,E)$ be a graph.
An \defi{induced} subgraph $H$ of $G$ is a subgraph that contains all edges $xy$ of $G$ with $x,y\in V(H)$. Note that $H$ is uniquely determined by its vertex set. The subgraph of $G$ \defi{spanned} by a vertex set $S\subseteq V(G)$ is the induced subgraph of $G$ with vertex set $S$.

The vertex set of a graph \g will be denoted by $V(G)$, and its edge set by $E(G)$.
A graph $G$ is \defi{(vertex-)transitive}, if \fe\ $x,y\in V(G)$ \ti\ an automorphism $\pi$ of $G$ mapping $x$ to $y$, where an \defi{automorphism} is a bijection $\pi$ of $V(G)$ that preserves edges and non-edges.

A \defi{planar graph} $G$ is a graph that can be embedded in the plane $\R^2$, i.e.\ it can be drawn in such a way that no edges cross each other. Such an embedding is called a  \defi{planar embedding} of the graph. A \defi{plane graph} is a (planar) graph endowed with a fixed planar embedding.

A plane graph divides the plane into regions called \defi{faces}. Using the faces of a plane graph $G$ we define its \defi{dual graph} $G^*$ as follows. The vertices of $G^*$ are the faces of $G$, and we connect two vertices of $G^*$ with an edge whenever the corresponding faces of $G$ share an edge. Thus there is a bijection $e\mapsto e^*$ from $E(G)$ to $E(G^*)$.

Given a finite subgraph $H$ of $G$, we define its \defi{internal boundary} $\partial H$ to be the set of vertices of $H$ that are incident with an infinite component of $G\setminus H$.
We define the \defi{vertex boundary} $\partial^V H$ of $H$ as the set of vertices in $V\setminus V(H)$ that have a neighbour in $H$. The \defi{edge boundary} $\partial^E H$ is the set of edges in $E\setminus E(H)$ that are incident to $H$. 

Consider now a vertex $x$ of $G$. We say that a set $S$ of edges of $G$ is an \defi{edge cut} of $x$ if $x$ belongs to a finite component of $G-S$. We say that $S$ is a \defi{minimal edge cut} of $x$ if it is minimal with respect to inclusion. For a finite connected subgraph $H$ of $G$, its minimal edge cut is the set of edges with one endvertex in $H$ and one in an infinite component of $G\setminus H$.

The \defi{diameter} $diam(H)$ of $H$ is defined as $\max_{x,y\in V(H)}\{d_G(x,y)\}$ where $d_G(x,y)$ denotes the graph-theoretic distance between $x$ and $y$.

\subsection{Exponential tail of the subcritical cluster size distribution: the \ABP} \label{sec ABP}

An important fact that will be used throughout the paper whenever we want to show the convergence of a series is the following  exponential decay of the cluster size distribution $p_n:= \Pr(|C(o)|= n)$ (or equivalently, of $f_n:= \Pr(|C(o)|\geq n)$) in the subcritical regime, which we will call the \defi{\ABP}\footnote{Some bibliographical remarks about \Tr{AB thm}: Kesten \cite{Ke81} proved exponential decay when $\chi<\infty$ for lattices in $\R^d$, and  Aizenman \& Newman \cite{AizNewTre} extended it to all models we are interested in (their precise formula is $\Pr_p(|C(o)|\geq m)\leq (e/m)^{1/2} e^{-m/(2\chi(p))^2}$). Aizenman \& Barsky \cite{AizBar} proved $\chi<\infty$ below $p_c$ on $\Z^d$, and \cite[Theorem 7.46.]{LyonsBook} claims that `their proof works in greater generality', that is, for all transitive graphs. Menshikov \cite{MenCoi} independently obtained the same result in a more restricted class of models. Antunovi{\'c} \& Veseli{\'c} \cite{AntVes} extended this to all quasi-transitive models. 
Duminil-Copin \& Tassion \cite{DCTa} gave a shorter proof that $\chi<\infty$ below $p_c$ (or $\beta_c$) for all independent, transitive bond and site models.}. For every vertex $o$, we define $\chi(p)=\chi_o(p):=\mathbb{E}_p(|C_o|)$, and we let $X(p):=\sup_{o\in V} \chi_o(p)$. 

\begin{theorem}[{\cite[Proposition~5.1]{AizNewTre}, \cite{AizBar,AntVes}}] \label{AB thm}
For every bond, site, or \lrm, (and any vertex $o$), if $X(p)<\infty$, then
$$\Pr_p(|C(o)|\geq n)=O(e^{-n/5X(p)^2}).$$
Moreover, for every quasi-transitive bond, site, or \lrm, if $p<p_c$, then $X(p)<\infty$.
\end{theorem}

We will say that a bond, site, or \lrm\  has the \textit{\ABP} if  for every $p<p_c$ and $o\in V$, there is a constant $c=c(p,o)>0$ such that $\Pr_p(|C(o)|\geq n)\leq e^{-cn}$. 

\subsection{The BK inequality} \label{secBK}

We define a partial order on our space $\Omega= \{0,1\}^{E(G)}$ of percolation instances as follows. For two configurations 
$\omega$ and $\omega'$ we write $\omega\leq \omega'$ if $\omega(e)\leq \omega'(e)$
for every $e\in E$. 

A random variable $X$ is called \defi{increasing} if whenever $\omega\leq \omega'$, then $X(\omega)\leq X(\omega')$. An event $A$ is called \defi{increasing} if its indicator function is increasing. For instance, the event $\{|C(o)|\geq m\}$ is increasing, where $C(o)$ as usual denotes the cluster of $o$. 

\comment{
One can show using the coupling of all percolation processes that $\mathbb{E}_{p}(X)$ and $\Pr_{p}(A)$ are increasing functions of $p$.

\begin{theorem}\cite{Grimmett}
If $X$ is an increasing random variable and $p_1 \leq p_2$, then
$$\mathbb{E}_{p_1}(X) \leq \mathbb{E}_{p_2}(X).$$
Similarly if $A$ is an increasing event, then
$$\Pr_{p_1}(A)\leq \Pr_{p_2}(A).$$
\end{theorem}

It is well-known that increasing functions are differentiable almost everywhere and so
$\Pr_{p}(A)$ is differentiable almost everywhere on $[0,1]$. When $A$ depends on a finite set of edges, $\Pr_{p}(A)$ is a polynomial in $p$ and hence it is differentiable everywhere on $[0,1]$. A natural question is what is the rate at which $\Pr_{p}(A)$ changes when we increase $p$. Russo's formula answers this question. \\

Before stating Russo's theorem we define what a pivotal edge is. Let A be an event and $\omega$ a percolation configuration. An edge $e$ is called \defi{pivotal} for the pair $(A,\omega)$ if $\mathbb{1}(\omega\in A) \neq \mathbb{1}(\omega' \in A)$, where $\omega'$ is the configuration that agrees with $\omega$ on every edge except $e$ where $\omega'(e)=1-\omega(e)$. In other words, $e$ is pivotal for $A$ if the occurrence of $A$ depends crucially on whether the edge $e$ is present or not.

\begin{theorem}\text{(Russo's formula)}\cite{Grimmett}
Let $A$ be an increasing event that depends only on the states of a finite number of edges. Then $$\dfrac{d\Pr_p(A)}{dp} = \mathbb{E}_p(N(A)),$$ where $N(A)$ is the number of pivotal edges for $A$. 
\end{theorem} 

}

For every $\omega\in \Omega$ and a subset $S\subset E$ we write
$$[\omega]_S=\{\omega' \in \Omega: \omega'(e)=\omega(e) \text{ for every } e\in S\}.$$

Let A and B be two events depending on a finite set of edges $F$. Then the disjoint occurrence of $A$ and $B$ is defined as
$$A\circ B =\{\omega \in \Omega: \text{ there is } S\subset F \text{ with } [\omega]_S 
\subset A \text{ and } [\omega]_{F\setminus S} \subset B\}.$$

\begin{theorem}\text{(BK inequality)}\cite{BK,Grimmett} \label{BKthm}
Let $F$ be a finite set and $\omega = \{0, 1\}^F.$ For all increasing events $A$
and $B$ on $\Omega$ we have $$\Pr_p(A\circ B)\leq \Pr_p(A) \Pr_p(B).$$
\end{theorem}

\subsection{Partitions of integers} \label{sec parts}
A \defi{partition} of a positive integer $n$ is an unordered multiset $\{m_1,m_2,\ldots,m_k\}$ of positive integers such that $m_1+m_2+\ldots+m_k=n$. Let $p(n)$ denote the number of partitions of $n$. An asymptotic expression for $p(n)$ was given by Hardy \& Ramanujan in their famous paper \cite{HarRam}. An elementary proof of this formula up to a multiplicative constant was given by Erd\H os \cite{ErdHR}. As customary we use $A\sim B$ to denote the relation $A/B\rightarrow 1$ as $n\rightarrow \infty$.

\begin{theorem}[Hardy-Ramanujan formula] \label{HRthm}
The number $p(n)$ of partitions of $n$ satisfies
$$p(n)\sim \dfrac{1}{4n\sqrt{3}}\exp\Big(\pi \sqrt{\dfrac{2n}{3}}\Big).$$
\end{theorem}

The above asymptotic formula for $p(n)$ implies in particular that $p(n)$ grows subexponentially, and this is all we will need in our several applications of \Tr{HRthm}. This weaker statement can be proved much more easily, and we offer the following elementary proof that makes our paper more self-contained. 

\begin{lemma}
Let $p(n)$ denote the number of partitions of $n$. Then
$$\limsup_{n\to\infty} p(n)^{1/n}=1.$$
\end{lemma}
\begin{proof}
Let us denote $f(z)$ the generating function of $p(n)$, i.e. $$f(z)=\sum_{n=0}^\infty p(n)z^n.$$ It is well-known that $f(z)=\prod_{n=1}^\infty \dfrac{1}{1-z^n}$ (this follows easily by considering the bijection between the set of partitions of $n$ and the set of sequences $(i_1,i_2,\dots,i_n)$ where the $i_j$'s are non-negative integers such that $i_1+2i_2+\dots+ni_n=n$).

The radius of convergence $R$ of $f$ is given by the formula $$R=\dfrac{1}{\limsup_{n\to\infty} p(n)^{1/n}}.$$ It suffices to prove that $R=1$. Since $f(1^-)=+\infty$ we have that $R\leq 1$. In order to show that $R\geq 1$ we will prove that $f$ is analytic on the open unit disk.  

Assume that $z\in [0,1)$. Taking the logarithm of the infinite product we obtain the infinite sum $\sum_{n=1}^\infty -\log(1-z^n)$. Using the fact $$\lim_{x\rightarrow 0}\dfrac{-\log(1-x)}{x}=1$$ and the convergence of the sum $\sum_{k=1}^\infty z^k$ we deduce that $\sum_{n=1}^\infty -\log(1-z^n)$ converges. It follows that $\prod_{n=1}^\infty\dfrac{1}{1-n^k}$, and hence $\sum_{n=0}^\infty p(n)z^n$, converges for every $z\in [0,1)$. 

Assume now that $z$ lies in the open unit disk. Since $p(n)\geq 0$ for every $n$, we have $$\Big|\sum_{n=m}^k p(n)z^n\Big|\leq \sum_{n=m}^k p(n)|z|^n$$ for every $m\leq k$, hence the convergence on the open unit disk can be deduced from the convergence on $[0,1)$. 
\end{proof}

\section{The basic technique}\label{sec tools}

A common ingredient of our analyticity results is the following technique, the main idea of which is present in \cite{Ke81} and was mentioned in the introduction. We express our function $f(p)$ as an infinite series $f(p)= \sum_{\nin} a_n f_n(p)$, where $f_n(p)$ is the probability of an event. For example, when $f=\chi$ is the expected size of the cluster $C(o)$ of $o$, then $f_n$ is the probability that $|C(o)|=n$, and $a_n=n$. To prove that $f(p)$ is analytic, our strategy is to extend the domain of definition of each $f_n$ to complex values of $p$ (we will usually write $z$ instead of $p$ when doing so). Our extended $f_n$ will turn out to be complex-analytic, and so $f$ is analytic if the series $\sum_{\nin} a_n f_n(p)$ converges uniformly by standard complex analysis (Weierstrass' \Tr{thmWei}). To show the latter, we employ the Weierstrass M-test (\Tr{MT}), using upper bounds on $|f_n(z)|$ inside appropriate discs (centered in the interval $[0,1]$ where $p$ takes its values). These upper bounds are obtained by \Lr{C equals S NN} below for nearest-neighbour models, and by its counterpart \Lr{C equals S LR} for long-range models. 


\subsection{Nearest-neighbour models} \label{NNM}

The following lemma, and its generalisation \Cr{D and not F NN} below, provides the upper bounds that we are going to plug into the M-test as explained above. 

Let $\Pr_p$ denote the law of Bernoulli percolation with parameter $p$ on an arbitrary graph $G$, as defined in \Sr{sec setup}. Let $D(x,M)$ denote the disc with center $x\in \C$ and radius $M\in \R_{+}$ in $\C$. For a subgraph $S$ of $G$, 
let $\partial S$ be the set of edges of \g that have at least one end-vertex in $S$ but are not contained in $E(S)$.

In this lemma, $x$ is to be thought of as a value of our parameter $p$ near which we want to show the analyticity of some function, and we are free to choose the radius $M$ of the disc we consider as small as we like.

\begin{lemma}\label{C equals S NN}
For every finite subgraph $S$ of $G$ and every $o\in V(G)$, the function $P_S(p) :=\Pr_p(C_o=S)$ admits an entire extension $P_S(z), z\in \C$, \st\ \fe\ $1> M>0$, every $1>x\geq 0$ with $x+M<1$ and every $z\in D(x,M)$, we have
$$|P_S(z)| \leq C^{|\partial S|} P_S(x+M),$$
where $C=C_{M,x}:= \frac{1-x+M}{1-x-M}$.

Moreover, for every $1\geq x>0$, every $x>M>0$ and every $z\in D(x,M)$, we have $|P_S(z)| \leq K^{|E(S)|} P_S(x-M)$, where $K=K_{M,x}:= \frac{x+M}{x-M}$. 
\end{lemma}
(The second sentence will be used to prove analyticity at $p=1$; the reader who is only interested in analyticity for $p\in [0,1)$ may ignore it and skip the last paragraph of the proof.)
\begin{proof} 
By the definitions, we have 
\labequ{Sproducts NN}{P_S(p)=(1-p)^{|\partial S|}p^{|E(S)|}}
because the event $\{C_o=S\}$ is satisfied exactly when all edges in $\partial S$ are absent and all edges in $E(S)$ present. This function, being a polynomial, admits an entire extension, which we will still denote by $P_S=P_S(z)$ with a slight abuse. 

To prove the upper bound in our first statement ---for $1>x\geq 0$, and $z\in D(x,M)$--- we will bound each of the two products appearing in \eqref{Sproducts NN} separately. Easily, 
$$|z|^{|E(S)|}\leq (x+M)^{|E(S)|}$$
when $z\in D(x,M)$ because $|z|\leq x+|z-x|\leq x+M$. 

Moreover, it is geometrically obvious that the distance $|1-z|$ between $1$ and $z$ is maximised at $z= x-M$, which implies 
$$|1-z|^{|\partial S|}\leq (1-x+M)^{|\partial S|}.$$

Plugging these two inequalities into \eqref{Sproducts NN} we obtain the desired inequality:
\begin{align*}
|P_S(z)|\leq (1-x+M)^{|\partial S|}(x+M)^{|E(S)|} & =\\ 
\Big(\frac{1-x+M}{1-x-M}\Big)^{|\partial S|}(1-x-M)^{|\partial S|}(x+M)^{|E(S)|}
& =\Big(\frac{1-x+M}{1-x-M}\Big)^{|\partial S|} P_S(x+M),
\end{align*} 
where we also applied \eqref{Sproducts NN} with $p=x+M$.

\medskip
For the second statement, let $x\in (0,1]$, $0<M<x$ $z\in D(x,M)$. Then $|z|\leq x+M$, and $|1-z|\leq 1-x+M$, and similarly to the above calculation we have
\begin{align*}
|P_S(z)|\leq (1-x+M)^{|\partial S|}(x+M)^{|E(S)|} &=\\ 
(1-x+M)^{|\partial S|}\Big(\frac{x+M}{x-M}\Big)^{|E(S)|}(x-M)^{|E(S)|}
&=\Big(\frac{x+M}{x-M}\Big)^{|E(S)|} P_S(x-M).
\end{align*}

\end{proof}
\begin{remark} \label{rem d} 
When $G$ has maximum degree $d$, we have the crude bound $|\partial S|\leq d|S|$, with which Lemma \eqref{C equals S NN} yields $|P_S(z)| \leq C_{M,x}^{d|S|} P_S(x+M).$
\end{remark}

Note that in the proof of \Lr{C equals S NN} we can replace $E(S)$ and $\partial S$ with any two disjoint finite sets of edges $D,F \subset E(G)$, to obtain the following:

\begin{corollary}\label{D and not F NN}
For every two disjoint finite sets of edges $D,F \subset E(G)$, the function $P(p) :=\Pr_p(D \subseteq \omega \text{ and } F \cap \omega = \emptyset)$ (i.e.\ the probability that all edges in $D$ are occupied and all edges in $F$ are vacant) admits an entire extension $P(z), z\in \C$, \st\ 
\begin{align}\label{ubound}
|P(z)| \leq \Big(\frac{1-x+M}{1-x-M}\Big)^{|F|}P(x+M)
\end{align} \fe\ $M>0$, $1>x\geq 0$ with $x+M<1$ and $z\in D(x,M)$. Moreover, for every $1\geq x>0$, every $x>M>0$ and every $z\in D(x,M)$, we have 
\begin{align}\label{lbound}
|P(z)| \leq \Big(\frac{x+M}{x-M}\Big)^{|D|} P(x-M).
\end{align} \qed 
\end{corollary}

\subsection{Long-range models} \label{LRM}

We now prove the analogue of \Lr{C equals S NN} for \lrm s. 
Recall that in our long-range setup, we have a vertex set $V$ and any two of its elements can form an edge. The parameters $x,M$ now take their values in $[0,\infty)$, as this is the case for our percolation parameter $t$. Let $\partial S$  be the set of pairs $\{x,y\}\subset V^2$ that are not contained in $E(S)$ but have at least one vertex in $S$.

\begin{lemma}\label{C equals S LR}
For every finite graph $S$ on a subset of $V$, and every $o\in V$, the function $P(t) :=\Pr_t(C(o)=S)$ admits an entire extension $P(z), z\in \C$, \st\ $|P(z)| \leq e^{2M|S|} P(x+M)$ \fe\ $M>0$, $x\geq 0$ and $z\in D(x,M)$. 
\end{lemma}

The proof of this is similar to that of \Lr{C equals S NN}, but as our function $P(t)$ is not exactly a polynomial now we will need some reshuffling of terms and the following basic fact about complex numbers.

\begin{proposition}\label{absz}
For every $\mu>0$ and every $z\in \C$ we have
$$|e^{\mu z}-1| \leq e^{\mu |z|}-1.$$ 
\end{proposition}
\begin{proof} 
Expressing $e^{\mu z}$ via its Maclaurin expansion and using  
the triangle inequality yields
\begin{align}\label{triangle}
|e^{\mu z}-1|=\abs*{\sum_{j=1}^\infty \dfrac{(\mu z
)^j}{j!}}\leq \sum_{j=1}^\infty \dfrac{|z
\mu|^j}{j!}.
\end{align}

Since $\mu>0$, the last expression coincides with the Maclaurin expansion of 
$e^{\mu r}-1$ evaluated at $r= |z|$, from which we obtain $|e^{\mu z}-1| \leq e^{\mu |z|}-1$.

\end{proof}

\begin{proof}[Proof of \Lr{C equals S LR}]
Similarly to \eqref{Sproducts NN}, we have 
\labequ{Sproducts LR}{\Pr_t(C(o)=S)=\prod_{e\in \partial S}e^{-t\mu(e)}\prod_{e\in E(S)}\big(1-e^{-t\mu(e)}\big),}
because the event $\{C(o)=S\}$ is satisfied exactly when all edges in $\partial S$ are absent and all edges in $E(S)$ present. Multiplying the second product by $\prod_{e\in E(S)}e^{t\mu(e)}$ and the first by its inverse, we obtain
\labequ{Sproducts II}{\Pr_t(C(o) = S)=\prod_{e \in \partial S \cup E(S)} e^{-t\mu(e)}\prod_{e\in E(S)}\big(e^{t\mu(e)} -1\big) =  e^{-t\mu(S)}\prod_{e\in E(S)}\big(e^{t\mu(e)} -1\big),}
where $\mu(S):= \sum_{e \text{ incident with $S$}} \mu(e)<\infty$
because the edges incident with $S$ are exactly the elements of $\partial S \cup E(S)$. This function clearly admits an entire extension, which we will still denote by $P=P(z)$ with a slight abuse. 

To prove the upper bound, we will bound each of the two products appearing in \eqref{Sproducts II} separately. Easily, 
$$|e^{-z \mu(S)}| \leq e^{2M|S|} e^{-(x+M) \mu(S)}$$
when $z\in D(x,M)$ because $|z|\leq x+|z-x|\leq x+M$ and $\mu(S)\leq |S|$. For the second product, we apply  \Prr{absz} to each factor to obtain
\labequ{series}{
|e^{z\mu(e)}-1| \leq e^{|z| \mu(e)}-1 \leq e^{(x+M)\mu(e)}-1} 
\fe\ for $z\in  D(x,M)$.

Combining these two inequalities, and then applying \eqref{Sproducts II} with $t=x+M$, we obtain the desired bound:
$$|P(z)| \leq e^{2M|S|} e^{-(x+M) \mu(S)} \prod_{e\in E(S)}\big( e^{(x+M)
\mu(e)}-1 \big) = e^{2M|S|} P(x+M).$$
 
\end{proof}

Again, in this proof we can replace $E(S)$ and $\partial S$ with any two disjoint finite sets of edges $D,F \subset E$, to obtain, in analogy with \Cr{D and not F NN}, the following statement:

\begin{corollary}\label{D and not F LR}
For every two disjoint finite sets of edges $D,F \subset E$, the function $P(t) :=\Pr_t(D \subseteq \omega \text{ and } F \cap \omega = \emptyset)$ (i.e.\ the probability that all edges in $D$ are occupied and all edges in $F$ are vacant) admits an entire extension $P(z), z\in \C$, \st\ 
$|P(z)| \leq e^{2M|V(D \cup F) |} P(x+M)$ \fe\ $M>0$, $x\geq 0$ and $z\in D(x,M)$, where $V(D \cup F)$ denotes the set of vertices that are incident with some edge  in  $D \cup F$. \qed
\end{corollary}

Similarly, if we replace $E(S)$ in  \Lr{C equals S LR} with a set of edges incident to a vertex $o$ and $\partial S$ with the remaining edges that are incident to $o$ we obtain the following corollary. We let $N(o)$ denote the neighbourhood of $o$ in the percolation cluster, i.e.\ the set of vertices sharing an occupied edge with $o$. 

\begin{corollary}\label{neighbourhood}
For every $o\in V$ and every  $L\subset V$, the function\\ $P(t) :=\Pr_t(N(o)=L)$ admits an entire extension $P(z), z\in \C$, \st\ $|P(z)| \leq e^{2M} P(x+M)$ \fe\ $M>0$, $x\geq 0$ and $z\in D(x,M)$.
\end{corollary}

\subsubsection{Analyticity of the probability of a given cluster size} \label{secPF}
Next, we prove that $p_m(t) :=\Pr_t(|C(o)| = m)$ is analytic, in the full generality of our \lrm s as above. For nearest-neighbour models this is trivial, because the corresponding probability can be expressed as a polynomial, but the long-range variant is more interesting. In addition to analyticity, the following result also provides the upper bound that we will plug into the Weirstrass M-test to deduce the analyticity of the susceptibility $\chi$ for subcritical \lrm s (\Tr{chi NN}).

\begin{theorem}\label{Pm entire}
For every $m\in \N$ and every $o\in V$, the function\\ $p_m(t) :=\Pr_t(|C(o)| = m)$ admits an entire extension $p_m(z), z\in \C$, \st\ $|p_m(z)| \leq e^{2Mm} p_m(x+M)$ \fe\ $M>0$, $x\geq 0$ and $z\in D(x,M)$.
\end{theorem}
\begin{proof} 
For $m\in \N$, let $\mathcal G_m(V)$ denote the set of finite graphs whose vertex set in a subset of $V$ with $m$ elements containing $o$ (to be thought of as possible percolation clusters of $o$).
For every such $S\in \mathcal G_m(V)$, \Lr{C equals S LR} yields an entire extension $P_S$ of $\Pr_{t}(C(o) = S)$.
We claim that the sum 
\labequ{sumS}{\sum_{S\in \mathcal G_m(V)} P_S(z),}
which for $t\in \R, t>0$ coincides with $\Pr_{t}(|C_o| = m)$, converges uniformly on each closed disc $D(x,M), M>0$, $x\geq 0$ to a function $p_m: \C \to \C$ which coincides with $\Pr_{t}(|C_o| = m)$ for $t\in \R, t>0$. By Weierstrass' \Tr{thmWei}, this means that $p_m$ admits an entire extension.

Indeed, this uniform convergence follows from the Weierstrass M-test: each summand $P_S$ can be bounded by $|P_S(z)| \leq e^{2M|S|} P_S(x+M) =e^{2Mm} P_S(x+M) $ \fe\ $M>0$, $x\geq 0$ and $z\in D(x,M)$ by \Lr{C equals S LR}. Moreover, the sum of these bounds satisfies
$$\sum_{S\in \mathcal G_m(V)}  e^{2Mm} P_S(x+M) = e^{2Mm} p_{m}(x+M) < \infty.$$ 
Thus the Weierstrass M-test can be applied to deduce that \eqref{sumS} converges uniformly on $D(x,M)$, and therefore on any compact subset of $\C$.

Finally, the above bounds also prove that $|p_m(z)| \leq e^{2Mm} p_m(x+M)$  as desired.
\end{proof} 
 
\begin{corollary}\label{fm entire}
For every $m\in \N$ and every $o\in V$, the function $f_m(t) :=\Pr_t(|C(o)| \geq m)$ admits an entire extension.
\end{corollary}
\begin{proof}
It follows from the formula $\Pr_t(|C(o)| \geq m)=1-\sum_{i=1}^{m-1}\Pr_t(|C(o)| = i)$ and \Tr{Pm entire}.
\end{proof}

\subsection{Analyticity of $\chi$ in the subcritical regime} \label{secChi}

In this section we prove that the \defi{susceptibility} $\chi(t):= \mathbb{E}_t(|C(o)|)$ of our models is an analytic function of the parameter in the subcritical interval. This applies to both nearest-neighbour and \lrm s. For this we need to assume that our model has the \ABP.

\begin{theorem}\label{chi LR}  
For every long-range model with the \ABP\ (in particular, \fe\ transitive model), $\chi(t)$ is  real-analytic in the interval $[0,t_c)$.
\end{theorem}

\begin{theorem}\label{chi NN} 
For every bounded-degree nearest-neighbour model with the\\ \ABP\ (in particular, \fe\ vertex-transitive graph), $\chi(p)$ is real-analytic in the interval $[0,p_c)$.
\end{theorem}

The proofs of these facts are very similar, and follow Kesten's proof \cite{Ke81} of the corresponding statement for (nearest-neighbour) lattices in $Z^d$, except that we simplify it by avoiding any mention to lattice animals.

\begin{proof}[Proof of \Tr{chi LR}]
Each summand in the definition\\ $\chi(t)=\sum_{m=1}^\infty m \Pr_t(|C(o)|= m)$ of $\chi$ admits an analytic extension to $\C$ by \Tr{Pm entire}. By Weierstrass' \Tr{thmWei}, it suffices to prove that for every $x\in [0,t_c)$ there is an open disk $D$ centred at $x$ such that $\sum_{m=1}^\infty m \Pr_t(|C(o)|= m)$ converges uniformly in $D$. 

Pick an arbitrary $x\in [0,t_c)$ and $x<y<t_c$. Since we are assuming the \ABP, there is a constant $c=c(y,o)>0$ such that $\Pr_y(|C(o)|\geq m)\leq e^{-cn}$ for every $n\geq 1$.
Since $\Pr_t(|C(o)|\geq m)$ is an increasing function of $t$, we deduce
\begin{align}\label{bound}
p_m(t)\leq e^{-cn}
\end{align}
for every $t\leq y$. Pick $M>0$ small enough that $x+M \leq y$ and $e^{2M}e^{-c}<1$. 
Combined with \Tr{Pm entire}, this implies that $|p_m(z)|\leq Ca^m$ for $z\in D(x,M)$, where $C$ is a positive constant and $a<1$. Since $\sum_{m=1}^\infty Cm a^{m}<\infty$,
we can use the Weierstrass M-test to conclude that the sum $\sum_{m=1}^\infty mp_m(z)$ converges uniformly on $D(x,M)$ and since each $p_m$ is analytic the sum is also analytic. Moreover, this sum coincides with 
$\chi(t)$ for $t\in D(x,M)\cap [0,t_c)$, and so our statement follows.
\end{proof}

\begin{proof}[Proof of \Tr{chi NN}]
This is similar to the above, but instead of \Tr{Pm entire} we use the corresponding statement for \nnm s. This is easier, as the sum \eqref{sumS} is finite. Applying \Lr{C equals S NN} (using the bounded degree assumption, see also Remark~\ref{rem d}) yields an upper bound of the form $|p_m(z)|\leq  c^{dm} p_m(x+M)$ which we use instead of that of \Tr{Pm entire} in our application of the M-test. The rest of the proof is identical to that of \Tr{chi LR}.
\end{proof}

The above proofs show that there is an open disk centred at any subcritical value $x$ of the parameter where $p_m$ converges exponentially fast to $0$. Easily, every higher moment $\mathbb{E}(|C(o)|^k)=\sum_{m=1}^\infty m^k \Pr_{t}(|C(o)|=m)$ (or for the same reason, the expectation of every sub-exponential function of $|C(o)|$) admits an analytic extension on the same disk, and so we obtain

\begin{corollary} 
Every moment $\mathbb{E}_x(|C(o)|^k)$ is an analytic function of the parameter $x$ in the subcritical interval for all models as in \Tr{chi NN} or \Tr{chi LR}. 
\end{corollary}

Let us summarize the ideas used to prove the analyticity of $\chi$. Our proofs had little to do with $\chi$ itself. The main idea was to express $\chi$ as a sum of multiples of probabilities of events, namely $\chi(t)=\sum_{m=1}^\infty m \Pr_t(|C(o)|= m)$, and use the exponential decay of those probabilities (\Tr{AB thm}) to counter the exponential growth of their complex extensions (as in \Lr{C equals S NN}) in small enough discs around every point $p$. The rest of the proof was standard complex analysis, namely the Weierstrass M-test and \Wthm. As we are going to use the same proof structure several times, we reformulate it as the following corollary, which is a straightforward generalisation of the proof of \Tr{chi NN}. To formulate it, we need the following definition.

\begin{definition} \label{complexity}
We say that an event $E$ ---of a \nnm\ on a graph \G--- has \defi{complexity} $n$, if it is a disjoint union of a family of events $(F_i)_{i\in A}$ where each $F_i$ is measurable \wrt\  a set of edges of \g of cardinality $n$ and $A$ is a set of indices.
\end{definition}

\begin{corollary}\label{cor general}
Let $\Pr_p$ denote the law of a nearest-neighbour model, and let
$f(p)$ be a function that can be expressed as 
$f(p)=\sum_{n\in\mathbb{N}}\sum_{i\in L_n} a_i \Pr_p(E_{n,i})$ in an interval
$(a,b) \subseteq [0,1]$, where $a_n\in \mathbb{R}$, $L_n$ is a finite index set, and each $E_{n,i}$ is an event measurable with respect to
$\Pr_p$ (in particular, the above sum converges absolutely for every $p\in (a,b)$).
Suppose that 
\begin{enumerate}
\item \label{cg i} $E_{n,i}$ has complexity of order $\Theta(n)$, and 
\item \label{cg ii} there is a constant $0<c<1$ \st\ $\sum_{i\in L_n} a_i \Pr_p(E_{n,i}) = O(c^n)$ for $p\in (a,b)$. 
\end{enumerate}
Then there is a constant $\varepsilon>0$ such that $f(p)$ is analytic in $(a-\varepsilon,b+\varepsilon)$. 
\end{corollary}
(The analyticity on the larger interval $(a-\varepsilon,b+\varepsilon)$ rather than $(a,b)$ is needed to handle the case $p=1$, in which case we simply choose $b=1$. The proof shows that \ref{cg ii} holds on $(a-\varepsilon,b+\varepsilon)$ with $c$ replaced by some other constant smaller than $1$.)
\begin{proof}
We imitate the proof of Theorem \ref{chi LR}, except that instead of the \ABP\ we use our assumption \ref{cg ii}, and instead of \Lr{C equals S NN} we use its
generalisation Corollary \ref{D and not F NN}, which we apply to the sequence of events witnessing that $(E_{n,i})$ satisfies \ref{cg i}. (Note that the complexity of an event governs the exponential growth rate of the maximum modulus of the extension of its probability to a complex disc as a function of the radius of that disc.) For $p\in (a,b)$, we can use either \eqref{ubound} or \eqref{lbound}. To obtain the analyticity at a neighbourhood of $a$ we need to use \eqref{ubound}, while to obtain the analyticity at a neighbourhood of $b$ we need to use \eqref{lbound}.
\end{proof}

\comment{
\begin{corollary}  \label{cor general}
Let $\Pr_p$ denote the law of a \nnm, 
and let $f(p)$ be a function that can be expressed as $f(p)=\sum_{\nin} a_n \Pr_p(E_n)$ in an interval $p\in (a,b) \subseteq [0,1]$, where $a_n\in \R$ and each $E_n$ is an event measurable \wrt\ $\Pr_p$ (in particular, the above  sum converges absolutely \fe\ $p\in I$). 

Suppose that 
\begin{enumerate}
\item \label{cg i} $E_n$ has complexity of order $\Theta(n)$, and 
\item \label{cg ii} there is a constant $0<c<1$ \st\ $a_n \Pr_p(E_n) = O(c^n)$ on $(a,b)$. 
\end{enumerate}
Then there is a constant $\varepsilon>0$ such that $f(p)$ is analytic in $(a-\varepsilon,b+\varepsilon)$.
\end{corollary}
\begin{proof}
We imitate the proof of \Tr{chi LR}, except that instead of the \ABP\ we use our assumption \ref{cg ii}, and instead of \Lr{C equals S NN} we use its generalisation \Cr{D and not F NN}, which we apply to the sequence of events witnessing that $(E_n)$ satisfies \ref{cg i}. 
\end{proof}
}

{\bf Remark:} A similar statement for \lrm s can be formulated, and proved, along the same lines, except that we use the total $\mu$-weight rather than the cardinality of an edge-set in \Dr{complexity}.

\section{$p_\C< 1$ for non-amenable graphs} \label{sec na}


The \defi{(edge)-Cheeger constant} of a graph $G$ is defined as $h_E(G):=\inf_{S} \frac{|\partial_E S|}{|S|}$, where the infimum ranges over all finite subgraphs $S$ of \G.  When $h_E(G)>0$ we say that \g is \defi{non-amenable}. A well-known theorem of Benjamini \& Schramm \cite{BeSchrPer} states that $p_c(G)\leq \frac1{1+h_E(G)}$. We show here that the same bound applies to $p_\C$. We use the same technique as in the subcritical case (\Sr{secChi}), except that we replace the \ABP\ with an observation of Pete that the arguments of Benjamini \& Schramm imply the exponential decay above the aforementioned threshold of the `truncated' cluster size for non-amenable graphs.

\begin{theorem} \label{thm nonam}	
For every bounded degree graph $G$ with $h:=h_E(G)>0$, we have 
$p_\C \leq \frac1{1+h_E(G)}$.
\end{theorem}
\begin{proof}
By the definitions, we have $1-\theta(p)= \sum_n \Pr_p(|C(o)|=n)$. 

The statement follows if we can apply \Cr{cor general} for $I=(\frac1{1+h_E(G)},1]$ and $E_n:=\{|C(o)|=n\}$ (and $a_n=1$). So let us check that the assumptions of \Cr{cor general} are satisfied.
 
The exponential decay condition \ref{cg ii} is established in \cite{Anchored}, which states that for every $p\in (\frac1{1+h_E(G)},1]$ we have
$\Pr_p(|C(o)|=n)\leq \Pr_p(n \leq |C(o)| < \infty) < e^{-rn}$ for some constant $r=r(p)>1$, and it is clear from the proof that $r(p)$ is monotone in $p$.   

For condition \ref{cg i}, we note that if $d$ is the maximum degree of $G$, then $E_n$ has complexity at most $dn$, as it is the disjoint union of the events of the form $C(o)=S$ where $S$ ranges over all connected subgraphs of \g with $n$ vertices containing $o$. 
We have thus proved that all assumptions of \Cr{cor general} are satisfied as claimed.
\end{proof}

\comment{
	\begin{proof}[Long proof]
Let $o$ be any vertex of $G$ and $p\in (\frac1{1+h_E(G)},1]$. We will concentrate on the case $p<1$; the case $p=1$ can be shown along the same lines. 

By the definitions, we have $1-\theta(p)= \sum_n \Pr_p(n \leq |C(o)| < \infty)$.
It is proved in \cite[Proposition~12.9]{PeteBook} that
$\Pr_p(n \leq |C(o)| < \infty) < e^{-rn}$ for some constant $r=r(p)>1$, and it is clear from the proof that $r(p)$ is monotone in $p$. 
Moreover, we have $|\partial C(o)|\leq d |C(o)|$, where $d$ is the maximum degree of $G$.
Combining this fact with \Lr{C equals S NN} we obtain
\begin{align*}
\begin{split}
|P_z(|C(o)|=n, |\partial C(o)|=m)|\leq c^{m} P_{p+M}(|C(o)|=n, |\partial C(o)|=m) \\ \leq c^{dn}P_{p+M}(|C(o)|=n, |\partial C(o)|=m),
\end{split}
\end{align*}
\fe\ $z\in D(p,M)$, where $M$ is a positive number such that $p+M<1$, and $c=c_{M,p}=\Big(\frac{1-p+M}{1-p-M}\Big)$.
Thus summing over all possible $m$ we deduce 
$$|P_z(|C(o)|=n)|\leq c^{dn}P_{p+M}(|C(o)|=n).$$ 
Combining this with the above exponential decay inequality we have
\labtequ{rpM}{$|P_z(|C(o)|=n)|\leq c^{dn} e^{-r(p+M)n}$.}
The assertion follows as in the proof of \Tr{chi LR} and \Tr{chi NN}: choosing $M$ small enough that $c_{M,p}^d< e^{r(p+M)}$ guarantees that the sum of the upper bounds provided by \eqref{rpM} over all $n$ is finite, and so we can use the M-test and Weierstrass'  \Tr{thmWei} to deduce the analyticity of $\theta$ in $D(p,M)$.
	\end{proof}
}

{\it Remark}: The same proof applies if we replace $\theta$ by some other subexponential funtion of the restriction of $|C(o)|$ to finite values.

It is well-known that when $G$ is amenable and transitive, there can never be more than one infinite cluster, whence $p_c = p_u$ \cite{BKunique} where
$$p_u=\inf \{p\in [0,1]: \text{ there exists a unique infinite cluster}\}.$$ 
On the other hand, Benjamini \& Schramm \cite{BeSchrPer} conjectured that $p_c<p_u$ holds  on every non-amenable transitive graph.

It is natural to ask whether $\theta$ witnesses the phase transition at $p_u$ whenever $p_c<p_u$, i.e. whether $\theta$ is non-analytic at  $p_u$. It turns out that this is not the case, i.e.\ there are examples of \Cg s where $\theta$ is analytic at  $p_u$. 
Indeed, Thom \cite{ThomRemark}, refining the result of Pak \& Smirnova-Nagnibeda \cite{PakSmir}, proved that whenever the spectral radius $\rho(G)$ of $G$ is at most $1/2$ we have $p_c<p_u$. In fact, it follows from their proof that $p_u>\frac1{1+h_E(G)}$. Moreover they proved that $\rho(G)\leq 1/2$, and so $p_u>\frac1{1+h_E(G)}$, for some Cayley graph of any non-amenable group. (See \cite{SchonMulti} for other conditions that imply $p_u>\frac1{1+h_E(G)}$.) But then \Tr{thm nonam} yields that $\theta$ is analytic at $p_u$.


\section{Trees} \label{sec trees}

In this section we study the analyticity of $\theta$ on trees. We start with the case of regular trees as a warm-up, before we consider the more general Galton-Watson trees in \Sr{sec GW}.

\subsection{Regular trees}
It is well-known that if \g is a $d$-regular tree for $d>2$, then $p_c= \frac1{d-1}$ \cite{LyonsBook,PeteBook}, and it is easy to prove that $h_E(G)= d-2$ in this case. Thus \Tr{thm nonam} immediately yields
\begin{corollary} \label{cor tree}
If $T$ is the $d$-regular tree, then  $p_\C=p_c= \frac1{d-1}$.
\end{corollary}

For $d=3$ this is rather trivial, since $\theta(p)$ can be computed exactly using a recursive formula: we have $1- \theta = (1-p\phi)^{d}$, where $\phi$ satisfies the equation $1- \phi = (1-p\phi)^{d-1}$. For $d=3$ we have $\phi(p)=\frac{2p-1}{p^2}$ and hence $\theta(p)= 1- (1-\frac{2p-1}{p})^3$. We remark that this function is convex, corroborating the common belief about the shape of $\theta$ in general (see e.g.\ \cite[p.~148]{Grimmett}). The cases $d=4,5$ can also be solved exactly as they boil down to finding roots of polynomials of degree 3 and 4 respectively. For high values of $d$ the Abel--Ruffini theorem kicks in, and Galois theory implies that our equation is in general not soluble in terms of radicals.



For $d\geq 6$, an alternative way to prove \Cr{cor tree} is as follows. Consider the function $F(p,s)=s+(1-ps)^{d-1}-1$ and let $(p,s)$ be such that $F(p,s)=0$ and $p\in (\frac{1}{d-1},1]$. We will prove that $\dfrac{\partial F}{\partial s}\neq 0$. Then the implicit function theorem gives that $\phi$, hence $\theta$, is analytic. To this end, consider the function $g(x)=x^{d-1}$, and note that $\dfrac{\partial F}{\partial s}=1-pg'(1-ps)$. For $p=1$, we have $\dfrac{\partial F}{\partial s}=1$. For $p<1$, since $g$ is a strictly convex function on $(0,\infty)$, it lies above its tangents, i.e. $$g(1)>g(1-ps)+g'(1-ps)ps=1-s+g'(1-ps)ps=1-s\dfrac{\partial F}{\partial s},$$ which easily implies that $\dfrac{\partial F}{\partial s}>0$, and so we have re-proved \Cr{cor tree}.


Our technique yields a more probabilistic approach which some readers may prefer.

\subsection{Galton-Watson trees} \label{sec GW}

In this section we prove that $p_\C= p_c$ holds for almost every supercritical Galton-Watson tree $T$ defined by any progeny distribution. This answers a question of 
Michelen, Pemantle  \& Rosenberg \cite[Question 3]{MiPeRoQue}, who observed the analogous annealed statement (which is easier, and can be proved using the implicit function theorem as in the previous section). 

Before stating the result formally we introduce the relevant terminology. 
Let $\Pr$ denote the law of a Galton-Walton tree with a fixed progeny distribution $\{p_n\}_{\nin}$ with mean $\mu>1$. A random tree sampled from $\Pr$ will be denoted by $T$, and its root will be denoted by $o$. We will assume  that $p_0=0$ in order to avoid repeating the conditioning to non-extinction in our statements; it is well-known that this assumption can be made without loss of generality (see \cite{MiPeRoQue}, where $p_0=0$ is also assumed), and our proofs below easily adapt to the general case. 

Given a locally finite rooted tree $T$ with root $o$, we denote $\Pr_{p,T}$ the bond percolation probability measure on $T$, and $C_o$ the random cluster of $o$ sampled from $\Pr_{p,T}$.

\begin{theorem} \label{thm GW}
For $\Pr$-almost every $T$, the percolation density $\theta_T(p)$ is analytic in the interval $(p_c,1]$.
\end{theorem}

We remark that it is perhaps a-priori not obvious that the analyticity of $\theta_T(p)$ in $(p_c,1]$ is $\Pr$-measurable, but our proof establishes it indirectly by showing that it is implied by other properties that are clearly $\Pr$-measurable and are satisfied almost surely. 

This proof of this is also based on our general method via \Cr{cor general}. The required exponential decay will be established in the following lemmas. 

Percolation on a Galton-Watson tree can be realised as another Galton-Watson tree. We will write $\Pr_p$ for the corresponding probability measure of the annealed model, i.e.\ $\Pr_p$ first samples a Galton-Watson tree with law $\Pr$, and then percolates it with parameter $p$. A random tree sampled from $\Pr_p$ will be denoted as $\mathcal{T}_p$. We write $\theta_T(p):= \Pr_{p,T}(|C_o|=\infty)$ for the corresponding `quenched' percolation density. It is well-known that $p_c(T)$ is almost surely equal to $1/\mu$, and so we let $p_c:= 1/\mu$. See \cite{Remco, LyonsBook} for an introduction to Galton-Watson trees and these statements.

It is well-known that conditioned on extinction, a supercritical Galton-Watson tree is distributed as a subcritical Galton-Watson tree with offspring distribution having exponential moments, hence the total progeny has an exponential tail (see e.g. \cite[Theorem 3.8]{Remco}). When $p>p_c$, $\mathcal{T}_p$ is a supercritical Galton-Watson tree, hence we have

\begin{proposition}\cite{Remco}\label{classic}
For every $p\in (p_c,1]$, there is a constant $c=c(p)>0$ such that $\Pr_p(|\mathcal{T}_p|=n)\leq e^{-cn}$ for every $n\geq 1$.
\end{proposition}

In the next lemma we turn this annealed statement into a quenched one, i.e.\ we prove the exponential decay of $\Pr_{p,T}(|C_o|=n)$ for every fixed $p>p_c$ and $\Pr$-almost every $T$.

\begin{lemma}\label{Borel-Cantelli}
Let $p\in (p_c,1]$, and let $c=c(p)>0$ be a constant such that $\Pr_p(|\mathcal{T}_p|=n)\leq e^{-cn}$ for every $n\geq 1$. Then for $\Pr$-almost every $T$ there is some $N=N(T,p)$ large enough  that 
$\Pr_{p,T}(|C_o|=n)\leq e^{-cn/2}$ for every $n\geq N$.
\end{lemma}
\begin{proof}
Let $B_n$ denote the following event, measured in the $\sigma$-algebra of $\Pr$:
$$\{\Pr_{p,T}(|C_o|=n) \geq e^{cn/2} \Pr_p(|\mathcal{T}_p|=n)\}.$$
Clearly, $$\Ex\big( \Pr_{p,T}(|C_o|=n)\big)=\Pr_p(|\mathcal{T}_p|=n),$$ 
and so Markov's inequality implies that $$\Pr(B_n)\leq e^{-cn/2}.$$
Since $\sum_{n=1}^\infty e^{-cn/2}<\infty$, we can apply the Borel-Cantelli lemma to obtain that $\Pr$-almost surely, $B_n$ occurs only for finitely many values of $n$. In other words, for $\Pr$-almost every $T$ there is some $N=N(T,p)$ large enough such that $$\Pr_{p,T}(|C_o|=n) \leq e^{cn/2} \Pr_p(|\mathcal{T}_p|=n)\leq e^{-cn/2}$$ for every $n\geq N$, as desired.
\end{proof}

In order to apply \Cr{cor general}, we need to extend this exponential decay uniformly to an open interval of each $p>p_c$. We do so in \Lr{exp dec trees} below, by using the following statement that compares $\Pr_{r,T}(|C_o|=n)$ for nearby values of $r$.

\begin{lemma}\label{cover}
Consider a locally finite tree $T$. Let $p\in (0,1)$ and $n\geq 1$. Then $$\Pr_{r,T}(|C_o|=n)\leq (r/p)^{n-1}\Pr_{p,T}(|C_o|=n)$$ 
for every $r\geq p$. Moreover, there is a constant $l=l(p)>0$ such that 
$$\Pr_{r,T}(|C_o|=n)\leq \Big(\dfrac{1-r}{1-p}\Big)^{s(n-1)}\Pr_{p,T}(|C_o|=n)+e^{-l(n-1)}$$
for every $p\geq r\geq \frac{p}{1+p}$, where $s=\frac{1+p}{p}$.
\end{lemma}
\begin{proof}
By the definitions, for every $r\in [0,1]$ we have $$\Pr_{r,T}(|C_o|=n)=\sum_{S\in \mathcal{S}_n(T)} \Pr_{r,T}(C_o=S),$$ where $\mathcal{S}_n(T)$ is the set of all finite subtrees $S$ of $T$ containing the root with $n$ vertices. Clearly $\Pr_{r,T}(C_o=S)= r^{n-1} (1-r)^{|\partial S|}$, where $\partial S$ is the the set of those edges of $T$ with exactly one endvertex in $S$. We can easily compute that for every $r\geq p$ and any $S\in \mathcal{S}_n(T)$
$$r^{n-1} (1-r)^{|\partial S|}\leq r^{n-1} (1-p)^{|\partial S|}= (r/p)^{n-1}  p^{n-1}(1-p)^{|\partial S|}.$$
Hence for every $r\geq p$ we have $$\Pr_{r,T}(|C_o|=n)\leq (r/p)^{n-1}\Pr_{p,T}(|C_o|=n)$$ 
as claimed. For the second statement, for every $r\leq p$ and any $S\in \mathcal{S}_n(T)$ we have 
$$r^{n-1} (1-r)^{|\partial S|}\leq p^{n-1} (1-r)^{|\partial S|}=\Big(\dfrac{1-r}{1-p}\Big)^{|\partial S|} p^{n-1} (1-p)^{|\partial S|}.$$

We will consider two cases according to whether $|\partial S|\leq s(n-1)$ or not, where $s=\frac{1+p}{p}$. For those $S\in\mathcal{S}_n(T)$ with $|\partial S|\leq s(n-1)$ the above inequality gives 
$$\Pr_{r,T}(C_o=S)\leq \Big(\dfrac{1-r}{1-p}\Big)^{s(n-1)} \Pr_{p,T}(C_o=S).$$
On the other hand, the function $g(r)=r(1-r)^s$ is strictly decreasing on the interval $[\frac{1}{1+s},1]$. This implies that there is a constant $c<1$ such that $$r(1-r)^s \leq c q(1-q)^s$$ for every $r\geq \frac{p}{1+p}$, where $q=\frac{1}{1+s}$. It is easy to see that for every $t\geq s$,
$$r(1-r)^t=r\Big(\dfrac{1-r}{1-q}\Big)^t (1-q)^t\leq r\Big(\dfrac{1-r}{1-q}\Big)^s (1-q)^t \leq cq(1-q)^t.$$
Therefore, for every $p\geq r\geq \frac{p}{1+p}$ we have
\begin{gather*}
\Pr_{r,T}(|C_o|=n)\leq \Big(\dfrac{1-r}{1-p}\Big)^{s(n-1)}\Pr_{p,T}(|C_o|=n)+c^{n-1}  \Pr_{q,T}(|C_o|=n) \leq \\ \Big(\dfrac{1-r}{1-p}\Big)^{s(n-1)}\Pr_{p,T}(|C_o|=n)+c^{n-1},
\end{gather*}
as desired.
\end{proof}

We are now ready to prove the exponential decay of $\Pr_{p,T}(|C_o|=n)$ in an open interval of $p$ for all values of $p>p_c$.

\begin{lemma}\label{exp dec trees}
For $\Pr$-almost every $T$ and every $p\in (p_c,1)$, there is an interval $(a,b)\subset (p_c,1)$ containing $p$, and a constant $t=t(T,p)>0$, such that $\Pr_{r,T}(|C_o|=n)\leq e^{-tn}$ for any $r\in (a,b)$ and any $n\geq 1$.
\end{lemma}
\begin{proof}
Consider some $p'>p_c$, and let $c=c(p')>0$ be the constant of \Lr{Borel-Cantelli}. Then there are positive constants $a=a(p'),b=b(p')$ with $\dfrac{p'}{1+p'}<a<p'<b$ such that both $re^{-c/2}/p'<1$ and $\Big(\dfrac{1-r}{1-p'}\Big)^s e^{-c/2}<1$ hold for every $r\in [a,b]$, where $s=\dfrac{1+p'}{p'}$.
As $p'$ varies over the interval $(p_c,1)$, the collection of the intervals $(a,b)$ covers $(p_c,1)$. Hence we can extract a countable collection of intervals $(a_i,b_i)$, $i\in I$ that covers $(p_c,1)$. In particular, at least one of these intervals $(a_i,b_i)$ contains $p$, and we choose $a:=a_i$ and $b:=b_i$. Write $p_i$ for the point in $(p_c,1)$ giving rise to $(a_i,b_i)$. 

Applying \Lr{Borel-Cantelli} to all $p_i$ and using the union bound, we deduce that for $\Pr$-almost every $T$ there is some $N=N(T,p_i)$ such that $$\Pr_{p_i,T}(|C_o|=n) \leq e^{-c_i n/2}$$ for every $n\geq N$ and every $i\in I$, where $c_i=c(p_i)$. By \Lr{cover} and our choice of $a_i$, $b_i$, for those Galton-Watson trees $T$, there is a constant $k=k(p_i)>0$ such that $$\Pr_{r,T}(|C_o|=n) \leq e^{-kn}$$ for every $r\in (a_i,b_i)$ and every $n\geq N$. For $n\leq N$ we can use the continuity of $\Pr_{r,T}(|C_o|=n)$ as a function of $r$ and the fact that $\Pr_{r,T}(|C_o|=n)<1$ to conclude that there is a constant $t$ with the desired properties.
\end{proof}

As we will see, \Tr{exp dec trees} implies the analyticity of $\theta_T$ in the open interval $(p_c,1)$. To prove the analyticity at a neighbourhood of $1$, we need the following lemma.

\begin{lemma}\label{calculations}
Consider a constant $h>0$ and the function $f(p) = p^n(1-p)^m$, where $n,m$ are 
positive integers such that $m\geq hn$. Then there is a constant $0<s<1$ such that $f(p)\leq s^m f(r)$ for every $p\in [t,1]$, where $r=1/(1+h)$ and $t=2/(2+h)$.
\end{lemma}
\begin{proof}
Consider the function $g(p) = p^{1/h}(1-p)$ and observe that
$f(p) = (g(p))^m p^{n-m/h}$. For every $p\in [t,1]$ we have that $p^{n-m/h}\leq r^{n-m/h}$, because $n-m/h\leq 0$. We will now compare $g(p)$ with $g(r)$. To this end, observe that $g$ is strictly decreasing on the interval $[r,1]$. Hence there is some constant $0 < s < 1$ such that $g(p)\leq
sg(r)$ for every $p\in [t,1]$. Raising the inequality to the power of $m$, we obtain $(g(p))^m \leq  s^m (g(r))^m$, which implies that $f(p)\leq s^m f(r)$,
as desired.
\end{proof}

We are now ready to prove the analyticity of $\theta_T$.

\begin{proof}[Proof of \Tr{thm GW}]
For the analyticity of $\theta_T(p)$ in $(p_c,1)$, we write $\theta_T(p)=1-\sum_{n=1}^{\infty} \Pr_{p,T}(|C_o|=n)$. The complexity of the events $\{|C_o|=n\}$ is not necessarily of order $n$ when $T$ has unbounded degree, and so we cannot immediately use Corollary \eqref{cor general}, but since we have $|E(C_o)|=|C_o|-1$, we can use the theorems of Weierstrass and \eqref{lbound}, as in the proof of Corollary \eqref{cor general}, coupled with \Tr{exp dec trees} to conclude.  

To prove the analyticity at a neighbourhood of $1$, we will use a different expression, namely $\theta_T(p)=1-\sum_{m=1}^{\infty} \Pr_{p,T}(|C_o|<\infty,|\partial C_o|=m)$. To this end, recall that $T$ is non-amenable almost surely \cite[Theorem 6.52]{LyonsBook}. Let us denote the Cheeger constant of $T$ by $h=h(T)$. For each possible connected component $S$ of $o$, we have $\Pr_{p,T}(C(o)=S)=p^n(1-p)^m$ where $m\geq hn$. Using \Lr{calculations} we obtain $\Pr_{p,T}(C(o)=S)\leq s^m \Pr_{r,T}(C(o)=S)$ for every $p\in [t,1]$, where $s,r,t$ are as in \Lr{calculations}, and thus $$\Pr_{p,T}(|C_o|<\infty,|\partial C_o|=m)\leq s^m\Pr_{r,T}(|C_o|<\infty,|\partial C_o|=m)\leq s^m$$ for every $p\in [t,1]$. We can now use Corollary \eqref{cor general} to conclude.
\end{proof}

\subsection{A tree with $\theta$ nowhere analytic on $(p_c,1]$} \label{nowhere}
We finish this section by giving an example of a (deterministic) tree $T$ for which $\theta$ is nowhere analytic on the interval $(p_c,1]$. In fact, we will show that $\theta$ is not differentiable at every $p\in (p_c,1)\cap \mathbb{Q}$. 
We recall that $\theta$ is an increasing function of $p$. It is a standard fact that increasing functions are differentiable almost everywhere, so in some sense, the constructed $\theta$ is the least well-behaved percolation density possible.

For every $q\in(1/2,1)\cap \mathbb{Q}$, we consider a rooted tree $T_q$, and write $\theta_{T_q}$ for the corresponding percolation density with respect to the root, which satisfies the following properties:
\begin{enumerate}[(i)]
\item \label{p_c} $p_c(T_q)=q$,
\item \label{zero} $\theta_{T_q}(p_c)=0$,
\item \label{derivative} the right derivative $\dfrac{d\theta_{T_q}(p_c)}{dp^+}$ exists and is strictly positive.
\end{enumerate}
For example, we can sample $T_q$ from the law $\Pr^q$ of the Poisson Galton-Watson distribution of parameter $1/q$, i.e. the Galton-Watson tree with progeny distribution $p_k=\lambda^k e^{-\lambda}/k!$, $\lambda=1/q$. We recall that properties \eqref{p_c},\eqref{zero} are satisfied $\Pr^q$-almost-surely conditioning on non-extinction \cite{LyonsTrees}. To ensure that \eqref{derivative} is satisfied, we can use e.g \cite[Proposition 2.2, Theorem 1.2]{MiPeRoQue}, which state that the right derivative of the annealed percolation density $g(p)$ at $p_c$ exists and is strictly positive, that the right derivative of the quenched percolation density at $p_c$ exists $\Pr^q$-almost-surely, and that $$\dfrac{dg(p_c)}{dp^+}=\mathbb{E}\Big(\dfrac{d\theta_{T_q}(p_c)}{dp^+}\Big).$$

We now construct $T$ as follows. Let $P$ be an infinite path with vertex set $\{u_0=o,u_1,\ldots\}$ and edge set  $\{\{u_n,u_{n+1}\} \mid \nin\}$. Consider a sequence $(T'_n)$ of trees, where each $T'_n$ coincides with some $T_q$, and each $T_q, q\in(1/2,1)\cap \mathbb{Q}$ appears infinitely many times in the sequence. (We think of $T_q$ as a fixed deterministic tree, although we used randomness to prove its existence.) Now define the tree $T''_n$ by taking $k_n$ copies of $T'_n$ and identifying their roots, where $k_0=1$ and for every $n\geq 1$, $k_n$ is a large enough integer that
\begin{equation}\label{large kn}
k_n p^n(1-p)\dfrac{d\theta_{T'_n}(p)}{dp^+}\prod_{i=0}^{n-1} (1-\theta_{T''_i}(p))\geq n
\end{equation}
for $p=p_c(T'_n)$. That such a number $k_n$ exists follows from \eqref{derivative} and the fact that $\theta_{T''_i}(p)<1$ at $p=p_c(T'_n)<1$.
Then for every $n\in\mathbb{N}$, we identify the root of $T''_n$ with $u_n$ to obtain $T$ from $P$. Clearly, $p_c(T)=1/2$.

We claim that $\dfrac{d\theta_{T}(p)}{dp^+}=\infty$ for every $p\in(1/2,1)\cap \mathbb{Q}$, which immediately implies that $\theta_{T}$ is nowhere analytic in $(1/2,1]$. To prove the claim, expressing $\theta=\theta_{o,T}$ according to the largest open path in $P$ starting from $o$ we obtain
$$\theta(p)=1-\sum_{n=0}^\infty p^n (1-p) \prod_{i=0}^n (1-\theta_{T''_i}(p))$$
for every $p\in(0,1)$. Consider some $p\in(1/2,1)\cap \mathbb{Q}$ and $0<h<1-p$. Then we have
\begin{equation} \label{thetah}
\begin{gathered}
\dfrac{\theta(p+h)-\theta(p)}{h}=\sum_{n=0}^\infty \dfrac{p^n(1-p)-(p+h)^n(1-p-h)}{h}\prod_{i=0}^n (1-\theta_{T''_i}(p))\\
+\sum_{n=0}^\infty (p+h)^n(1-p-h)\dfrac{\prod_{i=0}^n (1-\theta_{T''_i}(p))-\prod_{i=0}^n (1-\theta_{T''_i}(p+h))}{h}
\end{gathered}
\end{equation}
by adding and subtracting the term $\sum_{n=0}^\infty \dfrac{(p+h)^n(1-p-h)}{h}\prod_{i=0}^n (1-\theta_{T''_i}(p))$.

We claim that the first sum remains bounded when $0<h\leq h_0$, where $h_0$ is any constant in $(0,1-p)$. Indeed, we have $0\leq \prod_{i=0}^n (1-\theta_{T''_i}(p))\leq 1$ and $$\Big\lvert\dfrac{p^n(1-p)-(p+h)^n(1-p-h)}{h}\Big\rvert\leq n(p+h_0)^{n-1}(1-p)+(p+h_0)^n$$ by the Mean value theorem. Since the latter bound is clearly summable, our claim is proved.  So let us focus on the second sum of \eqref{thetah}, which we denote by $f(p,h)$. By the monotonicity of $\theta_{T''_i}$ we have that for every $n$,
\begin{equation}\label{deri bound}
\begin{aligned}
f(p,h) &\geq (p+h)^n(1-p-h)\dfrac{\prod_{i=0}^n (1-\theta_{T''_i}(p))-\prod_{i=0}^n (1-\theta_{T''_i}(p+h))}{h} \\
&\geq (p+h)^n(1-p-h)\dfrac{\theta_{T''_n}(p+h)-\theta_{T''_n}(p)}{h}\prod_{i=0}^{n-1} (1-\theta_{T''_i}(p)).
\end{aligned}
\end{equation}
Recalling the definition of $T''_n$, we see that $\theta_{T''_n}=1-(1-\theta_{T'_n})^{k_n}$. Choosing an integer $n$ so that $p_c(T'_n)=p$ and using the assumptions \eqref{zero}, \eqref{derivative}, we deduce that the right derivative of $\theta_{T''_n}$ at $p=p_c(T'_n)$ exists and is equal to 
$k_n\dfrac{d\theta_{T'_n}(p)}{dp^+}$. Hence the last term in \eqref{deri bound} converges to the left hand side of \eqref{large kn} as $h$ tends to $0$.  By construction there are infinitely many $n$ such that $p_c(T'_n)=p$, so letting $n$ tend to infinity along those values shows that
$\dfrac{d\theta_{T}(p)}{dp^+}=\infty$, as desired.

\section{Analyticity above the threshold for planar lattices} \label{sec th pl}

A \defi{\qtl} (in $\R^2$) is a connected, locally finite, plane graph $L$ \st\ for some pair of linearly independent vectors $v_1,v_2 \in \R^2$, translation by each $v_i$ preserves $L$, and this action has finitely many orbits of vertices. We will assume, as we may, that the edges of $L$ are piecewise linear curves.

\medskip
{\it Remark:} The seemingly more general definition as a plane graph admitting a semiregular action of the group $\Z^2$ (by isometries of $\R^2$ preserving $L$, or even more generaly by arbitrary graph-theoretic isomorphisms) with finitely many orbits of vertices can be proved to be in fact equivalent, but we will not go into the details; the main idea is to embed a fundamental domain of $L$ \wrt\ that action in a square, and then tile $\R^2$ by copies of that square. Another approach can be found in \cite[Proposition 2.1]{BeSchrPerHyp}. \Tr{theta analytic 2D} does not apply to a lattice $H$ in hyperbolic 2-space just because \Tr{pc star} below fails, but our proof shows that $p_\C(H) \leq 1-p_c(H^*)$, where $H^*$ denotes the dual of $H$. In this case we have $1-p_c(H^*)= p_u(H)$ by \cite[Theorem 3.8]{BeSchrPerHyp}\footnote{The fact that non-amenability is equivalent to hyperbolicity in this setup is well-known, see e.g.\ \cite{hypamen}.}; in other words, we have shown analyticity of $\theta$ above $p_u$ for all planar lattices. 
\medskip

In this section we prove

\begin{theorem}\label{theta analytic 2D}
For Bernoulli bond percolation on any \qtl\ we have $p_\C= p_c$.
\end{theorem}

This result is new even for the standard square lattice \Zs, i.e.\ the Cayley graph of $\Z^2$ \wrt\ the standard generating set $\{(0,1),(1,0)\}$. Slightly more effort is needed to prove it in the generality of \qtl s. The reader that just wants to see a simplest possible proof for the lattice $L=\Zs$ is advised to:
\begin{itemize} 
\item ignore \Tr{pc star}, and just recall that $p_c(\Zs)=1/2$ and $\Zsd=\Zs$; 
\item skip the definition of $X$ in \Sr{prel qtl}, and instead take $X$ to be the horizontal `axis' of \Zs, and $X^+$ the right `half-axis' starting at the origin $o$; and
\item notice that \autoref{quasi geo} holds trivially with $f=1$. 
\end{itemize} 

We will use the following important fact about the relation between the percolation thresholds in the primal and dual lattice. The history of this result starts with the Harris-Kesten theorem that $p_c(\Zs)=1/2$. A special case was obtained by Bollob{\'a}s \& Riordan \cite{BoRiPer}, and almost simultaneously the general case was proved by Sheffield \cite[Theorem 9.3.1]{SheRan} in a rather involved way. A shorter proof can be found in \cite{DCRaTaSha}.

\begin{theorem}[\cite{SheRan}; see also \cite{DCRaTaSha}] \label{pc star}
\Fe\ \qtl\ $L$, we have $p_c(L) + p_c(L^*)=1$.
\end{theorem}

The analogue of \Tr{theta analytic 2D} for Bernoulli site percolation 
can be proved along the same lines, see \Cr{triangular}. 

\subsection{Preliminaries on \qtl s} \label{prel qtl}

We will construct a 2-way infinite path $X$ in any \qtl\ $L$, which can be thought of as a  `quasi-geodesic' of both $L$ and $L^*$. Alternatively, we could take $X$ to be a 2-way infinite geodesic of $L$ and prove \Prr{quasi geo} below differently, but the approach we follow is not more complicated and has the advantage that it avoids the axiom of (countable) choice.

Since $L$ is a plane graph, we naturally identify $V(L)$ with a set of points of $\R^2$. Let $o\in \R^2$ be a vertex of $L$ and recall that $o + k v_1 \in V(L)$ for some non-zero vector $v_1\in \R^2$ and every $k\in \Z$. Fix a path $P$ from $o$ to $o+v_1$. We may assume \obda\ that $P$ does not contain $o+k v_1$ for $k\neq 0,1$, for otherwise we can replace $v_1$ by one of its multiples and $P$ by a subpath. Note that the union $\bigcup_k P + k v_1$ of its translates along multiples of $v_1$ contains a 2-way infinite path $X$. Moreover, it is not too hard to see that we can choose $X$ to be periodic, i.e.\ to satisfy $X+ t v_1= X$ for some $t\in \N$.
For convenience, we will assume without loss of generality that $o$ lies in $X$. Let $X^+=(x_0=o,x_1,\ldots), X^-=(\ldots,x_{-1},x_0=o)$ denote the two $1$-way infinite sub-paths of $X$ starting at $o$.

\begin{proposition} \label{quasi geo}
Let $L$ be a \qtl\ and $X^+$ the infinite path defined above. Then there is a constant $f=f(L)>0$ such that for every connected subgraph of $L$ that surrounds $o$ and has $N$ edges must contain one of the first $fN$ vertices $x_0,x_1,\ldots,x_{fN-1}$ of $X^+$. Similarly, every connected subgraph of $L^*$ that surrounds $o$ and has at most $N$ edges must cross one of the first $fN$ edges $e_1,e_2,\ldots,e_{fN}$ of $X^+$.
\end{proposition}
\begin{proof}
We first claim that $X$ is a \textit{quasi-geodesic}, i.e. there is a constant $c>0$ such that for every $x_i,x_j\in X$ we have $d_L(x_i,x_j)\geq c |i-j|$, where $d_L$ denotes the graph-distance in $L$.
This is not too hard to see directly, but it can be immediately deduced from the \v{S}varc--Milnor lemma \cite{Mi68}, which states that if 
a group $\Gamma$ acts  properly discontinuous and co-compactly by isometries on a geodesic metric space $M$, then any finitely generated \Cg\ of $\Gamma$ is canonically quasi-isometric to $M$. We apply this with $\Gamma =\Z^2$ twice, once with $M= \R^2$ and once with  $M= L$ endowed with its graph-metric, to deduce that $X$ is a quasi-geodesic of $L$ from the fact that $X$ is a quasi-geodesic of $\R^2$ by construction. Here the action of $\Z^2$ on $\R^2$ and $L$ is given by the translation by the vectors $v_1,v_2$ as in the definition of a \qtl.

We can apply the same argument to $L^*$ to deduce that $d_{L^*}(e_i,e_j)\geq c' |i-j|$ for some constant $c'>0$. Indeed, the aforementioned action of  $\Z^2$ on $\R^2$ induces an action of $\Z^2$ on $L^*$ by the definitions, and this action is co-compact too because $L$ is locally finite.

\medskip
Suppose now that some connected graph $S\subset L$ surrounds $o$. Then $S$ must separate $o$ from infinity, and so it must contain a vertex $x^+$ in $X^+$, and a vertex $x^-$ in $X^-$ ($x^+$ and $x^-$ may possibly coincide). If $S$ has at most $N$ edges, then the graph-distance between $x^+$ and $x^-$ is at most $N$ because $S$ is a connected graph. Since $X$ is a quasi-geodesic, the indices of $x^+$ and $x^-$ differ by at most $N/c$. We now see that $x^+$ is one of the first $1+N/c$ vertices of $X^+$.

The second sentence can be proved similarly by using the inequality\\  $d_{L^*}(e_i,e_j)\geq c' |i-j|$ from above.

\end{proof}

\subsection{Main result} \label{sec proof th}

Let $L$ be a \qtl, and $o$ a vertex of $L$ fixed throughout this section.
A subgraph $S$ of $L$ is called an \defi{\scv} (of $o$) if there is a finite connected subgraph $H$ of $L$ containing $o$ such that $S$ consists of the vertices and edges incident with the unbounded face of $H$.

The \defi{boundary $\partial S$} of an \scv\ $S$ is the set of edges of $L$ that are incident with $S$ and lie in the unbounded face of $S$. It is important to remember that $\partial S$ may contain edges that have both their end-vertices in $S$; our proof will break down (at \Lr{scs boundary}) if we exclude such edges from the definition of $\partial S$. Let $|S|:=|E(S)|$ be the number of edges in $S$.

\comment{\begin{lemma} \label{scs axis}	
	Every {\scv } contains one of the vertices $x_0,\ldots, x_{|S|-1}$, where $x_i:=(i,0)$ is the $i$th vertex on the positive horizontal axis of \Zs.
\end{lemma}
	\begin{proof}
}

Given a realisation $\oo\in 2^{E(L)}$ of our Bernoulli percolation on $L$, we say that an \scv\ $S$ \defi{occurs} in \oo\ if $S$ is the boundary of the unbounded face of some cluster of \oo. This happens exactly when all edges of $S$ are occupied and all edges in $\partial S$ are vacant. 

The following is an easy consequence of the definitions.
\begin{lemma} \label{scs disjoint}	
If two occurring {\scv s}  share a vertex then they coincide. \qed
\end{lemma}

The following is one of the reasons why our proof only applies to lattices rather than arbitrary planar graphs. Roughly speaking, it states that \scv s are uniformly non-amenable.

\begin{lemma} \label{scs boundary}	
For every {\scv } $S$ we have $| \partial S| \geq |S|/k$ for some integer $k=k(L)$. 
\end{lemma}
For example, if $L$ is the square lattice $\Zs$, then $k=2$. (And not $k=1$, because it can happen that most edges in $\partial S$ have both their end-vertices in $S$; for example, we can have a `space filling' \scv\ whose vertex set is an $n \times n$ box of \Zs. The following proof will give a worse bound than $k=2$, but we can afford to be generous.)
\begin{proof}
By quasi-transitivity, any face of $L$ has at most $C$ edges for some $C>0$. 
Recall that $S$ consists of the vertices and edges incident with the unbounded face $F$ of some $H \subset L$. If we walk along the boundary $S$ of $F$, we will never traverse $C$ or more edges of $S$ without encountering an edge in $\partial S$, and we will encounter each edge in $\partial S$ at most twice. Thus our assertion holds for $k= 2(C-1)$. 
\end{proof}

A \defi{\smc} $M$ is a finite set of pairwise vertex-disjoint \scv s. 

\begin{lemma} \label{dual connd}	
For every \scv\ $S$, the edge-set $\partial S^*$ spans a connected subgraph of $L^*$ surrounding $o$. Similarly, for every \smc\ $M$, the edge-set $\partial M^*$ spans a subgraph of $L^*$ the number of  components of which equals the  number of \scv s in $M$ (and each of these components surrounds $o$).
\end{lemma}
\comment{
\begin{proof}
Recall that $S$  consists of the vertices and edges incident with the unbounded face $F$ of a finite connected subgraph $H$ in a fixed embedding of $L$. Let $J$ be a Jordan curve disjoint from $H$ such that $H$ lies in the bounded side of $J$, and $J$ is close enough to $H$ that it meets all edges in $\partial S$ and no other edges of $L$. Then the cyclic sequence of faces and edges of $L$ visited by $J$ defines a closed walk in $L^*$, proving that $\partial S^*$ spans a connected subgraph of $L^*$. That this subgraph surrounds $o$ is an immediate consequence of the definition of an \scv.

Now let $M$ be a \smc\ comprising the \scv s $S_1, \ldots, S_m$. We just proved that each $\partial S_i^*$ spans a connected subgraph of $L^*$, so it only remains to show that $\partial M^*$ contains no path between $\partial S_i^*$ and $\partial S_j^*$ for $i\neq j$. Since $S_i$ and $S_j$ are vertex-disjoint, one of them is contained in a bounded face of the other. Let us assume that $S_i$ is contained in a bounded face of $S_j$. Then it is easy to see that the edges of $S_i$ contain a cut of $L^*$ separating $\partial S_i^*$ from $\partial S_j^*$. Since $\partial M^*$ contains no edge of $S_i^*$ by the vertex-disjointness of the $S_i$, this proves our claim that $\partial M^*$ contains no path between $\partial S_i^*$ and $\partial S_j^*$.
\end{proof}
}
\begin{proof}
Recall that $S$ consists of the vertices and edges incident with the unbounded face $F$ of a finite connected subgraph $H$ in a fixed embedding of $L$. Hence $S$ coincides with the topological boundary of the unbounded component of the complement of $S$ in $\mathbb{R}^2$. Let $S'$ be the union of the bounded components of the complement of $S$. We claim that there is a Jordan curve $J$ disjoint from $H$ such that $H$ lies in the bounded side of $J$, and $J$ is close enough to $H$ that it meets all edges in $\partial S$ and no other edges of $L$. Indeed, since the action defined by the vectors $u_1,u_2$ has finitely many orbits of vertices, for each endvertex $u$ of an edge in $S$, there is a constant $\epsilon=\epsilon(u)>0$ such that every other vertex of $L$ has distance at least $\epsilon$ from $u$. Consider $0<\delta<\epsilon/2$ small enough such that for every edge $e$ incident to $u$, the subset of $e$ lying in the open disk $D(u,\delta)$ is a line segment. Notice that all these disks are pairwise disjoint. In this way we can cover the vertices of $V(S)$ and some initial segments of its edges. By compactness, we can cover the uncovered subarcs of any edge in $S$ by finitely many open disks that are disjoint from any edge of $\partial S$ in such a way that disks corresponding to distinct edges are disjoint. Taking the union of $S'$ with all these open disks covering the edges of $S$, we obtain a simply connected domain. Hence its boundary is connected, and in fact a Jordan curve because it consists of a finite number of circular arcs. This Jordan curve has clearly the desired properties.

The cyclic sequence of faces and edges of $L$ visited by $J$ defines a closed walk in $L^*$. From the construction of $J$ we have that $J$ crosses each edge of $\partial S$ only once or twice (when both endvertices lie in $V(S)$). Therefore, the closed walk in $L^*$ is finite, proving that $\partial S^*$ spans a connected subgraph of $L^*$. That this subgraph surrounds $o$ is an immediate consequence of the definition of an \scv.

Now let $M$ be a \smc\ comprising the \scv s $S_1, \ldots, S_m$. We just proved that each $\partial S_i^*$ spans a connected subgraph of $L^*$, so it only remains to show that $\partial M^*$ contains no path between $\partial S_i^*$ and $\partial S_j^*$ for $i\neq j$. Since $S_i$ and $S_j$ are vertex-disjoint, one of them is contained in a bounded face of the other. Let us assume that $S_i$ is contained in a bounded face of $S_j$. Then it is easy to see that the edges of $S_i$ contain a cut of $L^*$ separating $\partial S_i^*$ from $\partial S_j^*$. Since $\partial M^*$ contains no edge of $S_i^*$ by the vertex-disjointness of the $S_i$, this proves our claim that $\partial M^*$ contains no path between $\partial S_i^*$ and $\partial S_j^*$.
\end{proof}

Let $\MS$ denote the set of \smc s of $L$. We say that $M \in \MS$  \defi{occurs} if each of the \scv s it contains occurs.  Let $|M| := \sum_{S_i \in M} |S_i|$ be the total number of edges in $M$. Let $\partial M:= \bigcup_{S_i \in M} \partial S_i$, and let
 $\MS_n:= \{ M \in \MS \mid |\partial M|=n\}$ be the set of \smc s with $n$ boundary edges. 

\begin{lemma} \label{HR bound}	
There is a constant $r\in \R$ \st\ \fe\ $n\in \N$ at most $r^{\sqrt{n}}$ elements of $\MS_n$ can occur simultaneously in any percolation instance \oo.
\end{lemma}
\begin{proof}
Suppose $M\in \MS_n$ occurs in \oo. 
Since occurring \scv s meet $X^+$ by \autoref{quasi geo}, and they are vertex-disjoint by \Lr{scs disjoint}, $M$ is uniquely determined by the subset $D$ of $\{x_0, x_1, \ldots\}$ it meets, in other words, $M= \bigcup_{x_i\in D} S(x_i,\oo)$, where $S(x_i,\oo)$ denotes the occurring \scv\ containing $x_i$.

Note that  $|S(x_i,\oo)| > i/f $ \fe\ $x_i \in D$ by \autoref{quasi geo}. Since $k n\geq |M| = \sum_{x_i\in D} |S(x_i,\oo)|$ by \Lr{scs boundary} and the above remark, we deduce $f k n> \sum_{x_i\in D} i$. This means that $D$ uniquely determines a partition of a number smaller than $f k n$. Moreover, distinct occurring \smc s in $\MS_n$ determine distinct subsets $D$ of $\{x_0, x_1, \ldots\}$, and therefore distinct partitions. By the \HRT\ (\Tr{HRthm}), the number of such partitions is less than $r^{\sqrt{n}}$ for some $r>0$. Thus less than $r^{\sqrt{n}}$ elements of $\MS_n$ can occur simultaneously in \oo.
\end{proof}

If $C(o)$ is finite, then there is exactly one \scv\ that occurs and is contained in $C(o)$, namely the boundary of the unbounded face of $C(o)$.
We denote the probability of this event by $P_S$, that is, we set 
$$P_S(p):= \Pr(\text{$S$ occurs and } S\subset C(o)).$$
Thus we can write the probability $\theta_o(p)$ that $C(o)$ is finite by summing $P_S$ over all $S\in \CS$, where $\CS$ denotes the set of \scv s:
\labtequ{thetaS2}{$1-\theta_o(p)= \sum_{S\in \CS} P_S(p)$}
\fe\  $p\in (p_c,1]$.

As usual, our strategy to prove the analyticity of $\theta$, is to express $\theta$ as an infinite sum of functions that admit analytic extensions, namely, probabilities of events that depend on finitely many edges, and then apply \Cr{cor general}. 
Formula \eqref{thetaS2} is a first step in this direction, however, the functions $P_S$ are not fit for our purpose: the event $\{\text{$S$ occurs and } S\subset C(o) \}$ is not measurable with respect to the set of edges incident with $S$ only. Therefore, we would prefer to express  $\theta$ in terms of the simpler functions 
$$Q_S:= \Pr_p(\text{$S$ occurs}).$$
These functions have the advantage that comply with the premise of \Cr{D and not F NN}, and hence $|Q_S(p)|$ is bounded in $D(p,M)$ by $e^{C_{M,p}|S|}Q_S(p+M)$, where $C_{M,p}$ is independent of $S$. 
But when trying to write $\theta$ as a sum involving these $Q_S$, we have to be more careful: we have $$1-\theta_o(p)=\Pr_p(|C(o)|<\infty)=\Pr_p(\text{at least one $S\in \CS$ occurs})$$ by the definitions, but more than one $S\in \CS$ might occur simultaneously. Therefore, we will apply the inclusion-exclusion principle to the set of events $\{\text{$S$ occurs}\}_{S\in \CS}$. We claim that
%
\labtequ{thetaQ}{$1-\theta_o(p)= \sum_{M\in \MS} (-1)^{c(M)+1} Q_M(p)$}
\fe\  $p\in (p_c,1]$, where $c(M)$ denotes the number of \scv s in the \smc\ $M$. 

To prove this, we need first of all to check that the sum in the right hand side converges. This is implied by \Lr{exp dec Z2} below, which states that\\ $\sum_{M\in \MS_n} Q_M(p)$ decays exponentially in $n$, and therefore our sum converges absolutely. Then, we need to check that this sum agrees with the inclusion-exclusion formula. This is so because, \fe\ set $I$ of \scv s, we have $\Pr(\text{every $S\in I$ occurs})=0$ unless the elements of $I$ are pairwise vertex-disjoint ---that is, $I \in \MS$--- by \Lr{scs disjoint} and so we can restrict the inclusion-exclusion formula to $\MS$ rather than consider sets of \scv s that intersect.

The main part of our proof is to show that the probability for at least one \smc\ in $\MS_n$ to occur  decays exponentially in $n$, which will imply the following lemma. The rest of the arguments used to prove \Tr{theta analytic 2D} are identical to those of e.g.\ \Tr{chi NN} or \ref{thm nonam}.
\begin{lemma} \label{exp dec Z2}
\Fe\ 	$p\in (p_c,1]$ there are constants $c_1=c_1(p)$ and $c_2=c_2(p)$ with $c_2<1$, \st\  \fe\ $\nin$,
\labtequ{dec Q}{$\sum_{M\in \MS_n} Q_M(p) \leq c_1 {c_2}^{n}.$} 
Moreover, if $[a,b]\subset (p_c,1]$, then the constants $c_1$ and $c_2$ can be chosen independent of $p$ in such a way that \eqref{dec Q} holds for every $p\in [a,b]$.
\end{lemma}

The proof of this is based on the fact that the size of the boundary of an \scv\ $S$ that contains a certain vertex $x$ has an exponential tail. This is because $\partial S$ is contained in a component of the dual $L^*$ by \Lr{dual connd}, and as our percolation is subcritical on $L^*$,  the \ABP\ holds. Still, the exponential tail of each $|\partial S|$ does not easily imply \Lr{exp dec Z2}. First of all, the sum in the left hand side of \Lr{exp dec Z2} is larger than the probability $\Pr(\MS_n \text{ occurs})$ that a \smc\ of $\MS_n$ occurs. Second, a \smc\ might consist of plenty of \scv s. Nevertheless, we will be able to overcome these difficulties. Using \Lr{HR bound} we prove that the aforementioned sum does not grow too fast when compared with the probability that a \smc\ of $\MS_n$ occurs. 

\begin{proof}[Proof of \Lr{exp dec Z2}]
We start by noticing that $$\sum_{M\in \MS_n} Q_M(p)=\mathbb{E}_p(\sum_{M\in \MS_n}\mathbb{\chi}_{\{M \text{ occurs}\}}),$$
where $\mathbb{\chi}_{A}$ denotes the characteristic function of 
the occurence of an event $A$. The number of \smc s\ $M\in \MS_n$ that can occur simultaneously is bounded above by $r^{\sqrt{n}}$ for some $r>0$ by \Lr{HR bound}. It follows that $$\sum_{M\in \MS_n}\mathbb{\chi}_{\{M \text{ occurs}\}}\leq r^{\sqrt{n}} \mathbb{\chi}_{\{\MS_n \text{ occurs}\}}$$ which in turn implies that $$\sum_{M\in \MS_n} Q_M(p)\leq r^{\sqrt{n}} \Pr_p(\MS_n \text{ occurs}).$$ 

Hence it suffices to show that $\Pr_p(\MS_n \text{ occurs})$ decays exponentially in $n$.
In order to do so we will employ the exponential tail of the size of a certain (subcritical) cluster in the dual $L^*$ given by the \ABP. 
For this we will use the natural coupling of the percolation processes on $L$ and $L^*$: given a percolation instance $\oo\in 2^{E(L)}$ on $L$, we obtain a percolation instance $\oo^*$ on $L^*$ by changing the state of each edge, i.e.\ letting $\oo^*(e^*) = 1-\oo(e)$ \fe\ $e\in E(L)$.
Let $C(k)$ denote the event that there is a connected subgraph of $\oo^*$ which crosses one of the first $fk$ edges of $X^+$, and has at least $k$ edges, where $f$ is the constant of \autoref{quasi geo}. \mymargin{I changed the definition of C(k)} Note that $C(k)$ is an increasing event for $\omega^*$ that depends on finitely many edges. We claim that 
\labtequ{BKclaim}{$\Pr_p(\MS_n \text{ occurs})\leq \sum_{\{m_1,\ldots,m_k\}\in P_n} \Pr_{1-p}(C(m_1)\circ \ldots\circ C(m_k)),$}
where $\circ$ means that the events occur edge--disjointly (see \Sr{secBK}). Here $P_n$ is the set of partitions $\{m_1,\ldots,m_k\}$ of $n$. Once this claim is established, we will be able to employ the BK inequality (\Tr{BKthm}) to bound $\Pr_p(\MS_n \text{ occurs})$.

To prove \eqref{BKclaim}, we remark that each \smc\ $M\in \MS_n$ defines a partition $\{m_1,\ldots,m_k\}$ of $n$ by letting $m_i$ stand for the number of edges in the $i$-th component $K_i$ of the subgraph of $L^*$ spanned by $\partial M^*$. By \autoref{quasi geo} if $M$ occurs, then $K_i$ is a witness of $C(m_i)$, and these witnesses are pairwise edge-disjoint. Thus the occurrence of $M$ implies the occurrence of the event $C(m_1)\circ \ldots\circ C(m_k)$ in $\oo^*$. To conclude that \eqref{BKclaim} holds, we apply the union bound to the family of events of the latter form, ranging over all partitions $\{m_1,\ldots,m_k\}\in P_n$.


The BK inequality \cite{Grimmett} states that $$\Pr_{1-p}(C(m_1)\circ\ldots\circ C(m_k))\leq \Pr_{1-p}(C(m_1)) \cdot \ldots \cdot \Pr_{1-p}(C(m_k)).$$
Using the union bound and applying the \ABP, we obtain $\Pr_{1-p}(C(m_i)) \leq f m_i c^{m_i}$ for some constant $0<c=c(p)<1$. In addition, if $[a,b]\subset (p_c,1]$, then 
the monotonicity of $\Pr_{1-p}(C(m_i))$ implies that the constant $c$ can be chosen uniformly for $p\in [a,b]$. As $\Pr_{1-p}(C(n))<1$ for every $n$, we deduce that $\Pr_{1-p}(C(m_i))\leq (c+\varepsilon)^{m_i}$ for some $\varepsilon>0$ such that $c+\varepsilon<1$; indeed, for any $\varepsilon$, this is satisfied for large enough $m_i$, and increasing $\varepsilon$ we can make it true for the smaller values of $m_i$. 
 
Combining all these inequalities starting with \eqref{BKclaim} we conclude that 
$$\Pr_p(\MS_n \text{ occurs})\leq |P_n|(c+\varepsilon)^{n}.$$
We have $|P_n| \leq h^{\sqrt{n}}$ for some constant $h$ by the \HRT\ (\Tr{HRthm}), and so 
$$\Pr_p(\MS_n \text{ occurs})\leq h^{\sqrt{n}}(c+\varepsilon)^n.$$
Thus $\Pr_p(\MS_n \text{ occurs})$ decays exponentially in $n$ as claimed.
\end{proof}

We are now ready to prove Theorem \ref{theta analytic 2D}. 

\begin{proof}[Proof of Theorem \ref{theta analytic 2D}]
As already explained, \Lr{exp dec Z2} implies that the inclusion--exclusion expression \eqref{thetaQ} holds. The assertion follows if we can apply Corollary \ref{cor general} for 
$I = (p_c,1]$, $L_n=\MS_n$, and $(E_{n,i})$ an enumeration of the events $\{M \text{ occurs}\}_{M\in \MS_n}$. 
So let us check that the assumptions of
Corollary \ref{cor general} are satisfied.

By definition, every $M\in \MS_n$ has $n$ vacant edges. Moreover, $|M|\leq k(L) n$ by \Lr{scs boundary}. 
Thus assumption \ref{cg i} of Corollary \ref{cor general} is satisfied. The fact that assumption \ref{cg ii} is satisfied is exactly the statement of \Lr{exp dec Z2}.
\end{proof}

\subsection{Site percolation}

We will now extend our results to site percolation on \qtl s. Whereas duality was key to bond percolation on \qtl s, the corresponding property for site percolation is that of a \defi{matching}. The matching version $L'$ of a \qtl\ $L$ is defined by adding all diagonals to all faces of $L$. Note that $L'$ is a quasi-transitive graph that shares the same vertex sets as $L$. Moreover, 
\labtequ{vertex cut}{for every finite subgraph $H$ of $L$, the minimal vertex cut in $L\setminus H$ separating $H$ from infinity forms a cycle in $L'$.}

Analogously to the case of bond percolation, we have that 
\begin{equation}\label{matching}
\pcs(L)+\pcs(L')=1.
\end{equation} 
Although we could not find a reference for \eqref{matching} in the generality that we need it, we believe that it is well-known to the experts. The result is a consequence of \eqref{vertex cut}, the \ABP\, the uniqueness of the infinite cluster and the following non-coexistence result.
Couple the percolation processes on $L$ and $L'$: given a percolation instance $\omega\in 2^{V(L)}$ on $L$, we define a percolation instance $\omega'$ on $L'$ by letting $\omega'(v)=1-\omega(v)$ for every $v\in V(L)$. Then for every $p\in [0,1]$, $\Pr_p$-almost surely, one of $\omega$ or $\omega'$ does not contain a unique infinite cluster. The analogous statement for bond percolation has been proved in \cite{SheRan,DCRaTaSha}, and one can check that the proof in \cite{DCRaTaSha} generalizes to the setting of site percolation.

\begin{corollary}\label{triangular}
For Bernoulli site percolation on a \qtl\ $L$, the percolation density $\theta(p)$ is analytic for $p\in (\pcs(L),1]$. 
\end{corollary}
\begin{proof}
The proof is similar to that of \Tr{theta analytic 2D}. The only difference is that instead of the coupling with percolation on the dual $L^*$ that we used there, which we combined with \Tr{pc star} to deduce the exponential decay of the probability that a fixed vertex lies on an \scv\ of length $n$, we now obtain this exponential decay by the coupling with percolation on the matching version $L'$ mentioned above. Indeed, notice that if an \scv\ $S$ occurs in $\omega$, then the vertices incident with $\partial S$ that do not lie in $S$ form a connected occupied subgraph in $\omega'$. But occupied subgraphs in $\omega'$ are subcritical when $p>\pcs(L)=1-\pcs(L')$, and so applying the \ABP\ to them yields the desired exponential decay.
\end{proof}

\section{Analyticity of $\theta$ in all dimensions}\label{theta}

In this section, we will prove that for percolation on $\mathbb{Z}^d , d\geq 2$, the percolation density $\theta$ is analytic on the supercritical interval. The case $d=2$ has already been handled in \Sr{sec th pl}, and although the current section can in principle be read independently, we recommend reading \Sr{sec th pl} first as a warm-up.

The \defi{cubic lattice} $\mathbb{L}^d$ is the standard \Cg\ of  $\mathbb{Z}^d$. In other words, we put an edge between two points $x,y\in \R^d$ with integer coordinates whenever $|x-y|=1$.

\begin{theorem}\label{analytic}
For Bernoulli bond percolation on $\mathbb{L}^d , d\geq 2$, the percolation density $\theta(p)$ is analytic on $(p_c,1]$. 
\end{theorem} 

As we will see in \Sr{sec fp}, our notion of interfaces can be generalised to all dimensions $d\geq 3$, and the method developed in \Sr{sec th pl} still yields the analyticity of $\theta$ for the values of $p$ close to $1$. However, several challenges arise when one tries to extend this result to the whole supercritical interval. The main obstacle is that for values of $p$ in the interval $(p_c,1-p_c)$, the distribution of the size of the interface of $C_o$ has only a stretched exponential tail, which follows from the work of Kesten and Zhang \cite{KeZhaPro}. In order to overcome this obstacle, we will employ 
renormalisation techniques 
similar to those of \cite{KeZhaPro}. Here we use the more refined version of Pete \cite{Pete}.

\subsection{Setting up the renormalisation} \label{sec boxes}

We start by introducing some necessary definitions. Consider a positive integer $N$. For every vertex $x$ of $\Z^d$, we let $B(x)=B(x,N)$ denote the box $\{y\in \Z^d: \lVert y-Nx \rVert_{\infty}\leq 3N/4\}$. With a slight abuse, we will use the same notation $B(x)$ to also denote the corresponding subset of $\R^d$, namely\\ $\{y\in \R^d: \lVert y-Nx \rVert_{\infty}\leq 3N/4\}$.

The collection of all these boxes can be thought of as the vertex set of graph canonically isomorphic to $\Z^d$. We will denote this graph by $N\mathbb{L}^d$. Whenever we talk about  percolation (clusters) from now on, we will be referring to percolation, with a fixed parameter $p>p_c$, on $\mathbb{L}^d$ and not on $N\mathbb{L}^d$; we will never percolate the latter. 

For any  percolation cluster $C$, 
we denote by $C(N)$ the set of boxes $B$ such that the subgraph of $C$ induced by its vertices lying in $B$ has a component of diameter at least $N/5$. The boxes with this property will be called \defi{C-substantial}. Notice that $C(N)$ is a connected subgraph of $N\mathbb{L}^d$.
The internal boundary of $C(N)$ is denoted by $\partial C(N)$ following the terminology of \Sr{sec GT}. Notice that $\partial C(N)$ is not necessarily connected. For technical reasons, we would like it to be, and therefore we modify our lattice by adding the diagonals: we introduce a new graph $N\mathbb{L}^d_\sboxt$, the vertices of which are the boxes $B(x), x\in \Z^d$, and we connect two boxes with an edge of $N\mathbb{L}^d_\sboxt$ whenever they have non-empty intersection. When $N=1$, the vertex set of $\mathbb{L}^d_\sboxt$ is simply $\Z^d$. It is not too hard to show (see  \cite[Theorem 5.1]{TimCut}) that 
\labtequ{timar}{If $C$ is finite then $\partial C(N)$ is a connected subgraph of $N\mathbb{L}^d_\sboxt$.}

Given two diagonally opposite neighbours $x$, $y$ of $\mathbb{L}^d$, we will write $B(x,y)$ for the intersection $B(x) \cap B(y)$. A percolation cluster $C$ is a \defi{crossing cluster}  for some box $B(x)$ or $B(x,y)$, if $C$ contains a vertex from each of the $(d-1)$-dimensional faces of that box. We say that a box $B(x)$ is \defi{good} in a percolation configuration $\omega$ if it has a crossing cluster $C$ with the property that the intersection of $C$ with each of the boxes $B(x,y)$ contains a crossing cluster (of $B(x,y)$), and every other cluster of $B(x)$ has diameter less than $N/5$. A box that is not good will be called \defi{bad}. It is known \cite[Theorem 7.61]{Grimmett} that, for every $p>p_c$, the probability of having a crossing cluster and no other cluster of diameter greater than $N/5$ converges to $1$ as $N\to \infty$. Combining this with a union bound we easily deduce that 
\labtequ{good to 1}{for every $p>p_c$, the probability of any box being good converges to $1$ as $N\to \infty$.} 
We will say that a set of boxes is bad if all its boxes are bad.

Our definition of good boxes is slightly different than the standard one in that it asks for all boxes $B(x,y)$ to contain a crossing cluster. The reason for imposing  this additional property is because now 
\labtequ{star}{every $N\mathbb{L}^d_\sboxt$-component $B$ of good boxes contains a unique percolation cluster $C$ such that some box of $B$ is $C$-substantial (and in fact all boxes of $B$ are $C$-substantial).}  
This follows easily once we notice that this holds for pairs of neighbouring boxes.

Observe that the boxes in $\partial C(N)$ are never good. Indeed, if some box $B\in \partial C(N)$ is good, then $C$ connects all the $(d-1)$-dimensional faces of $B$, hence all $N\mathbb{L}^d$-neighbouring boxes of $B$ contain a connected subgraph of $C$ of diameter at least $N/5$, and so they lie in $C(N)$. This contradicts the fact that $B$ belongs to $\partial C(N)$. 

Having introduced the above definitions, our aim now is to find a suitable expression for $1-\theta$ in terms of good and bad boxes surrounding $o$. 


With the above definitions we have that, conditioning on the event that $C_o$ is finite and has diameter at least $N/5$, there is a non-empty $N\Ls$-connected subgraph of bad boxes that separates $o$ from infinity, namely $T:=\partial C_o(N)$. However, the event $\{|C_o|<\infty\}$ is not necessarily measurable with respect to the configuration inside $T$. In other words, we cannot express $1-\theta$ in terms of just the configuration inside $T$, and instead we have to explore the configuration inside the finite components surrounded by $T$. To this end, we will expand $\partial C_o(N)$ into a larger object. 

\subsection{Separating components} \label{sec SCs}

A \defi{separating component} is a $N\mathbb{L}^d_\sboxt$-connected set $S$ of boxes, such that $o$ lies either inside $S$ or in a finite component of $N\mathbb{L}^d_\sboxt\setminus S$. We will write $\partial_\sboxt S$ for its vertex boundary ---defined in \Sr{sec GT}--- when viewed as a subgraph of $N\mathbb{L}^d_\sboxt$. We say that $S$ \defi{occurs} in a configuration $\omega$ if all the following hold:
\begin{enumerate}
\item \label{occ i} all boxes in $S$ are bad; 
\item \label{occ ii} all boxes in $\partial_\sboxt S$ are good, and
\item \label{occ iii} there is a configuration $\omega'$ which coincides with $\omega$ in $S\cup \partial_\sboxt S$, such that $C_o(\omega')$ is finite, and $S$ contains $\partial C_o(\omega')(N)$.
\end{enumerate}
We will say that $\omega'$ is a \defi{witness} for the occurrence of $S$ if \ref{occ i}--\ref{occ iii} all hold. 

One way to interpret \ref{occ iii} is that there exists a minimal cut set $F$ surrounding $o$ with the property that all its edges inside $S\cup \partial_\sboxt S$ are closed in $\omega$. If there is an infinite path in $\omega$ starting from $o$, then it has to avoid the edges of $F$ lying in $S\cup \partial_\sboxt S$. As we will see, \ref{occ ii} makes this impossible without violating that $C_o(\omega')$ is finite.

Note that \ref{occ iii} implies that 
\labtequ{S omega}{$\partial^V C_o(\omega')$ (and $C_o(\omega')$) does not share a vertex with the infinite component of $\mathbb{L}^d \setminus  S$.}

\subsection{Expressing $\theta$ in terms of the probability of the occurrence of a separating component} \label{theta sum S}

In this section we show that $C_o$ is finite exactly when some separating component occurs, unless $diam(C_o)< N/5$ which is a case that is easy to deal with. This will allow us to express $\theta(p)$ in terms of the probability of the occurrence of a separating component (see \eqref{cases}). In the following section we will expand the latter as a sum (with inclusion-exclusion) over all possible separating components. The summands of this sum are well-behaved polynomials, that will allow us to apply Corollary \ref{cor general} to deduce the analyticity of $\theta(p)$.

\begin{lemma} \label{an S occurs}
For every $p>p_c$ there is $N\in \N$ and an interval $(a,b)$ containing $p$ such that the following holds for every $q\in (a,b)\cap (p_c,1]$. Conditioning on $C_o$ being finite, and $diam(C_o)\geq N/5$, at least one separating component occurs almost surely.
\end{lemma}
\begin{proof}
Let $S$ be the maximal connected subgraph of $N\mathbb{L}^d_\sboxt$ that contains $\partial C_o(N)$ and consists of bad boxes only. This $S$ exists whenever $C_o$ is finite and $diam(C_o)\geq N/5$ because $\partial C_o(N)$ is connected by \eqref{timar}.

We claim that there is some $N$ and an interval $(a,b)$ containing $p$ such that $S$ is $\Pr_q$-almost surely finite for every $q\in (a,b)\cap (p_c,1]$. For this, it suffices to show that for some large enough $N$, the probability $\Pr_q(\text{S has size at least $n$})$ converges to $0$ as $n$ tends to infinity for each such $q$. The latter follows by combining the union bound with \Lr{exp dec} below, which states that 
$$\sum_{T \text{ is a separating component of size } n} \Pr_q(\text{T is bad})\leq e^{-tn}$$ 
for some constant $t=t(p)>0$, for some $N$, and every $q$ in an interval $(a,b)\cap(p_c,1]$.

Note that conditions \ref{occ i} and \ref{occ ii} are automatically satisfied by the choice of $S$. The configuration $\omega':=\omega$ satisfies condition \ref{occ iii}, since $C_o(\omega)$ is finite, and $S$ contains $\partial C_o(\omega)(N)$ by definition. Thus $S$ occurs in $\omega$, as desired.
\end{proof}




Note that the proof of \Lr{an S occurs} finds a concrete occurring separating component whenever $C_o$ is finite and $diam(C_o)\geq N/5$; we denote this separating component by $\mathcal{S}_o$ in this case.

\medskip
The next two lemmas provide a converse to \Lr{an S occurs}, namely that $C_o$ is finite whenever 
some separating component occurs. 

Whenever $\omega'$ is a witness for the occurrence of $S$, we let $R_o(\omega')$ denote the set of vertices of the infinite component of $\mathbb{L}^d\setminus C_o(\omega')$ lying in $S$. 

\begin{lemma}\label{finite}
Consider a separating component $S$, and assume that $S$ occurs in $\omega$. Let $\omega'$ be a witness of the occurrence of $S$. Then no vertex of $R_o(\omega')$ lies in $C_o(\omega)$.
\end{lemma}
\begin{proof}
Assume that some vertex $u$ of $R_o(\omega')$ lies in $C_o(\omega)$; we will obtain a contradiction.  

Since $C_o(\omega)$ contains $u$, there must exist a path $P$ in $\omega$ connecting $o$ to $u$. This path cannot lie entirely in $S\cup\partial_\sboxt S$ because $\omega$ and $\omega'$ coincide in that set of boxes and $u\not\in C_o(\omega')$. Hence $N\Ls\setminus (S\cup\partial_\sboxt S)$ must have some finite component. Let $E$ denote the minimal edge cut of $C_o(\omega')$. Clearly, $P$ must intersect $E$, since $u$ lies in the infinite component of $\mathbb{L}^d\setminus C_o(\omega')$. Let $e$ be an edge of $E$ that $P$ contains. Notice that no common edge of $P$ and $E$ lies in $S\cup \partial_\sboxt S$, because the edges of $E$ are closed in $\omega'$, the edges of $P$ are open in $\omega$, and the two configurations coincide in $S\cup \partial_\sboxt S$. Hence $e$ must lie in one of the finite components $\mathcal{B}_{in}$ of $N\Ls\setminus (S\cup\partial_\sboxt S)$. Write $\mathcal{B}$ for the set of those boxes in $\partial_\sboxt S$ that have a $N\mathbb{L}^d_\sboxt$-neighbour in $\mathcal{B}_{in}$. (Thus $\mathcal{B}$ is the vertex boundary of $\mathcal{B}_{in}$.) See Figure \ref{drawing-box}.

\begin{figure}[!ht] 
   \centering
   \noindent

\begin{overpic}[width=.6\linewidth]{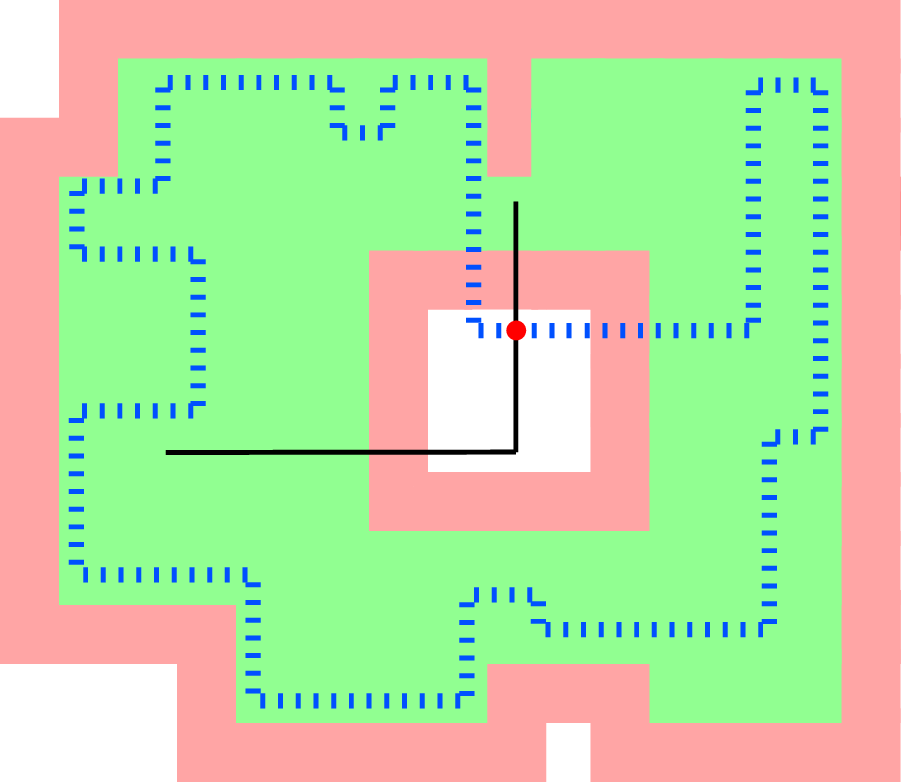} 
\put(23,36){$o$}
\put(62,65){$u$}
\put(32,73){$E$}
\put(36,43){$P$}
\end{overpic}
\caption{The situation in the proof of \Lr{finite}. The separating component $S$ is depicted in green and its boundary $\partial_\sboxt S$ in red (if colour is shown). When two boxes of $S$ and $\partial_\sboxt S$ overlap, their intersection is depicted also in green. The dashes depict the edges of the cut $E$, and $e$ is highlighted with a (red) dot.}\label{drawing-box}
\end{figure}

It is not hard to see that some box $B$ of $\mathcal{B}$ is $C_o(\omega')$-substantial, which then implies that all boxes of $\mathcal{B}$ are $C_o(\omega')$-substantial because they are all good. Indeed, notice that one of the two endvertices of $e$ lies in $C_o(\omega')$ by the definition of the set $E$. As $S$ contains a  $C_o(\omega')$-substantial box, some box $B$ of $\mathcal{B}$ must be $C_o(\omega')$-substantial, as claimed, because $\mathcal{B}$ is the vertex boundary of $\mathcal{B}_{in}$. 

Our aim now is to show that we can connect $u$ to the subgraph of $C_o(\omega')$ inside $\mathcal{B}$ with a path in $\omega'$ lying entirely in $S\cup\partial_\sboxt S$. This will imply that $u$ belongs to $C_o(\omega')$, contradicting that $u\in R_o(\omega')$.

For this, consider the subpath $Q$ of $P$ that starts at $u$ and ends at the last vertex of the intersection of $\mathcal{B}_{in}$ and $\mathcal{B}$ (notice that although $\mathcal{B}_{in}$ and $\mathcal{B}$ are disjoint sets of boxes, the subgraphs of $\mathbb{L}^d$ inside them overlap). If $Q$ is not contained in $S\cup\partial_\sboxt S$, then we can modify it to ensure that it does lie entirely in $S\cup\partial_\sboxt S$. Indeed, notice that each $N\mathbb{L}^d_\sboxt$-component $F$ of $\partial_\sboxt S$ contains a unique $\omega$-cluster $C$ such that some box of $F$ is $C$-substantial by \eqref{star}, because all its boxes are good. Moreover, each time $Q$ exits $S\cup \partial_\sboxt S$, it has to first visit the unique such percolation cluster of some $N\mathbb{L}^d_\sboxt$-component $F$ of $\partial_\sboxt S$, and then eventually revisit the same percolation cluster of $F$. We can thus replace the subpaths of $Q$ that lie outside of $S\cup\partial_\sboxt S$ by open paths lying entirely in $\partial_\sboxt S$ that share the same endvertices. Thus we may assume that $Q$ is contained in $S\cup\partial_\sboxt S$ as claimed.

Now notice that $Q$ contains a subpath of diameter greater than $N/5$ lying entirely in some box $B$ of $\mathcal{B}$. This box is $C_o(\omega')$-substantial, hence $C_o(\omega')$ and $Q$ must meet. Then following the edges of $Q$, which are all open in $\omega'$, we arrive at $u$, and thus $u$ belongs to $C_o(\omega')$, as desired. 
\end{proof}

We now use this to prove

\begin{lemma}\label{Co finite}
Whenever some separating component occurs in a configuration $\omega$, the cluster $C_o(\omega)$ is finite.
\end{lemma}
\begin{proof}
We will prove the following slightly stronger statement: whenever a separating component $S$ occurs in a configuration $\omega$, a minimal (finite) edge cut of closed edges occurs in $\omega$ which separates $o$ from infinity and lies in $S\cup \partial_\sboxt S$. 

For this, consider a witness $\omega'$ of the occurrence of $S$, and let $\omega''$ be the configuration which coincides with $\omega$ (and $\omega'$) on every edge lying in $S\cup\partial_\sboxt S$, and every other edge of $\omega''$ is open.  Note that $S$ occurs in $\omega''$ since it occurs in $\omega$. Thus $C_o(\omega'')$ contains no vertex of $R_o(\omega')$ by \Lr{finite}. This implies that $C_o(\omega'')$ contains no vertex in the infinite component $X$ of $N\Ls\setminus S$, because any path $P$ in $\mathbb{L}$ connecting $o$ to $X$ has to first visit $R_o(\omega')$. To see that the latter statement is true, consider the last vertex $u$ of $\partial^V C_o(\omega')$ that $P$ contains. Notice that the subpath of $P$ after $u$, which is denoted $Q$, visits only vertices of the infinite component of $\mathbb{L}^d\setminus C_o(\omega')$, and furthermore that $u$ lies either in $S$ or in a finite component of $\mathbb{L}^d\setminus S$ by \eqref{S omega}. In the first case, $u$ lies in $R_o(\omega')$. In the second case, $Q$ has to visit $S$, hence $R_o(\omega')$.

We have just proved that $C_o(\omega'')$ can only contain vertices in $S$ and the finite components of $N \Ls\setminus S$. Since $S$ is a finite set of boxes, $C_o(\omega'')$ is finite as well. Hence a minimal edge cut of closed edges separating $o$ from infinity occurs in $\omega''$. This minimal edge cut must lie entirely in $S\cup \partial_\sboxt S$, because all edges not in $S\cup \partial_\sboxt S$ are open. This is the desired minimal edge cut since it occurs in $\omega$ as well. We will 
denote it by $\partial^b \mathcal{S}_o$.
\end{proof}

Lemmas~\ref{an S occurs} and~\ref{Co finite} combined allow us to express the event that $C_o$ is finite in terms of the event that some separating component occurs. To do so, let us write $D_N$ to denote the event $\{diam(C_o)<N/5\}$.
Thus we have proved that  
\begin{align}\label{cases}
\begin{split}
1-\theta(p)=&\Pr_p(C_o \text{ is finite})\\ = &\Pr_p(D_N)+\Pr_p(|C_o|<\infty, D^\mathsf{c}_N)\\ =&\Pr_p(D_N)+\Pr_p(\text{some separating component occurs},D^\mathsf{c}_N).
\end{split}
\end{align}
Here and below, the notation $X,Y,\ldots$ denotes the intersection of the events $X,Y,\ldots$. 

\subsection{Expanding $\theta$ as an infinite sum of polynomials} \label{expand}

Notice that $\Pr_p(D_N)$ is a polynomial in $p$, since the event $D_N$ depends only on the state of finitely many edges. 

Following our technique from \Sr{sec th pl}, we will now use the inclusion-exclusion principle to expand the right-hand side of \eqref{cases} as an infinite sum of polynomials, corresponding to all possible separating components that could occur.

Notice that any two occurring separating components are disjoint because they are connected, their boxes are bad, and they are surrounded by good boxes by definition.

\begin{lemma}\label{inc-ex lemma}
For every $p>p_c$ there is some integer $N=N(p)>0$ and an interval $(a,b)$ containing $p$ such that the expansion
\begin{align}\label{inc-ex}
\Pr_q(\text{some S occurs},D^\mathsf{c}_N)=\sum_{S\in MS^N} (-1)^{c(S)+1} \Pr_q(\text{S occurs},D^\mathsf{c}_N)
\end{align}
holds for every $q\in (a,b)\cap(p_c,1]$, where $MS^N$ denotes the set of all finite collections of pairwise disjoint separating components $S$, and $c(S)$ denotes the number of separating components of $S$. 
\end{lemma}

\Lr{inc-ex lemma} follows easily from the next lemma. We will use the notation $MS^N_n$ to denote the set of those finite collections of pairwise disjoint separating components $\{S_1,S_2,\ldots,S_k\}$ such that $|S_1|+|S_2|+\ldots+|S_k|=n$. The superscript reminds us of the dependence of the boxes on $N$. 

\begin{lemma}\label{exp dec}
For every $p>p_c$, there are $N=N(p)>0$, $t=t(p)>0$ and an interval $(a,b)$ containing $p$ such that 
\begin{align}\label{ineq}
\sum_{S\in MS^N_n} \Pr_q(\text{S is bad})\leq e^{-tn}
\end{align}
for every $n\geq 1$ and every $q\in (a,b)\cap (p_c,1]$. 
\end{lemma}
\begin{proof}
To prove the desired exponential decay we will use a standard renormalisation technique  with a few modifications. We will first prove the exponential decay when $q=p$, and then we will use a continuity argument to obtain the desired assertion.

We will first show that there exists a constant $k>0$ depending only on $d$ such that for every $S\in MS^N_n$ we have $$\Pr_p(\text{S is bad})\leq c^{n/k},$$ where $c:=\Pr_p(B(o) \text{ is bad})$. Indeed, it is not hard to see that there is a constant $k=k(d)>0$ such that for every $S\in MS^N_n$ there is a subset $Y$ of $S$ of size at least $n/k$, all boxes of which are pairwise disjoint. As each box of $Y$ is bad whenever $S$ occurs, we have $$\Pr_p(\text{S is bad})\leq \Pr_p(\text{Y is bad}).$$
By independence $\Pr_p(\text{Y is bad})=c^{n/k}$ and the assertion follows.

We will now find an exponential upper bound for the number of elements of $S\in MS^N_n$. Since $N\Ls$ is isomorphic to $\Ls$, there is a constant $\mu>0$ depending only on $d$ and not on $N$, such that the number of connected subgraphs of $N\Ls$ with $n$ vertices containing a given vertex is at most $\mu^n$. However, an element of $MS^N_n$ might contain multiple separating components, and there are in general several possibilities for the reference vertices that each of them contains. To remedy this, consider one of the $d$ axis $X=(-x_1,x_0=B(o),x_1)$ of $N\Ls$ that contain the box $B(o)$, and let $X^+$, $X^-$ be its two infinite subpaths starting from $B(o)$. We will first show that any separating component of size $n$ contains one of the first $n$ elements of $X^+$. Indeed, consider an occurring separating component $S$ of size $n$, and notice that $S$ has to contain some vertex $x^+$ of $X^+$, and some vertex $x^-$ of $X^-$. The graph distance between $x^+$ and $x^-$ is at most $n$, as there is a path in $S$ connecting them. This implies that $x^+$ is one of the first $n$ elements of $X^+$, as desired.

Consider now a constant $M>0$ such that $m\mu^m\leq M^m$ for every integer $m\geq 1$. Consider also a partition $\{m_1,m_2,\ldots,m_k\}$ of $n$. It follows that the number of collections $\{S_1,S_2,\ldots,S_k\}$ with $|S_i|=m_i$ is at most $m_1m_2\ldots m_k \mu^n\leq M^n$, since we have at most $m_i\mu^{m_i}$ choices for each $S_i$. Recall that the number of partitions of $n$ is at most $r^{\sqrt{n}}$ for some constant $r>0$ by \Tr{HRthm} (even an exponential bound would be good enough at this point). We can now deduce that the size of $MS^N_n$ is at most $r^{\sqrt{n}}M^n$, implying that 
$$\sum_{S\in MS^N_n} \Pr_p(\text{S is bad})\leq r^{\sqrt{n}}M^n c^{n/k}.$$
Notice that in the right hand side of the above inequality only $c$ depends on $N$. It is a standard result that $c$ converges to $0$ as $N$ tends to infinity \cite[Theorem 7.61]{Grimmett}. Choosing $N$ large enough so that $Mc^{1/k}<1$, we obtain the desired exponential decay.

Now notice that $c(q)=\Pr_q(B(o) \text{ is bad})$ is a polynomial in $q$, hence a continuous function, since it depends only on the state of the edges inside $B(o)$. This implies that we can choose an interval $(a,b)$ containing $p$ such that $Mc(q)^{1/k}<1$ for every $q\in (a,b)\cap (p_c,1]$. This completes the proof.
\end{proof}

We are now ready to prove \Tr{analytic}.

\begin{proof}[Proof of \Tr{analytic}]
Consider some $p\in (p_c,1]$. Let $N,t>0$, and the interval $(a,b)$ containing $p$, be as in \Lr{exp dec}. Then the expression $$1-\theta(q)=\Pr_q(D_N)+\sum_{n=1}^\infty \sum_{S\in MS^N_n} (-1)^{c(S)+1} \Pr_q(\text{S occurs},D^\mathsf{c}_N)$$ holds for every $q\in (a,b)\cap (p_c,1]$, and furthermore $$\Bigl\lvert\sum_{S\in MS^N_n} (-1)^{c(S)+1} \Pr_q(\text{S occurs},D^\mathsf{c}_N)\Bigr\rvert\leq e^{-tn}$$ for every $q\in (a,b)\cap (p_c,1]$.
The probability $\Pr_q(D_N)$ is a polynomial in $q$, hence analytic, because it depends on finitely many edges. Moreover, the event $\{\text{S occurs},D^\mathsf{c}_N\}$ depends only on the state of the edges lying in $S\cup\partial_\sboxt S$ and the box $B(o,N)$. The number of edges of each box is $O(N^d)$, hence the event $\{\text{S occurs},D^\mathsf{c}_N\}$ depends only on $O(N^d n)$ edges. The desired assertion follows now from \Tr{cor general}.
\end{proof}

\subsection{Exponential tail of $\partial^b \mathcal{S}_o$} \label{Pete} 

\Lr{exp dec} easily implies that the size of $\partial^b \mathcal{S}_o$, as defined in the proof of \Lr{Co finite}, has an exponential tail:

\begin{theorem}\label{bad boundary}
For every $p>p_c$, there are constants $N=N(p)>0$ and $t=t(p)>0$ such that 
$$\Pr_p(|\partial^b \mathcal{S}_o|\geq n)\leq e^{-tn}$$ for every $n\geq 1$.
\end{theorem}
\begin{proof}
Assume that $|\partial^b \mathcal{S}_o|\geq n$, and consider the separating component $S$ associated to $C_o$. Then the boxes of $S\cup \partial_\sboxt S$ must contain at least $n$ edges. Hence the number of boxes of $S\cup \partial_\sboxt S$ is at least $cn/N^d$ for some constant $c>0$. Moreover, we have $|\partial_\sboxt S|\leq (3^d-1)|S|$, because each box of $\partial_\sboxt S$ has at least one neighbour in $S$, and each box in $S$ has at most $3^d-1$ neighbours. Therefore, $S$ contains at least $cn/(3N)^d$ boxes. The desired assertion follows from \Lr{exp dec}.
\end{proof}

We recall that for every $p\in (p_c,1-p_c]$, the probability $\Pr_p(|\partial C_o|\geq n)$ does not decay exponentially in $n$ \cite{KeZhaPro,ExpGrowth}. This implies that for those values of $p$, $\partial^b \mathcal{S}_o$ has typically smaller order of magnitude than the standard minimal edge cut of $C_o$.

\medskip
As a corollary, we re-obtain a result of Pete \cite{Pete} which states that when $C_o$ is finite, the number of touching edges between $C_o$ and the unique infinite cluster, which we denote $C_{\infty}$, has an exponential tail. A \defi{touching edge} is an edge in $\partial^E C_o\cap \partial^E C_{\infty}$. We denote the number of (closed) touching edges joining $C_o$ to the infinite component $C_{\infty}$ by $\phi(C_o,C_{\infty})$.

\begin{corollary}\label{touching}
For every $p>p_c$, there is some $c=c(p,d)>0$ such that
$$\Pr_p(|C_o|<\infty,\phi(C_o,C_{\infty})\geq t)\leq e^{-ct}$$
for every $t\geq 1$.
\end{corollary}
\begin{proof}
The result follows from \Tr{bad boundary} by observing that $C_{\infty}$ has to lie in the unbounded component of $\mathbb{L}^d\setminus \partial^b \mathcal{S}_o$, hence all relevant edges belong to $\partial^b \mathcal{S}_o$.
\end{proof}
 
\subsection{Analyticity of $\tau$}\label{section tau}

In the previous section we proved that $\theta$ is analytic above $p_c$. Some further challenges arise when one tries to prove that other functions describing the macroscopic behaviour of our model are analytic functions of $p$. The main obstacle is that events of the form $\{x \text{ is connected to }y\}$ are not fully determined, in general, by the configuration inside $S\cup \partial_\sboxt S$. In this section we show how one can remedy this issue, and we will prove that the $k$-point function $\tau$ and its truncated version $\tau^f$ are analytic functions above $p_c$ for every $d\geq 2$. We will then deduce that the truncated susceptibility $\mathbb{E}(|C_o|; |C_o|<\infty)$ and the free energy $\mathbb{E}(|C_o|^{-1})$ are analytic functions as well. Using similar arguments one can prove that the analogous statements hold also for percolation on \qtl s. A proof can be found in the preprint version of the current paper \cite{analyticity}.

Given a $k$-tuple $\vx=\{x_1,\ldots,x_k\}, k\geq 2$ of vertices of $\Z^d$, the function $\tau_{\vx}(p)$ denotes the probability that $\vx$ is contained in a cluster of Bernoulli percolation on $\Z^d$ with parameter $p$. Similarly, $\tau^f_{\vx}(p)$ denotes the probability that $\vx$ is contained in a {\it finite} cluster. We will write $MS^N(\vx)$ for the set of all finite collections of separating components surrounding some vertex of $\vx$, and $MS^N_n(\vx)$ for the set of those elements of $MS^N(\vx)$ that have size $n$.

Arguing as in the proof of \Lr{exp dec} we obtain the following: 
\begin{lemma}\label{exp dec x}
For every $p>p_c$, there are $N=N(p)>0$, $t=t(p)>0$, and an interval $(a,b)$ containing $p$, such that 
\begin{align}\label{ineq x}
\sum_{S\in MS^N_n} \Pr_q(\text{S occurs})\leq e^{-tn}
\end{align}
for every $n\geq 1$ and every $q\in (a,b)\cap (p_c,1]$. 
\end{lemma}

We are now ready to prove that $\tau$ and $\tau^f$ are analytic. 

\begin{theorem}\label{tau}
For every $d\geq 2$ and every finite set $\vx$ of vertices of $\Z^d$, the functions $\tau_{\vx}(p)$ and $\tau^f_{\vx}(p)$ admit analytic extensions to a domain of $\C$ that contains the interval $(p_c, 1]$.

Moreover, for every $p\in (p_c, 1]$ and every finite set $\vx$ such that $\mathrm{diam}(\vx)\geq N/5$, there is a closed disk $D(p,\delta),\delta>0$ and positive constants $c_1=c_1(p,\delta),c_2=c_2(p,\delta)$ such that $$|\tau^f_{\vx}(z)|\leq c_1 e^{-c_2{\mathrm{diam}(\vx)}}$$
\fe\ $z\in D(p,\delta)$ for such an analytic extension $\tau^f_{\vx}(z)$ of $\tau^f_{\vx}(p)$.
\end{theorem}
\begin{proof}
We start by showing that $\tau^f_{\vx}(p)$ is analytic. Suppose $\vx=\{x_1,\ldots,x_k\}$, and let $A$ be the event that $\mathrm{diam}(C_{x_i})\geq N/5$ for every $i\leq k$. We will write $\{\vx \text{ is connected}\}$ to denote the event that all vertices of $\vx$ belong to the same cluster, which we denote $C_{\vx}$. When $C_{\vx}$ is finite and both events $\{\vx \text{ is connected}\}$ and $A$ occur, we will write $\mathcal{S}_{\vx}$ for the separating component of the latter cluster, namely the $N\Ls$-component of $\partial C_{\vx}(N)$. The event $\{\text{S occurs}\}$ is defined as in the previous section except that now $C_o$ is replaced by $C_{\vx}$, i.e. the event $\{\vx \text{ is connected}\}$ occurs in a witness $\omega'$, and $S$ contains $\partial C_{\vx}(\omega')(N)$.  With the above definitions we have $$\tau^f_{\vx}(p)=\Pr_p(A^\mathsf{c}, \vx \text{ is connected})+\sum_S \Pr_p(A,\vx \text{ is connected}, \mathcal{S}_{\vx}=S),$$
where the sum ranges over all possible separating components separating all of $\vx$ from infinity. 

Our aim is to further decompose the events of the above expansion into simpler ones that we have better control of, and then use the inclusion-exclusion principle. We will first introduce some notation. Given a separating component $S$ as above, we first decompose $\vx$ into two sets $\vx_{out}$ and $\vx_{in}$, where $\vx_{out}$ denotes the set of those vertices of $\vx$ lying in some finite component of $N\Ls\setminus (S\cup \partial_\sboxt S)$, and $\vx_{in}:= \vx \backslash \vx_{out}$ its complement. We write $\{\vx \rightarrow S\}$ for the event that no separating component separating some $x_i\in \vx$ from $S$ occurs; to be more precise, the event $\{\vx \rightarrow S\}$ means  that for each $x_i\in \vx_{out}$, no separating component that surrounds $x_i$ and lies entirely in some of the finite components of $N\Ls\setminus (S\cup \partial_\sboxt S)$ occurs. 

Consider now some vertex $x$ in $\vx_{out}$, and let $F$ be the component of $\partial_\sboxt S$ that separates $x$ from $S$. We claim that when $S$ and the events $A$, $\{\vx \rightarrow S\}$ all occur, then $x$ is connected to the unique large cluster of $F$. In particular, if another vertex of $\vx$ lies in the same finite component of $N\Ls\setminus (S\cup \partial_\sboxt S)$ as $x$ does, then both vertices are connected to the unique large cluster of $F$, hence they are connected to each other. To see that the claim holds, notice that $C_{x}$ has to be finite, because $S\cup\partial_\sboxt S$ contains a minimal edge cut of closed edges that surrounds all vertices of $\vx$, hence $x$. Now $\partial C_{x}(N)$ has to intersect $S$, because it cannot lie entirely in $N\Ls\setminus (S\cup \partial_\sboxt S)$ by our assumption. This implies that $x$ is connected to some vertex inside $S$, hence it must first visit the unique large cluster of $F$, as desired.

We now define $\mathcal{C}$ to be the event that 
\begin{itemize}
\item all vertices of $\vx_{in}$ are connected to each other with open paths lying in $S\cup \partial_\sboxt S$,
\item the unique large percolation clusters of the components $F$ of $\partial_\sboxt S$ that separate some $x_i\in \vx_{out}$ from $S$ are connected to each other with open paths lying $S\cup \partial_\sboxt S$,
\item all vertices of $\vx_{in}$ are connected to all such percolation clusters with open paths lying in $S\cup \partial_\sboxt S$.
\end{itemize}
(It is possible that either $\vx_{in}$ or $\vx_{out}$ is the empty set, in which case the third item and one of the first two are empty statements.)
We claim that when $S$ and the events $A$, $\{\vx \rightarrow S\}$ and $\{\vx \text{ is connected}\}$ all occur, then the event $\mathcal{C}$ occurs as well. Indeed, consider a vertex $x\in\vx_{out}$, and let $F$ be the component of $\partial_\sboxt S$ that separates $x$ from $S$, as above. Any open path connecting $x$ to some vertex of $\vx_{out}$ which does not lie in the same finite component of $N\mathbb{L}^d_\sboxt$ that $x$ does, has to first visit the unique large percolation cluster of $F$. Hence it suffices to prove that when two vertices $x_i$ and $x_j$ of $\vx_{in}$ lie in the same cluster, there is always an open path connecting them lying entirely in $S\cup \partial_\sboxt S$.
To this end, assume that there is a path $P$ in $\omega$ connecting $x_i$ to $x_j$, which does not lie entirely in $S\cup \partial_\sboxt S$. Arguing as in the proof of \Lr{finite}, we can modify $P$ to obtain an open path $P'$ connecting $x_i$ to $x_j$ which lies entirely in $S\cup \partial_\sboxt S$. The desired claim follows now easily.

Combining the above claims, we conclude that the events $\{A,\vx \text{ is connected},$ $\mathcal{S}_{\vx}=S\}$ and $\{A,\mathcal{C},\vx\rightarrow S,S \text{ occurs}\}$ coincide, and thus $$\Pr_p(A,\vx \text{ is connected}, \mathcal{S}_{\vx}=S)=\Pr_p(A,\mathcal{C},\vx\rightarrow S,S \text{ occurs}).$$
Using the inclusion-exclusion principle we obtain that 
\begin{align}
\begin{split}
\Pr_p(A,\mathcal{C},\vx\rightarrow S,S \text{ occurs})=\Pr_p(A,\mathcal{C},S \text{ occurs})+ \\ \sum_{T}(-1)^{c(T)}\Pr_p(A,T \text{ occurs},\mathcal{C},S \text{ occurs}),
\end{split}
\end{align}
where the latter sum ranges over all finite collections $T$ of separating components separating $\vx$ from $S$. Collecting now all the terms 
we obtain that 
\begin{align}
\begin{split}
\tau^f_{\vx}(p)=\Pr_p(A^\mathsf{c},\vx \text{ is connected})+\\ \sum_S\Big(\Pr_p(A,\mathcal{C},S \text{ occurs})+\sum_{T}(-1)^{c(T)}\Pr_p(A,T \text{ occurs},\mathcal{C},S \text{ occurs})\Big).
\end{split}
\end{align}

Notice that by combining $S$ and $T$ we obtain an element of $MS^N(\vx)$, hence we can use \Lr{exp dec x}, and then argue as in the proof of \Tr{analytic} to obtain that $\tau^f_{\vx}$ is analytic above $p_c$.

\bigskip
We will now prove the analyticity of $\tau_{\vx}$.  Since $\tau^f_{\vx}$ is analytic, it suffices to prove that $\tau_{\vx}-\tau^f_{\vx}$ is analytic.
It is well-known that the infinite cluster is unique in our setup \cite{BKunique}, and this implies that 
$\tau_{\vx}-\tau^f_{\vx} = \Pr(|C_{x_1}|=\infty,\ldots,|C_{x_k}|=\infty)$. The  latter probability is complementary to $\Pr(\cup_{i=1}^k \{|C_{x_i}|<\infty\})$, which is in turn equal to
$$\Pr(\cup_{i=1}^k \{|C_{x_i}|<\infty\})=\Pr(A^\mathsf{c})+\Pr(\big(\cup_{i=1}^k \{|C_{x_i}|<\infty\}\big)\cap A).$$
Define the event $\{\text{S occurs for some } x_i\in \vx\}$ as in the previous section expect that now we require the existence of a witness $\omega'$ such that $S$ contains $\partial C_{x_i}(\omega')(N)$ for some $x_i\in \vx$.
We can expand the latter term as an infinite sum using the inclusion-exclusion principle to obtain
$$\Pr(\big(\cup_{i=1}^k \{|C_{x_i}|<\infty\}\big)\cap A)=\sum (-1)^{c(S)+1}\Pr(\text{S occurs for some } x_i\in \vx, A),$$
where now we require our separating components to surround some $x_i\in \vx$.
Arguing as in the proof of \Tr{analytic} we obtain that $\tau_{\vx}-\tau^f_{\vx}$ is analytic, as desired.

\bigskip
For the second claim of the theorem, notice that when $\mathrm{diam}(\vx)\geq N/5$, the probability $\Pr(A^\mathsf{c}, \vx \text{ is connected})$ is equal to $0$. Hence our expansion for $\tau^f_{\vx}$ simplifies to 
$$\tau^f_{\vx}(p)=\sum_S\Big(\Pr(A,\mathcal{C},S \text{ occurs})+\sum_{T}(-1)^{c(T)}\Pr(A,T \text{ occurs},\mathcal{C},S \text{ occurs})\Big).$$
Our goal is to show that for every $p>p_c$ there are some constants $\delta,t>0$  such that 
\begin{align}\label{abs}
\Bigl\lvert\sum_{|S|=n}\Big(\Pr_p(A,\mathcal{C},S \text{ occurs})+\sum_{T}(-1)^{c(T)}\Pr_p(A,T \text{ occurs},\mathcal{C},S \text{ occurs})\Big)\Bigr\rvert \leq e^{-tn}
\end{align}
for every $z\in D(p,\delta)$ for the analytic extensions of the above probabilities. Then the desired claim will follow easily from the observation that any plausible separating component $S$ of $\vx$ must have size $\Omega(\mathrm{diam}(\vx))$. 

Notice that the event $A$ depends only on the edges in the boxes $B(x_i)$, $x_i\in \vx$. Moreover, the events $\mathcal{C}$ and 
$\{S \text{ occurs}\}$ depend on $O(|S|)$ edges, while the event $\{T \text{ occurs}\}$ depends on $O(|T|)$ edges. We can now use \Lr{C equals S NN} to conclude that there is a constant $c=c(p,\delta,N)>1$ (perhaps slightly larger than that of \Lr{C equals S NN}) such that
$$|\Pr_z(A,\mathcal{C},S \text{ occurs})|\leq c^{|S|} \Pr_{p'}(A,\mathcal{C},S \text{ occurs})$$
and $$|\Pr_z(A,T \text{ occurs},\mathcal{C},S \text{ occurs})|\leq c^{|S|+|T|} \Pr_{p'}(A,T \text{ occurs},\mathcal{C},S \text{ occurs})$$
for every $z\in D(p,\delta)$, where $p'=p+\delta$ if $p<1$, and $p'=1-\delta$ if $p=1$. Moreover, we can always choose $c$ in such a way that $c\rightarrow 1$ as $\delta\rightarrow 0$. Hence we have 
\begin{align}
\begin{split}
\Bigl\lvert\sum_{|S|=n}\Big(\Pr_z(A,\mathcal{C},S \text{ occurs})+\sum_{T}(-1)^{c(T)}\Pr_z(A,T \text{ occurs},\mathcal{C},S \text{ occurs})\Big)\Bigr\rvert \leq \\
c^n\sum_{|S|=n}\Big(\Pr_{p'}(A,\mathcal{C},S \text{ occurs})+\sum_{T}c^{|T|}\Pr_{p'}(A,T \text{ occurs},\mathcal{C},S \text{ occurs})\Big)
\end{split}
\end{align}
by the triangle inequality. It follows from \Lr{exp dec x} that the sum 
$$\sum_{|S|=n}\Big(\Pr_{p'}(A,\mathcal{C},S \text{ occurs})+\sum_{T}\Pr_{p'}(A,T \text{ occurs},\mathcal{C},S \text{ occurs})\Big)$$ decays exponentially in $n$, and by choosing $\delta$ small enough we can ensure that 
$$c^n\sum_{|S|=n}\Big(\Pr_{p'}(A,\mathcal{C},S \text{ occurs})+\sum_{T}c^{|T|}\Pr_{p'}(A,T \text{ occurs},\mathcal{C},S \text{ occurs})\Big)$$ decays exponentially in $n$ as well, hence
\eqref{abs} holds.
The proof is now complete.
\end{proof}

Using \Tr{tau} we can now prove the following results.

\begin{theorem}\label{expectation}
For every $k\geq 1$ and every $d\geq 2$, the functions $\chi_k^f(p):=\mathbb{E}_p(|C(o)|^k; |C(o)|<\infty)$ are analytic in $p$ on the interval $(p_c,1]$. 
\end{theorem}
\begin{proof}
Let us show that $\chi^f(p):=\mathbb{E}(|C(o)|; |C(o)|<\infty)$ is analytic. The case $k\geq 2$ will follow similarly.
We observe that, by the definitions,
$$\chi^f(p)=\sum_{x\in \Z^d} \tau^f_{\{o,x\}}=1+\sum_{x\in \Z^d\setminus \{o\}} 
\tau^f_{\{o,x\}}.$$ The probabilities $\tau^f_{\{o,x\}}$ admit analytic extensions by \Tr{tau}, and so it suffices to prove that the 
sum $\sum_{x\in \Z^d\setminus \{o\}} \tau^f_{\{o,x\}}$ converges uniformly on an open neighbourhood of $(p_c,1]$. This follows easily from the estimates of the second sentence of \Tr{tau}, and the polynomial growth of $\Z^d$.
\end{proof}

\begin{theorem}\label{free energy}
For every $d\geq 2$, the free energy $\kappa=\mathbb{E}(|C_o|^{-1})$ is analytic in $p$ on the interval $(p_c,1]$. 
\end{theorem}
\begin{proof}
It is known \cite{AizKeNew} that $\kappa$ is differentiable on $(p_c,1)$ with derivative equal to $$f(p):=\dfrac{1}{2(1-p)}\sum_{x\in N(o)} \big(1-\tau_{\{o,x\}}(p)\big).$$ Since each $\tau_{\{o,x\}}$ is analytic on the interval $(p_c,1]$, and $\tau_{\{o,x\}}(1)=1$, $f$ is analytic on $(p_c,1]$ as well. So far we know that $\kappa$ coincides with a primitive $F$ of $f$ only on  $(p_c,1)$, which implies that $\kappa$ is analytic on that interval. In fact, $\kappa$ coincides with $F$ on the whole interval $(p_c,1]$. Indeed, we simply need to verify that $\kappa$ is continuous from the left at $1$. To see this notice that $\kappa(1)=1-\theta(1)=0$ and $\kappa(p)\leq 1-\theta(p)$. Since $\theta$ is continuous from the left at $1$, which follows e.g. by \Tr{analytic}, we have that $\kappa$ is continuous from the left at $1$ as well, hence coincides with $F$ on the whole interval $(p_c,1]$. It now follows that $\kappa$ is analytic in $p$ on the interval $(p_c,1]$, as desired. 
\end{proof}

\comment{
	\section{$p_\C= p_c$ in 2 dimensions (by AG)}

\begin{theorem}\label{theta analytic Z2}
For Bernoulli bond percolation on any planar lattice we have $p_\C= p_c$.
\end{theorem}
\begin{proof} 

\end{proof}
We will be concentrating on the square lattice $L$, i.e.\ the Cayley graph of the group $\Z^2$ \wrt\ the standard generating set $\{(0,1),(1,0)\}$. The same arguments apply to other lattices, and we will be pointing out the differences. We will be using the dual , but not the self-duality of $L$. We need the following important fact about the relation between the percolation thresholds in the primal and dual lattice. 

\begin{theorem}\label{pc star}
$p_c + p_c^*=1$ ...
\end{theorem}

A subgraph $S$ of $\Z^2$ is called a \defi{\scv} of $o$ if $o\in V(S)$ and there is a connected subgraph $G$ of $\Z^2$ cointaining $o$ such that $S$ consists of the vertices and edges incident with the unbounded face of $G$. \mymargin{For $d>2$ define $S$ as the set of points that can reach infinity in $\R^d \sm G$.}

The \defi{boundary $\partial S$} of an \scv\ $S$ is the set of edges of $\Z^2$ that are incident with $S$ and lie in the unbounded face of $S$. It is important to remember $\partial S$ may contain edges that have both their end-vertices in $S$; our proof will break down (at the first sentence of \Lr{scs boundary}) if we exclude such edges from the definition of $\partial S$.

\begin{lemma} \label{scs boundary}	
For every {\scv } $S$ we have $| \partial S| \geq |S|/2$. Moreover, the dual $\partial S^*$ of $\partial S$ is a connected subgraph of \Zsd.
\end{lemma}
(The factor $\frac12$ cannot be increased because it can happen that most edges in $\partial S$ have both their end-vertices in $S$; for example, we can have a `space filling' \scv\ whose vertex set is an $n \times n$ box of \Zs.)
\begin{proof}

\end{proof}

\begin{lemma} \label{scs axis}	
Every {\scv } $S$ of $o$ contains one of the vertices $x_0,\ldots, x_{|S|-1}$, where $x_i:=(i,0)$ is the $i$th vertex on the positive horizontal axis of \Zs.
\end{lemma}
\begin{proof}
By the definitions, $S$ must separate $o$ from infinity, and so it must meet $x_0,x_1\ldots$. If it contains $x_i$ for some $i\geq|S|$, then since $S$ is a connected graph, all its vertices lie at distance less than $|S|$ from $x_i$. This means that $S$ contains none of the vertices $\ldots,  x_{-1}, x_0$, contradicting the fact that it separates $o=x_0$ from infinity.
\end{proof}

Given a realisation $\oo$ of our Bernoulli percolation on $\Z^2$, we say that an \scv\ $S$ \defi{occurs} in \oo\ if $S$ is the boundary of the unbounded face of some cluster of \oo. This happens exactly when all edges of $S$ are occupied and all edges in $\partial S$ are vacant, and so the occurence of $S$ is a basic event\sss.

The following is an easy consequence of the definitions.
\begin{lemma} \label{scs disjoint}	
If two distinct occurring {\scv s} of $o$ share a vertex then they coincide. \qed
\end{lemma}

A \defi{\smc} $S$ of $o$ is a finite set of pairwise vertex-disjoint \scv s of $o$. 
We say that $S$  \defi{occurs} if each of the \scv s it contains occurs. Let $\MS_n:= \{ S \in \MS \mid |S|=n\}$, where $|S|$ denotes the number of vertices contained in all the \scv s in $S$.

\begin{lemma} \label{HR bound}	
There is a constant $r\in \R$ \st\ \fe\ $n\in \N$ at most $r^{\sqrt{n}}$ elements of $\MS_n$ can occur simultaneously in any \oo.
\end{lemma}
\begin{proof}
Suppose $S\in \MS_n$ occurs in \oo. 
Since occurring \scv s are vertex-disjoint by \Lr{scs disjoint}, $S$ is uniquely determined by the subset $D$ of $\{x_0, x_1, \ldots\}$ it meets, in other words, $S= \bigcup_{x_i\in D} S(x_i,\oo)$, where $S(x_i,\oo)$ denotes the occurring \scv\ of $o$ containing $x_i$.

Note that  $|S(x_i,\oo)| > i $ \fe\ $x_i \in D$ by \Lr{scs axis}. Since $n=|S| = \sum_{x_i\in D} |S(x_i,\oo)|$  by the above remark, we deduce $n> \sum_{x_i\in D} i$. This means that $D$ uniquely determines a partition of a number smaller than $n$. Moreover, distinct occurring \smc s in $\MS_n$ determine distinct subsets $D$ of $\{x_0, x_1, \ldots\}$, and therefore distinct partitions. By the Hardy--Ramanujan theorem, the number of such partitions is less than .... Thus less than ... elements of $\MS_n$ can occur simultaneously in any \oo.
\end{proof}

If $C(o)$ is finite, then there is exactly one \scv\ $S\in \CS$ that occurs and is contained in $C(o)$, namely the boundary of the unbounded face of $C(o)$.
We denote the probability of this event by $P_S$, that is, we set 
$$P_S(p):= \Pr(\text{$S$ occurs and } S\subset C(o)).$$
Thus we can write the probability $\theta_o(p)$ that $C(o)$ is finite by summing $P_S$ over all $S\in \CS$:
\labtequ{thetaS2}{$\theta_o(p)= \sum_{S\in \CS} P_S(p)$}
\fe\  $p\in (1/2,1]$.

As usual, our strategy to prove the analyticity of $\theta$, is to express it as an infinite sum of functions that admit analytic extensions, namely, probabilities of events that depend on finitely many edges, and then apply Weierstrass' \Tr{thmWei} to this sum using \Lr{C equals S} or \Cr{D and not F} to bound the absolute values of these functions in appropriate complex discs. Formula \ref{thetaS2} is a first step in this direction, however, the functions $P_S$ are not fit for our purpose: the event $\{\text{$S$ occurs and } S\subset C(o) \}$ is not measurable with respect to the set of edges incident with $S$ only. Therefore, we would prefer to express  $\theta$ in terms of the simpler functions 
$$Q_S:= \Pr_p(\text{$S$ occurs}).$$
These functions have the advantage that comply with the premise of \Cr{D and not F NN}, and hence $|Q_S(p)|$ is bounded in $D(p,M)$ by $e^{2M|S|}Q_S(p+M)$. 
When trying to write $\theta$ as a sum involving these $Q_S$, we have to be more careful: we have $\theta_o(p)=\Pr_p(\text{at least one $S\in \CS$ occurs})$ by the definitions, but more than one $S\in \CS$ might occur simultaneously. Therefore, we will apply the inclusion-exclusion principle to the set of events $\{\text{$S$ occurs}\}_{S\in \CS}$. We claim that
%
\labtequ{thetaQ}{$\theta_o(p)= \sum_{S\in \MS} (-1)^{c(S)+1} Q_S(p)$}
\fe\  $p\in (1/2,1]$, where $c(S)$ denotes the number of \scv s in the \smc\ $S$. 

To prove this, we need first of all to check that the sum in the right hand side converges. This is implied by \Lr{exp dec Z2} below, which states that\\ $\sum_{S\in \MS, |S|=n} (-1)^{c(S)} Q_S(p)$ decays exponentially in $n$, and therefore our sum converges absolutely. Then, we need to check that this sum agrees with the inclusion-exclusion formula. This is so because, \fe\ set $I$ of \scv s of $o$, we have $\Pr(\text{every $S\in I$ occurs})=0$ unless the elements of $I$ are pairwise vertex-disjoint, that is, $I \in \MS$, and so we can restrict the inclusion-exclusion formula to $\MS$ rather than consider sets of \scv s that intersect.

The main part of our proof is to show that the probability for at least one \smc\ in $\MS_n$ to occur  decays exponentially in $n$; the rest of the arguments are identical to those of e.g.\ \Tr{chi NN} or \ref{thm nonam}. More precisely,  we have
\begin{lemma} \label{exp dec Z2}
\Fe\ 	$p\in (1/2,1]$ there is a constant $c=c(p)$, $c<1$, \st\  \fe\ $\nin$,
\labtequ{dec Q}{$\sum_{S\in \MS_n} Q_S(p) \leq c^{n}.$}
\end{lemma}

The proof of this is based on the fact that for any vertex $x$, the length of the \scv\ $S(x)$ incident with $x$ has an exponential tail. This is because $\partial S(x)$ has a size of the same order as $|S(x)|$ and  is contained in a component of the dual \Zsd\ by \Lr{scs boundary}, and as our percolation is subcritical on \Zsd,  \Tr{} of Aizenman \& Barsky applies. Still, the exponential tail of each $|S(x)|$ does not easily imply \Lr{exp dec Z2}: a \smc\ might consist of plenty of \scv s, and although each of them has an exponential tail, it is not clear why their union should have an exponential tail too.
	... BK ...
} 

\section{Continuum Percolation} \label{sec cont}

In this section we will prove analyticity results for the Boolean model in $\mathbb{R}^2$ analogous to \Tr{theta analytic 2D}, answering a question of \cite{PenroseBooleanModel}. 

\medskip
Let $P_{\lambda}$ be a Poisson point process in $\mathbb{R}^d$ of intensity $\lambda$ and let
$\mathcal{N}(B)$ denote the number of points inside a bounded subset $B$ of $\mathbb{R}^d$. The Boolean model is obtained by taking the union $\Zn$ of disks of random radius $r$, called \defi{grains}, centred at the points of $P_{\lambda}$. The random radii are independent random variables and have the same distribution as another non-negative random variable $\rho$. They are also independent from $P_{\lambda}$. We denote $(P_{\lambda},\rho)$ the Boolean model with random radii sampled from $\rho$. If $\rho$ is equal to a positive constant $r$ we will write $(P_{\lambda},r)$. 

The random set $\Zn$ is called the \defi{occupied region} and its complement $\mathcal{V}$ is called the \defi{vacant region}. We will denote by $W(0)$ the connected component of $\Zn$ containing $0$ ($W(0)=\emptyset$ if $0$ is not occupied) and $V(0)$ the connected component of $\mathcal{V}$ containing $0$ ($V(0)=\emptyset$ if $0$ is occupied).

It is well-known that for every non-negative random variable $\rho$, there is a critical value $\lambda_c$ such that for every $\lambda>\lambda_c$ there is almost surely a (unique) occupied unbounded connected component $Z_{\infty}$, but no unbounded connected components exist whenever $\lambda<\lambda_c$. It is possible that the critical value is equal to $0$ or infinity. Under the assumptions that $\mathbb{E}(\rho^{2d-1})<\infty$, where $d$ is the dimension of our space, and $\Pr(\rho=0)<1$ we have $0<\lambda_c<\infty$. An important tool in the study of $Z_{\infty}$ is the \defi{perolation density} $\theta_0 :=\Pr_{\lambda}(0\in  Z_{\infty})$ of $Z_{\infty}$ (also called `volume fraction' or `percolation function'). For an introduction to the subject see \cite{MeesterRoyContPerc,PenRgg}. 

Under general assumptions on the grain distribution, $\theta_0$ is continuous for every $\lambda\neq \lambda_c$ and $d\geq 2$, and $\theta_0(\lambda_c)=0$ when $d=2$ \cite{MeesterRoyContPerc}. Similarly to the standard percolation model on $\mathbb{Z}^d$, it is expected that the latter holds for every $d\geq 3$ as well. 

Much more is known about the behaviour of $\theta_0$ on the interval $(\lambda_c,\infty)$. Recently, it has been proved in \cite{PenroseBooleanModel} that $\theta_0$ is infinitely differentiable on $(\lambda_c,\infty)$ under general assumptions on the grain distribution. The authors asked whether $\theta_0$ is analytic in that interval, and we answer this question in the affirmative when $d=2$. For simplicity we will assume that all discs have radius $1$, although our proof easily extends to the case where the radii are bounded above and below.

\begin{theorem}\label{Boolean}
Consider the Boolean model $(P_{\lambda},1)$ in $\mathbb{R}^2$. Then $\theta_0$ is analytic on $(\lambda_c,\infty)$.
\end{theorem}

The proof of \Tr{Boolean} will follow the lines of that of \Tr{theta analytic 2D}. One of the main tools in the proof of the latter is the exponential decay of the probability $\Pr_p(\text{some } S\in \MS_n \text{ occurs})$, which follows from the \ABP, duality, and the BK inequality. In the case of the Boolean model we will define another notion of \scv\ and our goal once again is to show that the probability of having large \smc s decays exponentially in their size. However, the Boolean model lacks a notion of duality which leads to certain complications. Nevertheless, it is still true that the probability $\Pr_{\lambda}(\mu(V(0))\geq a)$, where $\mu(V(0))$ denotes the area of $V(0)$, decays exponentially in $a$ for every fixed $\lambda>\lambda_c$, which we will combine with the more general Reimer inequality \cite{ContReimer}, instead of the BK inequality, to show the desired exponential decay.

Before stating the Reimer inequality let us fix some notation. We denote a sample of the Boolean model $(P_{\lambda},\rho)$ by $\omega=\{(x_i,r_i): i=1,2,\ldots\}$, where $(x_i)$ is the sequence of points of the Poisson point process and $(r_i)$ the associated sequence of radii. The \defi{restriction} of $\omega$ to a set $K\subset \mathbb{R}^d$ is
$$\omega_K :=\{(x_i,r_i)\in \omega: x_i\in K\}.$$
We also define $$[\omega]_K:=\{\omega': \omega'_K=\omega_K\}.$$
We say that an event $A$ \defi{lives on} a set $U$ if $\omega\in A$ and $\omega'\in [\omega]_U$ imply $\omega'\in A$. For $A$ and $B$ living on a bounded region $U$ we define the event
\begin{align}
\begin{split}
A\square B=\{\omega: \text{ there are disjoint sets } K,L, \text{ each a finite union of } \\ \text{ rectangles with rational coordinates, with } [\omega]_K \subset A, [\omega]_L \subset B\}.
\end{split}
\end{align}
When $A\square B$ occurs we say that $A$ and $B$ \defi{occur disjointly}.

\begin{theorem}\text{(Reimer inequality)}\cite{ContReimer} \label{Reimer}
Let $U$ be a bounded measurable set in $\mathbb{R}^d$. For any two events $A$
and $B$ living on $U$ we have $$\Pr(A\square B)\leq \Pr(A) \Pr(B).$$
\end{theorem}

Before delving into the details of the proof of \Tr{Boolean} let us give some more definitions. Let $x\in \mathbb{R}^2$ and let $\Omega$ be a bounded domain in $\mathbb{R}^2$ with piecewise $C^1$ boundary (the sets $\Omega$ we will consider are finite unions of disks). We define $dist(x,\Omega)=\inf_{y\in \Omega} \{|x-y|\}$ to be the Hausdorff distance between $x$ and $\Omega$. The area of $\Omega$ is denoted by $\mu(\Omega)$ and the length of its boundary $\partial \Omega$ by $\mathcal{L}(\partial \Omega)$.

The \defi{Minkowski sum} of two sets $\Omega_1,\Omega_2 \subset \mathbb{R}^2$ is defined as the set
$$\Omega_1+\Omega_2 :=\{a+b: a\in \Omega_1, b\in \Omega_2\}.$$ We also define $$r\Omega:=\{ra: a\in \Omega\}$$ for $r\in \mathbb{R}_{\geq 0}$. For $x\in\mathbb{R}^2$ we will write $x+\Omega$ instead of $\{x\}+\Omega$. Analogously, the \defi{Minkowski difference} is defined as the set
$$\Omega_1-\Omega_2:=\{x\in \mathbb{R}^2: x+\Omega_2\subset \Omega_1\}.$$ Note that in general $(\Omega_1-\Omega_2)+\Omega_2 \neq \Omega_1$. However, for the kind of sets we will consider equality will hold.

For $r\in \mathbb{R}_{\geq 0}$ the \defi{outer $r$-parallel set} of $\Omega$ is the set 
$$\Omega_r:=\Omega+r\overline{D},$$ where $\overline{D}=\overline{D(0,1)}$ is the closed unit disk. We will write $\overline{D(x)}$ for the closed unit disk centred at $x$. Notice that $\Omega_r$ coincides with the set $$\{x\in\mathbb{R}^2: dist(x,\Omega)\leq r\}.$$
Moreover it follows by the definitions that $(\Omega_r)_s = \Omega_{r+s}$. 

The \defi{inner $r$-parallel set} of $\Omega$ is the set 
$$\Omega_{-r}:=\Omega-r\overline{D}.$$ This set could be empty for some value of $r$ and for this reason we define the \defi{inradius} $r(\Omega)$ of $\Omega$ by
$$r(\Omega):=\sup\{r: \exists x\in \mathbb{R}^2 \text{ with } x+r\overline{D} \subset \Omega\}.$$

Given $Y=\{x_1,x_2,\ldots,x_n\} \subset \mathbb{R}^2$ we define $$\Omega(Y):=\cup_{i=1}^n \overline{D(x_i)}.$$ 
In case $\Omega(Y)$ is not simply connected, consider the bounded connected components $C_1,C_2,\ldots,C_k$ of its complement and define 
$$\tilde{\Omega}=\tilde{\Omega}(Y):= (\cup_{i=1}^k C_k) \cup \Omega(Y).$$

The next theorem upper bounds the measure of $\Omega_r$ in terms of the measure of $\Omega$ and the length of its boundary. It will be useful in the proof of \Tr{Boolean}. 

\begin{theorem}\textbf{(Steiner's inequality)}\label{Steiner}
Let $\Omega\subset \mathbb{R}^2$ be a compact simply connected set with piecewise $C^1$ boundary. Then
$$\mu(\Omega_r)\leq \mu(\Omega)+\mathcal{L}(\partial \Omega)r+\pi r^2.$$
If $\Omega$ is convex, then this inequality holds with equality. 
\end{theorem}

See \cite{FedererGeoMeasure} for a proof when $\Omega$ is convex
\cite{HadwigerSteinerschen,TreibergsIsoperimetric} for the general case.

\medskip

Let us now focus on the function $\theta_0$. If $0\not \in Z_{\infty}$, then there are two possibilities:
\begin{enumerate}
\item either there is no point of $P_{\lambda}$ in $\overline{D}$,
\item or there are points $x_1,x_2,\ldots,x_n$ of $P_{\lambda}$ in $W(0)$ such that $\Omega:=\Omega(\{x_1,\ldots,x_n\})$ is connected and there is no point of $P_{\lambda}\setminus \{x_1,\ldots,x_n\}$ at distance $r\leq 1$ from $\partial \tilde{\Omega}$.
\end{enumerate}

This observation leads to the following definition. Suppose that\\ $Y=\{x_1,x_2,\ldots,x_n\}$ is a  subset of $\mathbb{R}^2$ satisfying 
\begin{enumerate}\label{conditions}
\item \label{con i} $\Omega:=\Omega(Y)$ is connected; 
\item \label{con ii} $0\in \tilde{\Omega}$; and 
\item \label{con iii} $\overline{D(x_i)}\cap \partial \tilde{\Omega}$ contains an arc of positive length for every $i=1,\ldots,n$.
\end{enumerate}
Then we call $\partial \tilde{\Omega}$ an \defi{\scv} and we denote it by $J(Y)$. The set $S(Y):=J(Y)+\overline{D}$ is called a \defi{\ssp}. We say that a set $Y$ as above \defi{happens to separate in} $P_{\lambda}$ if $Y\subset P_{\lambda}$ and no other point of $S(Y)$ belongs to $P_{\lambda}$. We say that $S(Y)$ \defi{occurs} whenever $Y$ happens to separate in $P_{\lambda}$.

There is subtle point in the latter definition. It is possible that the boundary of $S(Y)$ contains points of the Jordan domain enclosed by $J(Y)$ ($\tilde{\Omega}$ with the above notation) that do not belong in $Y$. Moreover, it can happen that some of these points are occupied. However, having such a $Y$ in $P_{\lambda}$ is an event of measure $0$ and so we can disregard it.

To avoid such trivialities, we will always assume that no pair of points $x_i,x_j$ of $P_{\lambda}$ have distance $2$, which implies that no pair of disks touch. We can do so as this event has measure $0$.

The following lemma is an easy consequence of the definitions.
\begin{lemma} \label{ssp disjoint}	
If $Y_1$ and $Y_2$ happen to separate in $P_{\lambda}$ and $S(Y_1), S(Y_2)$ have non empty intersection, then $Y_1=Y_2$. \qed
\end{lemma}

This leads us to define a \defi{\smc} as a finite set of pairwise disjoint \scv s and a \defi{\sms} as a finite set of pairwise disjoint \ssp s. A \sms\ \defi{occurs} if each of its \ssp s occurs.

Using the above definitions we obtain
$$1-\theta_0(\lambda)=\Pr_{\lambda}(0\not \in Z_{\infty})=\Pr_{\lambda}(\text{some } S(Y) \text{ occurs})$$ for every $\lambda>\lambda_c$.
The second equality follows from the fact that whenever $0\not \in Z_{\infty}$ and no $Y$ happens to separate in $P_{\lambda}$, $0$ belongs to an infinite vacant component, and this event has measure $0$ for every $\lambda>\lambda_c$ \cite{MeesterRoyContPerc}.

Once again we intend to use the inclusion-exclusion principle to obtain the formula 
$$\Pr_{\lambda}(\text{some } S(Y) \text{ occurs}) = \sum_{k=1}^\infty (-1)^{k+1}\mathbb{E}_{\lambda}(N(k))$$
\fe\  $\lambda\in (\lambda_c,\infty)$, where $N(k)$ is the number of occurring \sms s comprising $k$ \ssp s. 

To prove the validity of the above formula we will show that the alternating sum converges absolutely. In order to do so, we first express the above expectations as an infinite sum according to the area of $S(Y_i)$, i.e. $$\mathbb{E}_{\lambda}(N(k))=\sum_{\{m_1,\ldots,m_k\}} \mathbb{E}_{\lambda}(N(k,\{m_1,\ldots,m_k\})),$$ where the sum in the right hand side ranges over all multi-sets of positive integers with $k$ elements, and $N(k,\{m_1,\ldots,m_k\})$ is the number of occurring \sms s $S=\{S_1,\ldots,S_k\}$ with $\left \lfloor{\mu(S_i)}\right \rfloor=m_i$.

Let us define $P_n$ to be the set of partitions of $n$ and $\MS_n$ to be the set of \sms s 
$S=\{S_1,\ldots,S_k\}$ with $\left \lfloor{\mu(S_1)}\right \rfloor + \ldots + \left \lfloor{\mu(S_i)}\right \rfloor=n$. We denote by $N_n$ the number of occurring \sms s of $\MS_n$.
The analogue of \Lr{exp dec Z2} is

\begin{lemma} \label{exp dec R2}
\Fe\ 	$\lambda\in (\lambda_c,\infty)$ there are constants $c_1=c_1(\lambda)$ and $c_2=c_2(\lambda)$ with $c_2<1$  \st\  \fe\ $\nin$,
\labtequ{dec Q R2}{$\mathbb{E}_{\lambda}(N_n) \leq c_1 c_2^{n}.$}
\end{lemma}

Notice that whenever a \ssp\ $S$ occurs, a subset of $S$ is vacant. Thus we are lead to use the exponential decay in $a$ of the probability \\ $\Pr_{\lambda}(\mu(V(0))\geq a)$ for every $\lambda>\lambda_c$ \cite{MeesterRoyContPerc}. However, we cannot directly apply the aforementioned exponential decay as it is possible for the area of the vacant subset of $S$ to be relatively small compared to the area of $S$. 

In order to overcome this difficulty we fix a $\lambda>\lambda_c$ and consider a small enough $1>\varepsilon>0$ such that $\lambda_c(B_{1-\varepsilon})<\lambda$, where $\lambda_c(B_{1-\varepsilon})$ is the critical point of the Poisson Boolean model $(P_{\lambda},1-\varepsilon)$. We couple the two models by sampling a Poisson point process with intensity $\lambda$ in $\mathbb{R}^2$ and placing two disks, one of radius $1$ and another of radius $1-\varepsilon$, centred at each point of the process. We notice that whenever a \ssp\ $S=S(Y)$ occurs in $(P_{\lambda},1)$, the set $S(\varepsilon) :=J(Y)+D(0,\varepsilon)$ is vacant in $(P_{\lambda},1-\varepsilon)$ in our coupling and our goal is to show that this happens with probability that decays exponentially in the area of $S$.

First we need to show that $\mu(S(\varepsilon))$ and $\mu(S)$ are of the same order. We do so in the following purely geometric lemma.

\begin{lemma}\label{ssp area}
Let $1>\varepsilon>0$. Then there are constants $\gamma_1=\gamma_1(\varepsilon)>0, \gamma_2=\gamma_2(\varepsilon)>0$ such that for every \ssp\ $S=S(Y)$ we have $$\mu(S(\varepsilon))\geq \gamma_1\mu(S)-\gamma_2.$$
\end{lemma} 
\begin{proof}
Let $J=J(Y)$ be the corresponding \scv\ of $S$. Easily, we can assume that $J$ is not a single circle. We define $\Omega=\Omega(Y)$ to be the closure of the Jordan domain bounded by $J$. Let $S_{-1}(\varepsilon)$ be the intersection of $S(\varepsilon)$ with $\Omega$. We will show that 
\begin{align}\label{first inequality}
\mu(\Omega_1)-\mu(\Omega) \leq 2(\mu(\Omega)-\mu(\Omega_{-1}))+\pi
\end{align}
and 
\begin{align}\label{second inequality}
\mu(S_{-1}(\varepsilon))\geq d(\mu(\Omega)-\mu(\Omega_{-1}))
\end{align} 
for some constant $d=d(\varepsilon)>0$ independent of $S$. Then the assertion follows immediately, as $\mu(S)=\mu(\Omega_1) - \mu(\Omega_{-1})$.

For inequality \eqref{first inequality} it suffices to prove that 
\begin{align}\label{length area}
\mathcal{L}(J)\leq 2(\mu(\Omega)-\mu(\Omega_{-1}))
\end{align} 
because by Steiner's inequality (\Tr{Steiner}) we have $$\mu(\Omega_1)\leq \mu(\Omega)+\mathcal{L}(J)+\pi.$$

For every $x\in Y$ the intersection of $J$ with the the circle $C(x)$ of radius $1$ centred at $x$ may contain several connected components. Let $(J_i)$ be an enumeration of all these connected components and $(x_i)$ the corresponding sequence of centres, i.e. $x_i$ is the centre of the arc $J_i$ (some $x\in Y$ may appear more than once). Every arc $J_i$ has two endpoints $A_i,B_i$ and each endpoint $E_i\in \{A_i,B_i\}$ belongs to two disks $\overline{D(x_i)}$ and $\overline{D(x_i')}$ for some $i'=i'(E_i)$. 

Let $S(i)$ be the open sector of $D(x_i)$ enclosed by the radii $x_i A_i$, $x_i B_i$ and the arc $J_i$. Notice that $S(i)$ is a subset of $\Omega\setminus \Omega_{-1}$. We claim that any two distinct $S(i), S(j)$ are disjoint. To see this, let $x_{i'}$ be the second center that has distance $1$ from $E_i$. Observe that no centres $x\in Y$ belong to the open disk $D(E_i)$, where $E_i\in \{A_i,B_i\}$, because otherwise $E_i$ would not belong to the boundary of $S$. 
Moreover, every segment $E_k x_j$ that intersects $E_i x_i$ has to intersect $C(x_i)$ as well, because $E_k$ belongs to the boundary of $S$ and thus it does not belong to the open disk $D(x_i)$. Hence if $E_k x_j$ intersects $E_i x_i$, then $x_j$ is at distance at most $1$ from $C(x_i)$. It is easy to deduce geometrically that for every $P\in C(x_i)$ the only points $Q$ of $\overline{D(P)}\setminus \{x_i\}$ such that $Q P$ intersects $E_i x_i$ belong to $D(E_i)$ (see \fig{figintersect}), which implies that the $S(i)$'s are disjoint.

\begin{figure}
\centering{
\resizebox{75mm}{!}{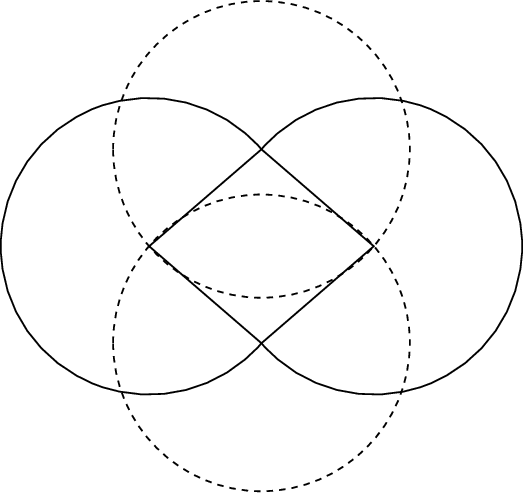}
\put(-167,99){$x_i$}
\put(-55.5,98.5){$x_{i'}$}
\put(-110.9,128){$E_i$}}
\caption{Four disks of radius $1$ centred at $x_i,x_{i'},E_i$ and another point of $C(x_i)$.} \label{figintersect}
\end{figure} 

These observations imply that $$\sum_{i}\mu(S(i))\leq \mu(\Omega)-\mu(\Omega_{-1}).$$ An elementary computation yields $\mathcal{L}(J_i)=2\mu(S_i)$, which implies that 
$$\mathcal{L}(J)=\sum_{i} \mathcal{L}(J_i)=2 \sum_{i}\mu(S(i))\leq 2(\mu(\Omega)-\mu(\Omega_{-1}))$$ 
establishing \eqref{length area}.

For inequality \eqref{second inequality} we will assume for technical reasons that $\varepsilon<1/2$. The case $\varepsilon\geq 1/2$ follows readily, because $S_{-1}(\varepsilon)$ increases as $\varepsilon$ increases.  

We will split both $S_{-1}$ and $S_{-1}(\varepsilon)$ into several smaller pieces. Let us first focus on $S_{-1}$. The two radii $E_i x_i$ and $E_i x_i'$ that emanate from the endpoint $E_i\in \{A_i,B_i\}$ of $J_i$ define an open sector $T(E_i)$ of $D(E_i)$. By the definitions, the collection of all the $\overline{T(E_i)}$'s together with all the $\overline{S(i)}$'s cover $S_{-1}$ (see \fig{figcover}). The elements of the collection are not necessarily pairwise disjoint, but this works only in our favour as we need a mere upper bound for the area of $S_{-1}$.

\begin{figure}
\centering{
\resizebox{75mm}{!}{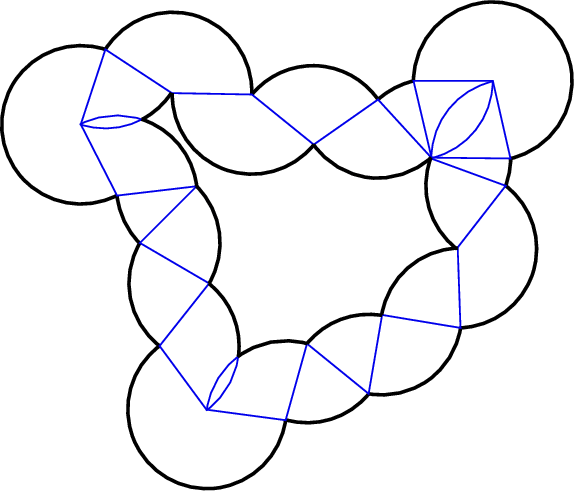}}
\caption{The domain $S_{-1}$ enclosed by the the black curves and the sectors $T(E_i)$ enclosed by the blue radii and the blue/black arcs.} \label{figcover}
\end{figure}

We will now compare the areas of $S(i)$ and $T(E_i)$ with those of their subsets $S(i,\varepsilon)=S(i)\cap S_{-1}(\varepsilon)$ and $T(E_i,\varepsilon)=T(E_i)\cap S_{-1}(\varepsilon)$. As the sectors $S(i)$ do not intersect, the sets $S(i,\varepsilon)$ do not intersect either. It is a matter of simple calculations to see that 
\begin{align}\label{first estimate}
\mu(S(i,\varepsilon))=(1-(1-\varepsilon)^2)\mu(S(i)).
\end{align}
On the other hand, the $T(E_i,\varepsilon)$'s may intersect. Our goal is to associate to every $T(E_i)$ a domain $\Omega(E_i)$ that contains $T(E_i)$ and every other $T(E_j)$ such that $T(E_j,\varepsilon)$ intersects $T(E_i,\varepsilon)$. Later on we will be generous and keep only some $\Omega(E_i)$ that we need to cover $S_{-1}$. In order to define $\Omega(E_i)$, notice first that whenever $T(E_i,\varepsilon)$ and $T(E_j,\varepsilon)$ intersect, $E_j$ has distance at most $2\varepsilon<1$ from $E_i$. Hence any other point of $T(E_j)$ has distance at most $1+2\varepsilon$ from $E_i$. Consider the points $y=y(E_i,x_i,\varepsilon)$ and $y'=y'(E_i,x_{i'},\varepsilon)$ in $E_i x_i$ and $E_i x_{i'}$, respectively, that have distance $2\varepsilon$ from $E_i$ (see \fig{figomega}). Extend each of $y x_{i'},y' x_{i},E_i x_i$ and $E_i x_{i'}$ up to distance $1+2\varepsilon$ from $E_i$, and let $z',z,w$ and $w'$ be the endpoints of these new segments. 
Define $\Omega(E_i)$ as the domain enclosed by the segments $E_i x_i,E_i x_{i'},x_{i} z, x_{i'} z'$ and the arc of the circle  $C(E_i,1+2\varepsilon)$ from $z$ to $z'$ that contains $w$ and $w'$. 

It is easy to see from the construction of $\Omega(E_i)$ that any $T(E_j)$ such that $T(E_j,\varepsilon)$ intersects $T(E_j,\varepsilon)$, is contained in $\Omega_i$. This follows from the fact that the $S(i)$'s are disjoint as proved above, and so no other sector $T(E_j)$ intersects $E_i x_{i}$ or $E_i x_{i'}$.

\begin{figure}
\centering{
\resizebox{75mm}{!}{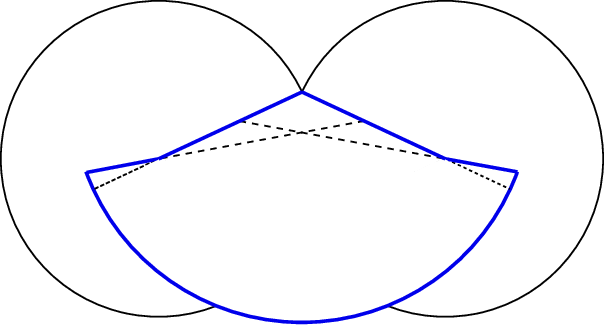} 
\put(-115,40){$\Omega(E_i)$}
\put(-189,43){$w$}
\put(-184.5,56){$z$}
\put(-161,51){$x_i$}
\put(-132.5,76.5){$y$}
\put(-110,71.5){$E_i$}
\put(-86,76){$y'$}
\put(-61,50.5){$x_{i'}$}
\put(-33,55.5){$z'$}
\put(-31.5,42.5){$w'$}}
\caption{The domain $\Omega(E_i)$.}
\label{figomega}
\end{figure}

We claim that there is a constant $\delta=\delta(\varepsilon)>0$ such that 
\begin{align}\label{second estimate}
\mu(T(E_i,\varepsilon))\geq \delta \mu(\Omega(E_i))
\end{align}
for every $i$. Indeed, the area of the sector $S(E_i,w,w')$ of $D(E_i,1+2\varepsilon)$ bounded by the radii $E_i w$ and $E_i w'$ is of the same order as the area of $\Omega(E_i)$, because the angles of the segments $x_{i} w, x_{i} z$ and $x_{i'} w', x_{i'} z'$ are of the same order as the angle $\theta$ of the segments $E_i x_i, E_i x_{i'}$. Moreover, there is some constant $\Theta=\Theta(\varepsilon)>0$ such that if $\theta$ is smaller than $\Theta$, then $x_i$ and $x_{i'}$ are close enough that the sector of $D(E_i,\varepsilon)$ defined by the segments $E_i x_i$ and $E_i x_{i'}$ is contained in $S_{-1}(\varepsilon)$. A simple computation shows that the area of this sector is of the same order as the area of $S(E_i,w,w')$. On the other hand, $\mu(T(E_i,\varepsilon))$ is bounded from below by a strictly positive constant for every $\theta\geq \Theta$. Combining all the above we conclude that \eqref{second estimate} holds.

Let us consider a set $F$ of endpoints that is maximal with respect to the property that $T(E_i,\varepsilon)$ and $T(E_j,\varepsilon)$ do not intersect for any $E_i,E_j\in F$ with $i\neq j$. The maximality of $F$ implies that the collection $\mathcal{S}$ of all the $S(i)$'s together with the collection $\mathcal{O}$ of the $\Omega(E_i)$'s for $E_i\in F$ cover $S_{-1}$, because for any other set $T(E_k)$ with $E_k\not\in F$, $T(E_k,\varepsilon)$ intersects some $T(E_i,\varepsilon)$ with $E_i\in F$ and thus $T(E_k)$ is contained in $\Omega(E_i)$. However, it is possible that some element of $\mathcal{S}$ intersects some element of $\mathcal{O}$. Nevertheless, each intersection point is counted exactly twice, because the elements of $\mathcal{S}$ and $\mathcal{O}$ are disjoint. Hence $$\mu(S_{-1}(\varepsilon))\geq 1/2 \big( \sum_{i} \mu(S(i,\varepsilon)) + \sum_{x \in F} \Omega(x) \big),$$ which combined with \eqref{first estimate} and \eqref{second estimate} implies inequality \eqref{second inequality}.
\end{proof}

Notice that every $S(\varepsilon)$ has a non-empty intersection with the non-negative real line $[0,\infty)$, because $S$ has this property. In fact if $x$ is the point of $J\cap [0,\infty)$ which has greatest distance from $0$, where $J$ is the \scv\ that defines $S$, then the interval $[x,x+\varepsilon)$ is contained in $S(\varepsilon)\cap [0,\infty)$. We conclude that $S(\varepsilon)$ contains one of the points
$\{0,\varepsilon,2\varepsilon,\ldots,N\varepsilon\}$ for some $N\in\mathbb{N}$ depending on  $S(\varepsilon)$. The next lemma provides a uniform upper bound for $N$ that depends only on the area of $S$.

\begin{lemma}\label{ssp distance}
For every \ssp\ $S=S(Y)$ we have
$$S\subset D(0,3\mu(S)).$$
\end{lemma}
\begin{proof}
Let $J=J(Y)$ be the \scv\ that defines $S$, and $\Omega=\Omega(Y)$ the closure of the Jordan domain bounded by $J$. By the definition of $J$ we have $0\in \Omega$. Thus the distance of any point of $J$ from $0$ is bounded from above by $\mathcal{L}(J)$. This implies that the distance of any point in $S$ from $0$ is bounded from above by $\mathcal{L}(J)+1$. Combining \eqref{length area} with the fact that 
$\mu(\Omega)-\mu(\Omega_{-1})\leq \mu(S)$ we obtain $$\mathcal{L}(J)\leq 2\mu(S).$$ Moreover, $\mu(S)> 1$ because by definition $S$ contains a disk of radius $1$. Therefore 
$$\mathcal{L}(J)+1< 3\mu(S).$$ Combining these inequalities yields the desired assertion.
\end{proof}

We deduce from \Lr{ssp distance} that $N$ can be chosen to be $\left \lfloor{3\mu(S)/{\varepsilon}}\right \rfloor $. We are now almost ready to prove the desired exponential decay. Before we do so we need to upper bound the number of occurring \sms s of $\MS_n$.

\begin{lemma} \label{ssp occur}	
There is a constant $R\in \R$ \st\ \fe\ $n\in \N$ at most $R^{\sqrt{n}}$ elements of $\MS_n$ can occur simultaneously in any $\omega$.
\end{lemma}
\begin{proof}
Notice that a \ssp\ $S=S(Y)$ contains an interval of the form $[x,x+1]$ for some $x\in [0,\infty)$. Combined with \Lr{ssp distance} this implies that $S$ contains some element of the set $\{0,1,\ldots,\left \lfloor{3\mu(S)}\right \rfloor \}$. We can now proceed as in the proof of \Lr{HR bound}.
\end{proof}

We are now ready to prove \Lr{exp dec R2}.

\begin{proof}[Proof of \Lr{exp dec R2}]
Since $$N_n\leq R^{\sqrt{n}}\mathbb{\chi}_{\{\text{some } S\in\MS_n \text{ occurs}\}}$$ by \Lr{ssp occur}, we conclude that $$\mathbb{E}_{\lambda}(N_{n})\leq R^{\sqrt{n}} \Pr_{\lambda}(\text{some } S\in \MS_n \text{ occurs}).$$ Hence it suffices to show that $\Pr_{\lambda}(\text{some } S\in \MS_n \text{ occurs})$ decays exponentially. 

Recall our coupling between the Boolean models $(P_{\lambda},1)$ and $(P_{\lambda},1-\varepsilon)$, and the fact that whenever $Y$ happens to separate in $P_{\lambda}$ the set $S(\varepsilon)$ is a vacant connected subset of $(P_{\lambda},1-\varepsilon)$ in our coupling.
For $m\in \mathbb{N}$, let $V(m)$ denote the event that there is a subset $V$ of a vacant component with $\mu(V)\geq \gamma_1 m -\gamma_2$, where $\gamma_1,\gamma_2$ are the constants of \Lr{ssp area}, and some element of the set $\{0,\varepsilon,\ldots,\left \lfloor(3m+3)/\varepsilon \right \rfloor \varepsilon\}$ belongs to $V$, and $V$ is contained in $D(0,3m+3)$. We claim that
$$\Pr_{\lambda}(\text{some } S\in \MS_n \text{ occurs})\leq \sum_{\{m_1,m_2,\ldots,m_k\} \in P'_n} \Pr_{\lambda,1-\varepsilon}(V(m_1) \square \ldots \square V(m_k)),$$
where as above $\square$ means that the events occur disjointly, $P'_n$ is the set of partitions of $n$ with the property that for every $N\leq n$ at most $3N+3$ elements of the partition have size at most $N$, and the probability measure $\Pr_{\lambda,1-\varepsilon}$ refers to the Boolean model $(P_{\lambda},1-\varepsilon)$. The upper bound $3N+3$ on the number of elements of size at most $N$ comes from the fact that any \ssp\ $S=S(Y)$ contains some element of the set $\{0,1,\ldots,\left \lfloor{3\mu(S)}\right \rfloor \}$, as remarked in the proof of \Lr{ssp occur}. The inequality follows similarly to \eqref{BKclaim}.

Reimer's inequality \cite{ContReimer} states that 
$$\Pr_{\lambda,1-\varepsilon}(V(m_1) \square \ldots \square V(m_k))\leq \Pr_{\lambda,1-\varepsilon}(V(m_1))\cdot \ldots \cdot \Pr_{\lambda,1-\varepsilon}(V(m_k)).$$
Combining the fact that $\Pr_{\lambda,1-\varepsilon}(\mu(V(0))\geq a)\leq c^a$ \cite{MeesterRoyContPerc} for every $\lambda>\lambda_c$ and some $c=c(\lambda)<1$ with the union bound we obtain 
$$\Pr_{\lambda,1-\varepsilon}(\mu(V(m))\leq c_1{c_2}^m,$$
where $c_1=(\left \lfloor{(3m+3)/\varepsilon}\right \rfloor + 1)c^{-\gamma_2}$ and $c_2=c^{\gamma_1}<1$.
We can now argue as in the proof of \Lr{exp dec Z2} to obtain the desired exponential decay.
\end{proof}

We proceed by establishing the analyticity and the necessary estimates of the functions involved in \Lr{exp dec R2} that we will combine with their exponential decay to prove the analyticity of $\theta_0$.

Given a partition $\{m_1,m_2,\ldots,m_k\}$ of a number $n$, we define $N(\{m_1,\ldots,m_k\})$ to be the number of occurring \sms s $S=\{S_1,\ldots,S_k\}$ such that $\left \lfloor{\mu(S_i)} \right \rfloor = m_i$.

\begin{lemma}\label{estimates}
Let $\{m_1,m_2,\ldots,m_k\}$ be a partition of $n$. Then the function $f(\lambda):=\mathbb{E}_{\lambda}(N(\{m_1,\ldots,m_k\}))$ admits an entire extension satisfying 
\labtequ{bound f}{$|f(z)|\leq e^{4nM} f(\lambda+M)$} for every $\lambda\geq 0$, $M>0$ and $z\in D(\lambda,M)$.
\end{lemma}
\begin{proof}
To ease notation we will prove the assertion for $k=2$ and $m_1\neq m_2$. The general case can be handled similarly.  

Given two disjoint sets $Y_1=\{x_1,\ldots,x_{j_1}\}$ and $Y_2=\{x_{j_1+1},\ldots,x_{j_1+j_2}\}$,
we let $L(x_1,\ldots,x_{j_1+j_2})$ denote the indicator function of the event that the sets $Y_1$ and $Y_2$ satisfy all three properties 
\eqref{con i}-\eqref{con iii} in the definition of a \ssp, and furthermore, $\left \lfloor{\mu(S(Y_i)} \right \rfloor = m_i$, $i=1,2$. The indicator function of the event $\{Y_i \text{ happens to separate in } P_{\lambda}\}$ is denoted by $\mathbb{\chi}_{Y_i}$. Let us also define the functions $$g(x_1,\ldots,x_{j_1+j_2}):= \mu(S(x_1,\ldots,x_{j_1}))+\mu(S(x_{j_1+1},\ldots,x_{j_2}))$$ and
$$h(x_1,\ldots,x_{j_1+j_2}):= L(x_1,\ldots,x_{j_1+j_2})e^{-\lambda g(x_1,\ldots,x_{j_1+j_2})}.$$

First, we will find a suitable formula for $f$. We claim that 
\labtequ{sum formula}{$$f(\lambda)=\sum_{j_1=1}^\infty \sum_{j_2=1}^\infty \dfrac{(\lambda \mu(6nD))^{j_1+j_2}}{{j_1}! {j_2}!} f(\lambda,j_1,j_2),$$}
where 
\labtequ{int formula}{$$f(\lambda,j_1,j_2)= \int_{6nD}\dfrac{dx_1}{\mu(6nD)} \ldots \int_{6nD}\dfrac{dx_{j_1+j_2}}{\mu(6nD)} h(x_1,\ldots,x_{j_1+j_2}).$$}

Indeed, expressing $f$ according to the size of $Y_1$ and $Y_2$ we obtain
$$f(\lambda)=\sum_{j_1=1}^\infty \sum_{j_2=1}^\infty \mathbb{E}_\lambda \Big(N\big(\{(m_1,j_1),(m_2,j_2)\}\big)\Big)$$
where $N\big(\{(m_1,j_1),(m_2,j_2)\}\big)$ denotes the number of sets $Y_1,Y_2$ that happen to separate with the property that $\left \lfloor{\mu(S(Y_i)} \right \rfloor = m_i$, $|Y_i|=j_i$, $i=1,2$. This expression holds because we have assumed that $m_1\neq m_2$, and so each $\{(m_1,j_1),(m_2,j_2)\}$ appears exactly once. Next notice that
\labtequ{measure}{$\mu(S(Y_1))+\mu(S(Y_2))\leq (k_1+1)+(k_2+1) \leq 2k_1+2k_2= 2n$,}
since $1\leq k_1,k_2$, which combined with \Lr{ssp distance}, implies that\\ $N\big(\{(m_1,j_1),(m_2,j_2)\}\big)$ depends only on the points of the Poisson point process inside the disk $6nD$.
Now regard $P_{\lambda}\cap 6nD$ as a finite Poisson point process whose total number of points has a Poisson distribution with parameter $\lambda \mu(6nD)$, each point being uniformly distributed over $6nD$. Notice that conditioned on the number of points $\mathcal{N}(6nD)$ inside $6nD$, the distribution of the sets $Y_1,Y_2$ depends only on their sizes. 

Conditionally on the event $\{\mathcal{N}(6nD)=m\}$ and the sets $Y_1=\{x_1,\ldots,x_{j_1}\}$ and $Y_2=\{x_{j_1 +1},\ldots,x_{j_1+j_2}\}$ being contained in $P_{\lambda}$, the expectation of $\mathbb{\chi}_{Y_1}\mathbb{\chi}_{Y_2}$ is equal to $$H_m(x_1,\ldots,x_{j_1 +j_2}):=L(x_1,\ldots,x_{j_1+j_2})\Big(\dfrac{\mu(6nD)-g(x_1,\ldots,x_{j_1+j_2})}{\mu(6nD)}\Big)^{m-j_1-j_2},$$
because every other point of the Poisson point process must lie outside of $S(Y_1)$, $S(Y_2)$. Hence expressing $f$ according to the number of points of the Poisson process inside $6nD$ and the size of the sets $Y_1,Y_2$ we obtain 
$$f(\lambda)=\sum_{j_1=1}^\infty \sum_{j_2=1}^\infty \sum_{m=j_1+j_2}^\infty e^{-\lambda \mu(6nD)}\dfrac{(\lambda \mu(6nD))^m}{m!}{m \choose j_1}{m-j_1 \choose j_2}F(j_1,j_2,m),$$
where
$$F(j_1,j_2,m)=\int_{6nD}\dfrac{dx_1}{\mu(6nD)} \ldots \int_{6nD}\dfrac{dx_{j_1+j_2}}{\mu(6nD)} H_m(x_1,\ldots,x_{j_1 +j_2}).$$

The factors $e^{-\lambda \mu(6nD)}\dfrac{(\lambda \mu(6nD))^m}{m!}$ and ${m \choose j_1}{m-j_1 \choose j_2}$ correspond to the probability $\Pr_{\lambda}(\mathcal{N}(6nD)=m)$ and the number of ways to choose two disjoint subsets of size $j_1$ and $j_2$ from a set of size $m$ (here the order of the sets matters because $m_1\neq m_2$), respectively. After changing the order of the second summation and integration, using the Taylor expansion 
$$\sum_{m=j_1+j_2}^\infty \dfrac{(\lambda(\mu(6nD)-g(x_1,\ldots,x_{j_1+j_2})))^{m-j_1-j_2}}{(m-j_1-j_2)!}=e^{\lambda(\mu(6nD)-g(x_1,\ldots,x_{j_1+j_2}))}$$ and cancelling some terms, we arrive at formula \eqref{sum formula}.

Using \eqref{sum formula} we see that $f$ extends to an entire function. Indeed, the assertion will follow from the standard tools once we have shown that every summand of $f$ is an entire function and that the upper bound \eqref{bound f} holds for the summands of $f$ in place of 
$f$.

First we express $e^{-\lambda g(x_1,\ldots,x_{j_1+j_2})}$ via its Taylor expansion
$$e^{-\lambda g(x_1,\ldots,x_{j_1+j_2})}=\sum_{s=0}^\infty \dfrac{(-\lambda g(x_1,\ldots,x_{j_1+j_2}))^s}{s!}.$$ We will plug this into \eqref{int formula}. We notice that the coefficient 
$$\int_{6nD} \frac{dx_1}{\mu(6nD)}\ldots \int_{6nD} \frac{dx_{j_1+j_2}}{\mu(6nD)} L(x_1,\ldots,x_{j_1+j_2}) (-g(x_1,\ldots,x_{j_1+j_2}))^s /{s!}$$
is bounded in absolute value by $(2n)^s/{s!}$, as $g(x_1,\ldots,x_{j_1+j_2})=\mu(S(Y_1))+\mu(S(Y_2))\leq 2n$ by \eqref{measure} and $0\leq L(x_1,\ldots,x_{j_1+j_2})\leq 1$. Therefore the function defined by the Taylor expansion 
$$\sum_{s=0}^\infty \lambda^s \int_{6nD} \frac{dx_1}{\mu(6nD)}\ldots \int_{6nD} \frac{dx_{j_1+j_2}}{\mu(6nD)} L(x_1,\ldots,x_{j_1+j_2}) (-g(x_1,\ldots,x_{j_1+j_2}))^s /{s!}$$ is entire and by reversing the order of summation and integration we conclude that it coincides with $f(\lambda,j_1,j_2)$.

Now let $\lambda\geq 0$ and $M>0$. Since $|z|^{j_1+j_2}\leq (\lambda+M)^{j_1+j_2}$ for every $z\in D(\lambda,M)$, inequality \eqref{bound f} will follow once we prove that 
\labtequ{j_1 j_2}{$|f(z,j_1,j_2)|\leq e^{4nM} f(\lambda+M,j_1,j_2)$ for every $z\in D(\lambda,M)$.} 
Using once again \eqref{measure} we obtain
\begin{align*}
\begin{split}
|e^{-zg(x_1,\ldots,x_{j_1+j_2})}|\leq e^{-(\lambda-M)g(x_1,\ldots,x_{j_1+j_2})}= \\
e^{2M g(x_1,\ldots,x_{j_1+j_2})}e^{-(\lambda+M)g(x_1,\ldots,x_{j_1+j_2})}\leq 
e^{4nM}e^{-(\lambda+M)g(x_1,\ldots,x_{j_1+j_2})}.
\end{split}
\end{align*}
Hence \eqref{j_1 j_2} follows from the triangle inequality. This proves \eqref{bound f}.

Combining \eqref{j_1 j_2} with \eqref{sum formula} and the theorems of Weierstrass in the Appendix imply that $f$ is analytic as well.
\end{proof}

We are finally ready to prove \Tr{Boolean}.

\begin{proof}[Proof of \Tr{Boolean}]
Consider the functions $$f(\lambda)=\sum_{k=1}^\infty (-1)^{k+1}\mathbb{E}_{\lambda}(N(k))$$
and $$g_n(\lambda):=\sum_{\{m_1,m_2,\ldots,m_k\} \in P_n} (-1)^{k+1} \mathbb{E}_{\lambda}(N(\{m_1,\ldots,m_k\}).$$
Notice that $$f=\sum_{n=1}^\infty g_n.$$
By \Lr{exp dec R2} we have 
$$\sum_{k=1}^\infty \mathbb{E}_{\lambda}(N(k))<\infty $$
for any $\lambda>\lambda_c$. Hence $f$ coincides with $1-\theta_0$ on the interval $(\lambda_c,\infty)$ by the inclusion-exclusion principle as remarked above. 
Combining \Lr{exp dec R2} with \Lr{estimates} we conclude that for every $\lambda>\lambda_c$ there are constants $M=M(\lambda)>0$, $c_1=c_1(\lambda)>0$ and $0<c_2=c_2(\lambda)<1$ such that $|g_n(z)|\leq c_1 {c_2}^n$ for every $z\in D(\lambda,M)$.
As usual, by the theorems of Weierstrass in the Appendix we conclude that $f$, and thus $\theta_0$, is analytic on the interval $(\lambda_c,\infty)$.
\end{proof}

\section{Finitely presented groups} \label{sec fp}

In this section we will prove that $p_\C<1$ holds for every finitely presented \Cg. The ideas used involve a refinement of Peierls' argument as in Timar's proof \cite{TimCut} of the theorem of Babson \& Benjamini \cite{BaBeCut} that $p_c<1$ for those graphs, combined with the ideas of \Sr{sec th pl}. We start with a sketch of these ideas.

\medskip
Peierls' classical argument for proving e.g.\ that $p_c<1$ for bond percolation on a planar lattice $G$ goes as follows. If the cluster $C(o)$ of the origin $o$ is finite in a percolation instance, then $C(o)$ is surrounded by a `cut' of vacant edges, which form a cycle in the dual lattice $G^*$. But the number of candidate cycles of $G^*$ with length $n$ is at most $d_*^n$, where $d_*$ is the degree of $G^*$, and each of them occurs with probability $(1-p)^n$ in a percolation instance. Therefore, the union bound implies that we can make the probability that at least one of them occurs smaller than 1 if we choose $p$ is close enough to 1, because the exponential decay of $(1-p)^n$ outperforms the at most exponential growth of the number of candidate cycles.

For this argument it was not crucial that the cut separating $C(o)$ from infinity was a cycle: to deduce that there are at most $c^n$ candidate cuts for some constant $c$, it suffices if the edges of any such cut $B$ are close to each other in the following sense. If we build an auxiliary graph, with vertex set $B$, by connecting any two edges of $B$ with an edge whenever their distance is at most some bound, then this auxiliary graph is connected. For if this is the case, then using the fact that every regular graph has at most exponentially many connected subgraphs containing a fixed vertex and $n$ further vertices (see \Sr{sec App Enum}), we deduce that there are at most $c^n$ candidates for our $B$. The upper bound on the closeness of the edges of $B$ arises from the length of the longest relator in the group-presentation of $G$. This is the aforementioned argument of Timar \cite{TimCut}.

Since Peierls' argument relies on the union bound, and many candidate cuts can occur simultaneously in a percolation instance, it is not good enough for our purposes because we need equalities rather than inequalities in formulas like \eqref{thetaQ}, where we add probabilities of events similar to the event that a cut as above occurs. To prove that $p_\C=p_c$ in the planar case we therefore considered the \scv\ rather than the cut separating $C(o)$ from infinity. An \scv\ consists of a connected (occupied) subgraph $I_O$ of $C(o)$, namely the boundary of its unbounded face, as well as the set $I_V$ of (vacant) edges disconnecting $I_O$ from infinity. 

\medskip
Most of the work of this section is devoted to combining these two ideas in the setup of a finitely presented \Cg\ $G$. We introduce a notion of interface $(I_V,I_O)$, generalising our interfaces from earlier sections, with the following properties. 
\begin{enumerate}
\item Given a percolation instance $\omega$, every finite cluster $C$ in $\omega$ is `bounded' by such an interface $(I_V,I_O)$, where
\item $I_V$ consists of the vacant edges separating $C$ from infinity, and 
\item $I_O$ defines a connected sub-cluster of $C$, incident with all edges in $I_V$.
\end{enumerate}
So far this is trivial to satisfy, as we could have taken $I_O=C$. But we need \begin{enumerate}
\setcounter{enumi}{3}
\item the size of $I_V$ to be proportional to that of $I_O$
\end{enumerate} 
in order to use a Peierls-type argument, so we need $I_O$ to be a `thin' layer near the boundary $I_V$ of $C$. In addition, we need 
\begin{enumerate}
\setcounter{enumi}{4}
\item $(I_V,I_O)$ to be unambiguously determined by $C$
\end{enumerate}
in order to express $\theta$ in an equality like \eqref{thetaQ} (see \eqref{thetaQ fp} below). Moreover, we need 
\begin{enumerate}
\setcounter{enumi}{5}
\item the event that $(I_V,I_O)$ is an interface of some cluster in a percolation instance to depend on the state of the edges in $I_V \cup I_O$ only,
\end{enumerate} 
in order to have a formula (of the form $p^{|I_O|} (1-p)^{|I_V|}$) for the probability of this event that we can do our complex analysis with. (Some complications here are imposed by the fact that we will use an inclusion-exclusion formula as above.) 

Finally, we want $I_V \cup I_O$ to span a connected subgraph of some power $\G^k$ of \G\ in order to guarantee that there are at most exponentially many `candidate' interfaces of $C(o)$, as in Timar's aforementioned proof. But we will be able to instead obtain a stronger statement by just letting $k=1$ with no additional effort:
\begin{enumerate}
\setcounter{enumi}{6}
\item $I_V \cup I_O$ to span a connected subgraph of \G.
\end{enumerate} 

Satisfying all these properties at once is non-trivial, as we need the balance of choices between too large and too small subgraphs of $C \cup \partial C$ to stabilise at a uniquely determined middle. After some preliminaries, we offer our notion of interface in \Dr{sec D pint}, followed by proofs of the aforementioned properties. We then exploit our notion to prove our analyticity results in \Sr{pint analyt}.

\medskip
{\em The reader wishing to get a feeling of the results of this section without all their combinatorial details may do so by reading \Sr{sec c c} up to \Dr{def con dir}, \Sr{sec D pint}, the statement of \autoref{uniq pint}, perhaps the proof of \autoref{prop b con}, and as much of \Sr{pint analyt} needed to be convinced that the above proof ideas can be carried out along the lines of the proof of the planar case. }

\subsection{The setup and notation} \label{sec 81}

The \defi{edge space}  of a graph \g is the direct sum $\ce(G):= \bigoplus_{e\in E(G)} \Z_2$, where   $\Z_2=\{0,1\}$ is the field of two elements, which we consider as a vector space over $\Z_2$. 
The cycle space $\cc(G)$ of \g is the subspace of $\ce(G)$ spanned by the \defi{circuits} of cycles, where a circuit is an element  $C\in \ce(G)$ whose non-zero coordinates $\{e\in E(G) \mid C_e=1\}$ coincide with the edge-set of a cycle of \G.

Let $P=\left< \cs \mid \cgr\right>$ be a group presentation, and let $\g=\text{Cay}(P)$ be the corresponding Cayley graph.
Let $\cp$ be the set of closed walks of \g induced by the relators in \cgr. It is straightforward to prove that \cp\ forms a basis of the cycle space $\cc(G)$ of \G.

More generally, we can let \g be an arbitrary graph, and let $\cp$ be any basis of $\cc(G)$. For the applications of the theory developed in this section to percolation it will be important for \g to be of bounded degree and 1-ended, and for the elements of \cp\ to have a uniform upper bound on their size. 

We will assume for simplicity that all elements of \cp\ are cycles (rather than more general closed walks with self-intersections); this assumption comes without loss of generality.


We let $vw=wv$ denote the edge of \g joining two vertices $v$ and $w$.
Every edge $e=vw\in E(G)$ has two \defi{directions} $\ar{vw},\ar{wv}$, which are the two directed sets comprising $v,w$. The head $head(\ar{vw})$ of $\ar{vw}$ is $w$.

For $F\subset E(G)$ we let $\dar{F}$ denote the set of directions of the edges of $F$. Thus $|\dar{F}|=2|F|$. In particular, $\dar{E(G)}$ denotes the set of directed edges of \G.

A percolation instance is an element \oo\ of $\OO=\{0,1\}^{E(G)}$.

\subsection{A connectedness concept} \label{sec c c}

We say that $(B_1,B_2)$ is a \defi{\pb}  of a set $B$, if $B_1\cup B_2=B$ and $B_1\cap B_2=\emptyset$ and $B_1,B_2 \neq \emptyset$.

Recall that Timar's argument involved the idea that the edges of the cut $B$ separating $C(o)$ from infinity form a connected auxiliary graph. This can be reformulated by saying that for every proper bipartition $(B_1,B_2)$ of $B$, there are edges $b_1\in B_1, b_2\in B_2$ that are `close' to each other. The measure of closeness used was that there is a relator in the presentation inducing a cycle containing both (in particular, $b_1, b_2$ are then close in graph distance). We use a similar idea here, but for technical reasons we need to reformulate this in the language of directed edges. 

A \defi{\cp-path connecting} two directed edges $\ar{vw}, \ra{yx} \in \dar{E(G)}$ is a path $P$ of \g \st\ the extension $vw P yx$ is a subpath of an element of \cp. Here, the notation $vw P yx$ denotes the path with edge set $E(P) \cup \{vw, yx\}$, with the understanding that the endvertices of $P$ are $w,y$. Note that $P$ is not endowed with any notion of direction, but the directions of the edges $\ar{vw}, \ra{yx}$ it connects do matter. We allow $P$ to consist of a single vertex $w=y$.

We will say that $P$ \defi{connects} an undirected edge $e\in E(G)$ to $\ar{f}\in \dar{E(G)}$ (respectively, to a set $J\subset \dar{E(G)}$), if $P$ is a \cp-path connecting one of the two directions of $e$ to $\ar{f}$ (resp.\ to some element of $J$).

\begin{definition} \label{def con dir}
We say that a set ${J}\subset \dar{E(G)}$ is \defi{$F$-connected} for some $F\subset E(G)$, if \fe\ proper bipartition $ ({J_1},{J_2})$ of ${J}$, there is a \cp-path in $G-F$ connecting an element of ${J_1}$ to an element of ${J_2}$.
\end{definition}


As usual, a notion of `connectedness' gives rise to a corresponding notion of `components'. 
In our case, an \defi{$F$-component} of any set ${K}\subset \dar{E(G)}$ is a maximal $F$-connected subset of ${K}$.
It is an immediate consequence of the definitions that if two sets ${J},{J'} \subset \dar{E(G)}$ are both $F$-connected, and their intersection is non-empty, then ${J} \cup {J'}$ is $F$-connected too. Therefore, 
\labtequ{com part}{the $F$-components of ${K}$ form a partition of any ${K}\subset \dar{E(G)}$. }
This implies the following monotonicity property of $F$-components.

\begin{proposition} \label{BFJ}
If $Y \subset \dar{E(G)}$ is contained in an $F$-component of some $J \subset \dar{E(G)}$ (with $J\supseteq Y$), then $Y$ is 
 contained in an $F'$-component of $J'$ whenever $F' \subseteq F$ and $J'\supseteq J$.
\end{proposition}
\begin{proof}
If $Y$ is not contained in an $F'$-component of $J'$, then in particular $J'$ is not $F'$-connected. As $J'$ is partitioned by its $F'$-components by \eqref{com part}, we can then find a \pb\ $(J'_1,J'_2)$ of $J'$ such that both $Y \cap J'_1$ and $Y \cap J'_2$ are non-empty and \ti\ no \cp-path in $G-F'$ connecting  $J'_1$ to $J'_2$. Consider then the bipartition  $(J'_1 \cap J,J'_2\cap J)$ of $J$, which is proper since both sides meet $Y$. As $Y$ is contained in an $F$-component of $J$, \ti\ a \cp-path $P$ in $G-F$ connecting  $J'_1\cap J$ to $J'_2\cap J$. But $P\subset G-F'$ since $F' \subseteq F$, and it connects $J'_1$ to $J'_2$, contradicting our assumption.
\end{proof}

It is easy to see that 
\labtequ{heads}{
if ${J}$ is $F$-connected, then there is a component of $G-F$ containing the head of every element of ${J}$. 
}

\subsection{\cp-Interfaces} \label{sec D pint}

Given $F\subset E(G)$ and a subgraph $D$ of $G$, let $\ard{F}{D}:= \{\ar{vz} \mid vz\in F, z\in V(D) \}$. Thus if $f \in F \cap \partial D$ then \ard{F}{D} contains the direction of $f$ towards $D$ only, if $f \in F \cap E(D)$ then \ard{F}{D} contains both directions of $f$, and otherwise \ard{F}{D} contains no direction of $f$. Fix a vertex $o\in V(G)$.

We now give the crucial definition of this section, following the intuition sketched in the beginning of this section.

\begin{definition} \label{def pint}

A \defi{\pint} is a pair $I= (I_V,I_O)$ of sets of edges of \g with the following properties
\begin{enumerate}
\item \label{pint a} $I_V$ 
separates $o$ from infinity; 
\item \label{pint x} There is a unique finite component $D$ of $G-I_V$ containing a vertex of each edge in $I_V$;
\item \label{pint c} 
\ard{I_V}{D} is $I_V$-connected; and\\
(Note that by \ref{pint x}, \ard{I_V}{D} contains at least one of the two directions of each edge in $I_V$. It may contain both directions of some edges.)
\item \label{pint b} $I_O = \{e \in E(D) \mid \text{ \ti\ a \cp-path in $G-I_V$ connecting $e$ to $\ard{I_V}{D}$ } \}$.\\
(This is equivalent to\\ $I_O = \{vz \in E(D) \mid \{\ar{vz}\} \cup \ard{I_V}{D} \text{ or } \{\ar{zv}\} \cup \ard{I_V}{D} \text{ is $I_V$-connected} \}$.)
\end{enumerate}
\end{definition}

Note that $I_V$ is always non-empty, but $I_O$ is empty when $I_V$ consists of the set of edges incident with $o$. It is not hard to see that $I_O\neq \emptyset$ for all other $I_V$ when \g is 1-ended. 

Clearly, $I_O$ is  determined by $I_V$ via \ref{pint b}, so any $I_V$ satisfying the other three properties introduces a \pint\ by defining $I_O$ via \ref{pint b}. The reason why we do not define $I_V$ alone to be the \pint\ is to satisfy the uniqueness property in \autoref{uniq pint} below. It follows from this definition that $I_V$ also separates $I_O$ from infinity.

\medskip

{\bf Examples:} if \cp\ is the standard presentation $\left<x,y \mid xy=yx \right>$ of $\Z^2$, then the \pint s coincide with the \scv s from \Sr{sec th pl}. 

An important aspect of the definition of a \pint\ is that (vacant) edges with both endvertices in the same cluster need to be accepted in $I_V$ to satisfy \autoref{port pint}. This is why in \ref{pint a} $I_V$ is declared to be a superset of a $o$--$\infty$~cut $B$, rather than $B$ itself. It is a good exercise to try to visualise a \pint\ of the standard presentation of $\Z^3$, i.e.\ the cubic lattice in $\R^3$ presented by its 4-cycles. A further good exercise is to try to visualise how \pint s of  $\Z^2$ or $\Z^3$ grow as we allow further (redundant) relators in our presentation, e.g. all cycles up to a given length.
\medskip

\subsection{Properties of \pint s} \label{pint props}

We now prove that the notion of \pint\ we introduced satisfies the many properties needed in order to carry out the Peierls-type argument sketched at the beginning of this section.

\medskip
From now on we assume that
\labtequ{assumptions}{\g is an infinite, 1-ended, finitely presented \Cg\ fixed throughout, or more generally, an 1-ended 
bounded degree graph, admitting a basis  \cp\ of $\cc(G)$ whose elements are cycles of bounded lengths (as discussed in \Sr{sec 81}).}
We say that a \pint\ $I= (I_V,I_O)$ \defi{occurs} in a percolation instance $\oo \in \{0,1\}^{E(G)}$, if every edge in $I_O$ is occupied and every edge in  $I_V$ is vacant in \oo. 

We say that $I$ \defi{meets} a cluster $C$ of \oo, 
if either $I_O \cap E(C)\neq \emptyset$, or $I_O=E(C)=\emptyset$ and $I_V= \partial C$ (in which case $C$ consists of $o$ only).

\begin{theorem} \label{uniq pint}
For every finite percolation cluster $C$ of \g such that $\partial C$ separates $o$ from infinity, \ti\ a unique \pint\ $(I_V,I_O)$ that meets $C$ and occurs. Moreover, we have $I_O\subseteq E(C)$ and $I_V\subseteq \partial C$ for that \pint.

Conversely, every occurring \pint\ meets a unique percolation cluster $C$, and $\partial C$ separates $o$ from infinity (in particular, $C$ is finite).
\end{theorem}

The proof of this is rather involved, and needs some intermediate steps which we gather now.
We remark that the assumption of bounded lengths of the elements of $\cp$ is not needed for the proof of \Tr{uniq pint}; it will only be used in the next section.

\medskip
The following proposition is based on Timar's \cite{TimCut} aforementioned proof of the theorem of Babson \& Benjamini \cite{BaBeCut}, and contains the quintessence of the notion of a \pint.

A \defi{minimal cut} of \g is a minimal set of edges that disconnects \G. 
Note that if $B$ is a minimal cut, then $G-B$ has exactly two components, and every edge in $B$ has an end-vertex in each of these components.
\begin{proposition} \label{prop b con}
Let $B$ be a minimal cut of \g and let $L\subset E(G)$ be a superset of $B$ \st\ some component $D$ of $G - L$ contains a vertex of each edge in $B$. Then \ard{B}{D}\ is contained in an $L$-component of \ard{L}{D}.
\end{proposition}
\begin{proof}

Suppose to the contrary that there are directed edges $e,f\in \ard{B}{D}$ that lie in distinct $L$-components of \ard{L}{D}. Note that $e,f$ cannot be the two directions of the same undirected edge because no edge of $B$ has both end-vertices in $D$ by the above remark about minimal cuts. Let 
$(L_1,L_2)$ be a \pb\ of \ard{L}{D} such that $e\in L_1, f\in L_2$, and there is no \cp-path in $G-L$ connecting $L_1$ to $L_2$, which exists by the definitions and the fact that \ard{L}{D} is partitioned by its  $L$-components by \eqref{com part}.

Let $R$ be an $e$-$f$~path in $D$, which exists because $D$ is assumed to contain a vertex of each edge in $B$. Let $Q$ be an $e$-$f$~path in the component of $G -B$ avoiding $D$; this component exists because $G-B$ has exactly two components, one of which contains $D$ since $L\supseteq B$ (Figure~5). 


\begin{figure}
\centering{
 \includegraphics[scale=.65]{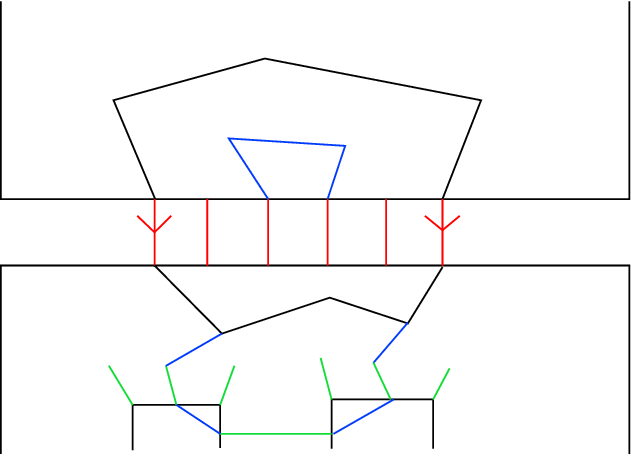}
\put(-65,40){$R$}
\put(-155,104){$Q$}
\put(-158,66){$e$}
\put(-67,64){$f$}
\put(-92,65){\textcolor{myred}{$B$}}
\put(-90,22){\textcolor{mygreen}{$L$}}
\put(-150,3){\textcolor{myblue}{$C_i$}}
\put(-181,30){$D$}
\caption{The situation in the proof of \Prr{prop b con}.}}
\label{figBL}
\end{figure} 

Let $K$ 
be the cycle obtained by joining these paths $R,Q$ using $e$ and $f$. Since \cp\ is a basis for the cycle space $\cc(G)$, we can express $K$ as a sum $\sum C_i$ of cycles $C_i\in \cp$, where this sum is understood as taking place in $\cc(G)$.

Note that no cycle $C_i$ contains a path in $G-L$ connecting $L_1$ to $L_2$, because no such path exists by the choice of $(L_1,L_2)$. 
Let $L_{C_i}:= \overleftrightarrow{L \cap E(C)}$ be the directions of edges of $L$ appearing in $C_i$.
The previous remark implies that $L_{C_i}$ has an even number of its elements in each of $L_1,L_2$, because each component of $C_i- L$ (which is a subpath of $C_i$) is incident with either 0 or 2 such elements pointing towards the component, and they lie both in $L_1$ or both in $L_2$ or both in none of the two.



This leads into a contradiction by a parity argument: notice that our cycle $K$  contains an odd number of directions of edges in each of $L_1,L_2$, namely exactly one in each ---$e$ and $f$ respectively--- because $P$ avoids $L$ and $Q$ avoids $D$, hence $\ard{L}{D}$, by definition. But then our equality $K=\sum C_i$ is impossible by the above claim because sums in $\cc(G)$ preserve the parity of the number of (directed) edges in any set. This contradiction proves our statement.
\end{proof}

We can use the same ideas to prove the following proposition.

\begin{proposition} \label{prop newer}
Let $L\subseteq E(G)$, let $D$ be a component of $G- L$, and let $e=vz$ be an edge of $L$ \st\ $v,z\in V(D)$. Then $\ar{vz},\ar{zv}$ lie in the same $L$-component of $\ard{L}{D}$.
\end{proposition}
\begin{proof}
It is not hard to adapt the proof of \Prr{prop b con} to our setup to prove our statement; the only difference is that instead of the cycle $K$ we now consider a cycle consisting of the edge $vz$ and a \pth{v}{z}\ in $D$. But we can in fact just apply \Prr{prop b con} to an auxiliary graph to deduce \Prr{prop newer} as follows. Subdivide the edge $vz$ into two edges $vw,wz$ by adding a new vertex $w$. Consider the minimal cut $B$ of the resulting graph that consists of these two edges $vw,wz$ (and separates $w$ from the rest of \G). Applying \Prr{prop b con} to this graph after replacing $L$ with $L':= L - vz \cup \{vw,wz\}$ we deduce that $\ar{wz},\ar{wv}$ lie in the same $L'$-component of $\ard{L'}{D}$, and it is straightforward to deduce that $\ar{vz},\ar{zv}$ lie in the same $L$-component of $\ard{L}{D}$ from this.
\end{proof}

Next, we prove one of the desired properties of \pint s, namely that $I_V \cup I_O$ spans a connected subgraph of \G.

\begin{proposition} \label{prop Io}
\Fe\ \pint\ $I= (I_V,I_O)$ of \G, the edge-set $I_O$ spans a connected subgraph of \g incident with all edges in $I_V$, unless $I_O = \emptyset$ (in which case $I_V$ is the set of edges incident with $o$).
\end{proposition}
\begin{proof}
Let $D$ be defined as in \ref{pint x}  of \Dr{def pint}. By \ref{pint b} of \Dr{def pint}, for every $e\in I_O$ \ti\ a \cp-path $P$ in  $G-I_V$ connecting $e$ to the head of an element of $\ard{I_V}{D}$. Note that all edges of $P$ belong to $I_O$ as we can apply \autoref{pint b} to any of them, where we use the fact that since $P$ meets $D$, it is contained in $D$ because $D$ is a component of $G-I_V$. This means that every component of the graph $G_O \subseteq G$ spanned by the edges in $I_O$ contains the head of an element of $\ard{I_V}{D}$.

Therefore, if $G_O$ has more than one components, then these components define a proper bipartition $ ({J_1},{J_2})$  of $\ard{I_V}{D}$, by letting ${J_1}$ be the set of all ${j}\in \ard{I_V}{D}$ such that $head({j})$ lies in one of these components. Applying \Dr{def con dir} to this bipartition we obtain a contradiction, since for any \cp-path $P$ in $G-I_V$ connecting $j_1\in {J_1}$ to $j_2\in {J_2}$, all edges of $P$ lie in $I_O$ by the above remark, which implies that the heads of $j_1$ and $j_2$ lie in the same component of $G_O$. This proves that $G_O$ is connected as claimed.
\smallskip

Finally, if some $e\in I_V$ is not incident with $G_O$, then we can apply the same argument to the bipartition of $\ard{I_V}{D}$ one partition class of which consists of the one or two directions of $e$ that lie in $\ard{I_V}{D}$ (recall the remark after \ref{pint c} of \Dr{def pint}). If $I_O\neq \emptyset$, then this bipartition is proper because each component of $G_O$ is incident with an element of $\ard{I_V}{D}$ as we have proved, and we obtain a contradiction as above. 

If $I_O = \emptyset$, and there are at least two vertices $x,y$ of $D$ incident with  $I_V$, then we obtain a \pb\ of $I_V$ by letting one of the classes be the set of edges incident with $x$, say, and reach a contradiction with the same arguments. Thus all edges of $I_V$ are incident with a vertex $x$ of $D$ in this case, and  in order to satisfy \ref{pint a} $I_V$ must be the set of edges incident with $x=o$.
\end{proof}

We have now gathered enough tools to prove our main result about \pint s.
\begin{proof}[Proof of \Tr{uniq pint}]
{\bf Existence:} To begin with, given such a cluster $C$ we will find an occurring \pint\ $(I_V,I_O)$ such that $I_O\subseteq E(C)$ and $I_V\subseteq \partial C$.
For this, let
$$B:= \{e\in \partial C \mid \text{ \ti\ a path from ${e}$ to $\infty$ in $G-\partial C$}\}.$$  
This is the minimal subset of $\partial C$ separating $C$ from infinity. 

Fix an enumeration of the elements of \ard{B}{C} (this notation was introduced before \Dr{def pint}), and let $X_i, 1\leq i \leq |\ard{B}{C}|$ be the $\partial C$-component of $\dar{\partial C}$ containing the $i$th element of $\ard{B}{C}$ in that enumeration (the definition of $F$-components is given after \Dr{def con dir}). It will turn out that these components $X_i$ coincide with each other, but we cannot use this fact yet. Let ${J}:= \bigcup_i X_i$, and let $I_V$ be the corresponding undirected edges, that is, $I_V:=\{vw\in \partial C \mid \ar{vw} \in {J} \}$. 

\smallskip
We will start by proving that 
$I_V$ satisfies properties  \ref{pint a}, \ref{pint x} and \ref{pint c}, after which we can define $I_O$ via \ref{pint b} to ensure that $(I_V,I_O)$ is indeed a \pint.

To see that \ref{pint a} is satisfied, 
we recall that $B\subseteq I_V$ by the definitions, and we claim that $B$ 
%
%
separates $o$ from infinity. This is true because if $Q$ is an infinite path starting at $o$, then it has to contain an edge in  $\partial C$ by our assumption that $\partial C$ separates $o$ from infinity. The last such edge of $Q$ then lies in $B$ by the definitions. 
Thus all paths from $o$ to infinity  meet $B$, proving that \ref{pint a} is satisfied.

\medskip
It is easy to see that \ref{pint x} is satisfied by letting $D$ be the component of $G-I_V$ containing $C$, which exists since $I_V \subseteq \partial C$. Indeed, $C \subseteq D$ meets all edges in  $\partial C$, hence  all edges in  $I_V$.

\medskip
We will now check that \ard{I_V}{D} is $I_V$-connected, that is, \ref{pint c} is satisfied.  \Prr{prop b con} ---applied with  $L=I_V$, so that $D$ meets all edges in $B\subseteq I_V$ as remarked above--- yields that $\ard{B}{D}$ is contained in some $I_V$-component $X$ of \ard{I_V}{D}. 
We will prove that $X$ contains the other edges of \ard{I_V}{D} too. For this, recall that $X_i$ is a $\partial C$-component of $\dar{\partial C}$, and so $X_i$ is $\partial C$-connected by the definition of $\partial C$-components. We can reformulate this by saying that $X_i$ is (contained in) a $\partial C$-component of $X_i$. Recall that $J= \bigcup_i X_i$. Using \eqref{heads} with $F=\partial C$ we will show that $J\subseteq \ard{I_V}{D}$. Indeed, the component $C$ of $G- \partial C$ contains the head of an element of $X_i$ in $\ard{B}{C}$ by the definition of $X_i$, and so the head of every element of $J$ lies in $C$ by \eqref{heads}. Since $C\subseteq D$, we deduce $J\subseteq \ard{I_V}{D}$.  Plugging these facts into \Prr{BFJ} ---with $Y=X_i$--- we obtain that $X_i$ is contained in an $I_V$-component of $\ard{I_V}{D}$, because $X_i \subseteq \ard{I_V}{D}$ and $I_V\subseteq \partial C$. Since each $X_i$ meets $\ard{B}{D}$, which is contained in the $I_V$-component $X$, \eqref{com part} yields that $X$ contains $J= \bigcup_i X_i$. 

To conclude that \ard{I_V}{D} is $I_V$-connected, or in other words, that $X=  \ard{I_V}{D}$, it remains to show that if $e\in \ard{I_V}{D} - J$ then $e$ lies in $X$ as well. To see this, note that for any such $e= \ar{vz}$ the reverse direction $e':= \ar{zv}$ lies in $J$, because all edges of $I_V$ have at least one of their directions in $J$ by the definitions. Moreover, we have $z,v\in V(D)$ since $e,e' \in \ard{I_V}{D}$, where we used the fact that $J\subseteq \ard{I_V}{D}$. Thus \Prr{prop newer} ---with $L=I_V$--- yields that $e,e'$ lie in a common $I_V$-component of \ard{I_V}{D}. Using \eqref{com part} again, combined with the fact that $(e'\in) J\subseteq X$ proved above,
we deduce that $e\in X$ as desired. To summarize, we have proved that all elements of \ard{I_V}{D} lie in a common $I_V$-component $X$, in other words, \ard{I_V}{D} is $I_V$-connected, establishing \ref{pint c}.

\medskip
We proved above that $J\subseteq \ard{I_V}{D}$. Next, we claim that actually $\ard{I_V}{D}=J$, which will be used below. 
Suppose this is not the case, and consider the \pb\ $(J, \ard{I_V}{D} - J)$ of \ard{I_V}{D}. Since \ard{I_V}{D} is $I_V$-connected, \ti\ a \cp-path $P$ in $G-I_V$ connecting directed edges $e\in J$ to $f\in \ard{I_V}{D} - J$. Let $g$ be the first edge of $P$ that lies in $\partial C$, directed towards $e$, if such an edge exists, and let $g=f$ otherwise. In both cases, the subpath $P'$ of $P$ from $e$ to $g$ avoids $\partial C$, and hence proves that $e$ and $g$ lie in a common  $\partial C$-component of $\dar{\partial C}$. But then $g$ must lie in $J$ since $J$ is a union of $\partial C$-components of $\dar{\partial C}$. This contradicts that $g\not\in J$ when $g=f$ and  $g\not\in I_V$ otherwise. This contradiction proves that $\ard{I_V}{D}=J$.

\medskip
Thus using \ref{pint b} of \Dr{def pint} to define $I_O$, we obtain a \pint\ $I:=(I_V,I_O)$. Since $I_V\subseteq \partial C$ which is vacant, to show that $I$ occurs it remains to show that $I_O$ is occupied in \oo. This is true because if $P$ is a \cp-path in $G-I_V$ connecting some edge $e$ of $I_O$ to $\ard{I_V}{D}=J$
, then the last vacant edge $f$ of the extended path $\{e\} \cup P$, if such an edge $f$ exists, would have to lie in $I_V$ by the definitions and the fact that $\ard{I_V}{D}=J$, contradicting that $\{e\} \cup P$ avoids $I_V$. Hence no such $f$ exists, and in particular any $e\in I_O$ is occupied as desired. 
Moreover, $I$ meets $C$ because $I_O\cup I_V$ spans a connected subgraph of \g by \autoref{prop Io}, and that subgraph contains $B$, hence meets $C$. 

\medskip
To prove the claim that $I_O\subseteq E(C)$, recall that $I_O$ spans a connected subgraph $G_o$ of \g by \autoref{prop Io}. This subgraph meets $C$ unless it is empty, because $G_o$ is incident with all of $I_V \supseteq B$, and it cannot meet the infinite component of $G- B$ as it is contained in $D$. Since $I_O$, being occupied, avoids $\partial C$, we deduce that $I_O\subseteq E(C)$ indeed.

\comment{
Finally, we check that \ref{pint b} is satisfied. Recall that we defined $I_O$ as the set of edges $e\in E(C)$ \st\ there is an $e$-${J}$ connecting \cp-path in $G-I_V$, but we want $I_O$ to coincide with the set $I'_O$ of edges  $e\in E(D)$ \st\ there is an $e$-$\ard{I_V}{D}$ connecting \cp-path in $G-I_V$. Firstly, we have already proved that $J = \ard{I_V}{D}$ and so any $e$-${J}$ connecting path is an $e$-$\ard{I_V}{D}$ connecting path, proving that $I_O\subseteq I'_O$ since $C \subset D$. It is more interesting to show that $I'_O\subseteq I_O$. For this, we assume that for some $e\in E(D)$ there is an $\ar{e}$-$\ard{I_V}{D}$ connecting \cp-path in $G-I_V$ for some direction $\ar{e}$ of $e$, and wish to show that $e\in E(C)$. Let $L:= \{\ar{e} \} \cup \ard{I_V}{D}$, and consider the bipartition $(L_C,L'_{C})$ of $L$ where $L_C:= \{j\in L \mid head(j)\in C\}$ and $L'_C:= \{j\in L \mid head(j)\in D-C\}$; this is a bipartition of $L$ because every element of $L$ has its head in $D$ by the definitions.

Recall that we have proved that \ard{I_V}{D}\ is $I_V$-connected, and so it is easy to prove that $L$ is $I_V$-connected too by the definitions. Thus if the above bipartition of $L$ is proper, then \ti\ a \cp-path $P$ in $G-I_V$ connecting its two sides. 
This $P$ contains at least one edge in $\partial C$ because $\partial C$ separates $C$ from $D-C$. Let $P'$ be an (initial or final) subpath of $P-\partial C$ connecting such an edge $yz\in \partial C\cap E(P)$ to a directed edge $f$ in the side of our bipartition that does not contain $\are$, and assume that $z$ is an end-vertex of $P'$. Since $f\in \ard{I_V}{D} =J$, we have $head(f)\in V(C)$, hence also $z\in V(C)$ since $P'$ avoids $\partial C$. As in the proof of $\ard{I_V}{D}=J$, we deduce that $\ar{yz}$ must lie in $J$ contradicting that $P$ avoids $I_V$.

This contradiction proves that our bipartition is not proper, and so in particular $head(e)\in V(C)$ as $\ard{B}{C}\subseteq J$. Thus if $e\not\in E(C)$ then $e\in \partial C$. In this case, let $Q$ be a \cp-path in $G-I_V$ connecting ${e}$ to $\ard{I_V}{D}$, which we assumed to exist. Let $e'$ be the last edge of the path $Q':= \{e\} \cup Q$ that lies in $\partial C$; it is possible that $e'=e$. But then the final subpath of $Q'$ starting at $e'$ is an ${e}$-$\ard{I_V}{D}$ connecting \cp-path in $G-\partial C$. Since  $\ard{I_V}{D}=J$, this path proves that $e'\in I_V$, contradicting that $Q'$ avoids $I_V$. This contradiction proves that $e\in E(C)$, and so $I'_O= I_O$. This establishes \ref{pint b} and completes the proof that   $(I_V,I_O)$ is a  \pint\ with the desired properties.
}



\bigskip

{\bf Uniqueness:} Suppose that our cluster $C$ is met by a further occurring \pint\ $I'=(I_V',I_O') \neq I$. By \Lr{prop Io} the subgraph of \g spanned by $I_O'\cup I_V'$ is connected, and therefore contained in $C \cup \partial C$ since $I_O'$ meets $E(C)$. It follows that $I_V' \subseteq \partial C$ since $I'$ occurs. 

Let $D'$ be the component of $G - I_V'$ defined in \ref{pint x}. We claim that $B \subset I_V'$. 
Indeed, if $I_V'$ misses some edge of $B$, then $I_V'\subseteq \partial C$ does not separate $C$ from infinity, hence $C \cap D' = \emptyset$, contradicting that $I_O' \subseteq E(D')$ and $I_O'\cap E(C)\neq \emptyset$ unless $E(C)=\emptyset$, in which case $I_V'$ cannot separate $o$ from infinity violating \ref{pint a}.

Moreover,  we have $D' \supseteq C $ since  $I_V' \subseteq \partial C$ (because $I'$ occurs) and $I_O'$ meets $E(C)$.

We will first prove that $I_V' \subseteq I_V$. So let $f \in I_V'$, and suppose for a contradiction that $f \not\in I_V$. 
In this case, the  bipartition $(J, J':= \dar{\partial C} - J)$ of $\dar{\partial C}$, where $J$ is as in the definition of $I_V$ in the existence part, is \st\ $\ard{B}{C}\subseteq J$ and both directions $\ar{f},\ra{f}$ of $f$ lie in $J'$ and there is no \cp-path in $G- \partial C$ connecting $J$ to $J'$.

Consider now the bipartition $(J\cap \ard{I_V'}{D'}, J'\cap \ard{I_V'}{D'})$ of \ard{I_V'}{D'}, which is proper because $\ard{B}{C}\subseteq \ard{I_V'}{D'}$  (because $D'\supseteq C$ and $B \subset I_V'$) and $\{\ar{f},\ra{f}\}\cap \ard{I_V'}{D'}\neq \emptyset$ (by the definition of $D'$). Therefore, since $\ard{I_V'}{D'}$ is $I_V'$-connected by \ref{pint c}, there is a \cp-path $P$ in $G-I_V'$ connecting $J\cap \ard{I_V'}{D'}$ to $J'\cap \ard{I_V'}{D'}$. Let $e$ be the last edge of $P$ in $\partial C$, which exists because $P$ cannot avoid $\partial C$ by the aforementioned property of the bipartition $(J, J')$, and let $P'$ be the final subpath of $P$ starting at $e$. But then applying \ref{pint b} to $I'$ using the path $P'$ we deduce that $e\in I_O'$, contradicting that $I'$ occurs and $e\in \partial C$ is vacant. This contradiction proves that $I_V' \subseteq I_V$.

\medskip
Next, we prove that $I_V \subseteq I_V'$ as well. Indeed, if $I_V \not\subseteq I_V'$, then the bipartition $(\ard{I_V}{D} \cap \dar{I_V'}, \ard{I_V}{D} - \dar{I_V'})$ of \ard{I_V}{D} is proper because $\ard{B}{C} \subseteq \ard{I_V}{D} , \dar{I_V'}$. Since \ard{I_V}{D} is $I_V$-connected, \ti\ a \cp-path $P$ in $G- I_V$ connecting some edge $f\in {I_V'}$ to some edge $e\in {I_V} - {I_V'}$. Since we have proved that $I_V' \subseteq I_V$, we deduce that $P$ lies in $G- I_V'$. But then applying \ref{pint b} to $I'$ using the path $P$ we deduce that $e\in I_O'$, contradicting that $I'$ occurs and $e\in I_V \subseteq \partial C$ is vacant. This contradiction proves that $I_V \subseteq I_V'$, and hence $I_V'= I_V$.

\medskip
To conclude that $I$ is the unique occurring \pint\ that meets $C$, it remains to prove that $I_O'=I_O$. But this is now obvious from \ref{pint b}, since $I_V'= I_V$ and hence $D'=D$ by \ref{pint x}.

\bigskip
{\bf Converse:} Suppose now that $(I_V,I_O)$ is a \pint\ occurring in a percolation instance \oo. Then by \Lr{prop Io} it meets a unique cluster $C$ of \oo, and we have $I_V \subseteq \partial C$ by what we proved above. By \ref{pint a} $I_V$, and hence  $\partial C$, separates $o$ from infinity.
\end{proof}

\comment{
\begin{proposition} \label{prop heads}
Suppose $\ar{J}\subset \dar{E(G)}$ is $F$-connected for some $F\subset E(G)$. Then there is a component of $G-F$ containing all the heads of $\ar{J}$, i.e.\ the set $\{z\in V(G) \mid \vec{vz} \in \ar{J}\}$.
\end{proposition}
\begin{proof}
Suppose to the contrary that $\ar{J}$ contains two edges whose heads lie in distinct components $C,C'$ of $G-F$. Consider the bipartition $ (\ar{J_1},\ar{J_2})$ of $\ar{J}$ where $\ar{J_1}$ contains all $\ar{j}\in \ar{J}$ such that the head of $\ar{j}$ lies in $C$. Since no path in $G-F$ can meet two components of $G-F$, this bipartition contradicts the assumption that $\ar{J}$ is $F$-connected.
\end{proof}

Given an $F$-connected set $\ar{J}\subset \dar{E(G)}$, \autoref{prop heads} allows us to define $C(\ar{J},F)$ to be the component of $G-F$ containing all the heads of $\ar{J}$. Let us say that 

\begin{proposition} \label{prop halfedges}
Let $J,F\subset E(G)$ be edge-sets \st\ $J$ is $F$-connected. Let $\ar{K}$ be the set of directions of edges in $J$  is $F$-connected.
\end{proposition}
\begin{proof}

\end{proof}
}

\subsection{Using \pint s to prove analyticity} \label{pint analyt}

Define the \defi{boundary size} of a \pint\ $I=(I_V,I_O)$ to be $|I_V|$. Note that every set $S$ of edges which is $S$-connected (according to \Dr{def con dir}) corresponds to a connected induced subgraph of the $m$th power of the line graph $L(G)^m$ of $G$, where $m=m_\cp:=\left \lfloor{t/2} \right \rfloor$ and $t$ is the length of the longest cycle in \cp.  
The degree of each vertex of $L(G)$ is at most $2d-2$, where $d$ is the maximum degree of  \G\ (we are still assuming that \g satisfies \eqref{assumptions}), and so the degree of each vertex of $L(G)^m$ is at most $(2d-2)^m$. Applying the remark after Corollary~\ref{lattice animals} to $L(G)^m$, combined with the fact that any \pint\ $I=(I_V,I_O)$ is determined by $I_V$ by the definitions, we thus deduce that 

\begin{lemma} \label{grow pint}
The number of \pint s $(I_V,I_O)$ of \g of boundary size $n$ such that $I_V$ contains a fixed edge of \g is less than $c \gamma_{\cp}^n$, where $c$ is a constant, $\gamma_{\cp}=((2d-2)^{m_\cp}-1)e$,  and $d$ is the maximum degree of \G.
\end{lemma}

The following is the analogue of \autoref{scs boundary}. 

\begin{lemma} \label{port pint}
\Fe\ \pint\ $I= (I_V,I_O)$ of \G, we have $|I_V| \geq  |I_O|/d^{t}$,  where $d$ is the maximum degree of \g and $t$ is the length of the longest cycle in \cp. \mymargin{can you improve this bound significantly?}
\end{lemma}
\begin{proof}
By \ref{pint b} of \Dr{def pint}, each $e\in I_O$ has distance less than $t$ from $I_V$ in the subgraph $G_I$ of \g spanned by $I_V \cup I_O$. Using this fact we can assign each $e\in I_O$ to an edge $f(e)$ of $I_V$ so that the distance between $e$ and $f(e)$ in $G_I$ is less than $t$. Then the number $|f^{-1}(g)|$ of edges of $I_O$ assigned to any $g\in I_V$ is at most the size of the ball of radius $t-1$ around $g$ in \G, which is at most $d^{t-1}$ since \g is $d$-regular. Thus $|I_V| \geq  |I_O|/d^{t-1}$ by the pigeonhole principle.
\end{proof}

Let $R=\ldots,r_{-1},r_0,r_1,\ldots$ be 2-way infinite geodesic with $r_0 = o$ (such a geodesic exists in every Cayley graph by an elementary compactness argument, provided we assume e.g.\ the Axiom of Countable Choice). Let $f_i$ denote the edge $r_i r_{i+1}$ of $R$. 

\begin{lemma}\label{fp scs axis}
For every \pint\ $I=(I_V,I_O)$ of $o$ with boundary size $|I_V|=n$, the set $I_V$ contains at least one of the edges $f_{0}, f_{1}, \ldots, f_{d^t n - 1}$.
\end{lemma}
\begin{proof}
Each of the two 1-way infinite subpaths of $R$ starting at $o$ connects $o$ to infinity, so $I_V$ must contain an edge from each of them. By \autoref{prop Io} and \autoref{port pint}, $I_O$ is connected, incident to both of these edges, and $|I_O|\leq d^t n$. Thus if $I_V$ contains some edge $r_i r_{i+1}$ with $i\geq d^t n$, then it cannot meet $\ldots r_{-1} r_0$ because $R$ is a geodesic.
\end{proof}

A \mpint\ $S$ is a finite set of \pint s $\{(I_V^i,I_O^i)\}_{1\leq i \leq k}$ such that the corresponding graphs $G_O^i$, i.e.\ the subgraphs of \g spanned by the edges in $I_O^i$, are pairwise vertex disjoint. Define the \defi{boundary} $\partial S$ of $S$ to be $\bigcup_{1\leq i \leq k} |I_V^i|$.
Let $\MS$ denote the set of \mpint s and $\MS_n$ the set of \mpint s of total boundary size $n$.
Using the above lemma and \autoref{fp scs axis} we can upper bound the number of elements of $\MS_n$ that can occur simultaneously in any \oo\ similarly to the proof of \autoref{HR bound}.

\begin{lemma}\label{fp HR bound} 
There is a constant $x\in \R$ \st\ \fe\ $n\in \N$ at most $x^{\sqrt{n}}$ elements of $\MS_n$ can occur simultaneously in any \oo.
\end{lemma}

We will now prove that $p_\C<1$ for every finitely presented \Cg\ following the approach of \Sr{sec proof th}, replacing the use of exponential decay of the dual by \Lrs{grow pint} and \ref{port pint}. 

\begin{theorem} \label{thm pc fp}
Let \g be an 1-ended  \Cg\ with a finite presentation \cp. Then $p_\C\leq 1-1/\gamma_{\cp}$ for bond percolation on \G. 
\end{theorem}
\begin{proof}
Similarly to \eqref{thetaQ}, we claim that
\labtequ{thetaQ fp}{$1-\theta_o(p)= \sum_{S\in \MS} (-1)^{c(S)+1} Q_S(p)$}
\fe\  $p\in (q,1]$, where $c(S)$ denotes the number of \pint s in the \mpint\ $S$, and $Q_S(p):= \Pr_p(\text{$S$ occurs})$. \\
We will use \autoref{uniq pint} to prove that the above formula holds. 
By  that proposition, $C(o)$ is finite \iff\ it meets a \pint. Since for any pair of distinct occurring \pint s the graphs $G_O$ do not share a vertex, the inclusion-exclusion principle yields 
$$1-\theta_o(p) = \Pr(\text{ at least one \pint\ occurs }) = \sum_{S\in \MS} (-1)^{c(S)+1} Q_S(p)$$
provided the latter sum converges absolutely.

Once again $$\sum_{S\in \MS_n} Q_S(p)=\mathbb{E}_p\Big(\sum_{S\in \MS_n}\mathbb{\chi}_{\{S \text{ occurs}\}}\Big)$$ and by \autoref{fp HR bound} we conclude that 
$$\sum_{S\in \MS_n} Q_S(p)\leq x^{\sqrt{n}} \Pr_p(\text{some } S\in\MS_n \text{ occurs}).$$
The event $\{\text{some } S\in\MS_n \text{ occurs}\}$ implies that a set of edges with certain properties is vacant and our goal is to use Peierls' argument to conclude that the probability of 
the latter event decays exponentially for large enough $p$.

Let $S\in \MS_n$ and let $X_1,X_2,\ldots,X_k$ be the components of  the subgraph of $L(G)^m$ spanned by $\partial S$, where $m= \left \lfloor{t/2} \right \rfloor$. By the argument at the beginning of \Sr{pint analyt}, each $X_i$ contains the boundary of a \pint\ of size at most $n_i:=|X_i|$. Thus by \autoref{fp scs axis}, $X_i$ contains one of the edges of $f_{0}, f_{1}, \ldots, f_{d^t n_i - 1}$. The \HRT\ and \autoref{grow pint} now easily yield that the number of all possible boundaries of $\MS_n$ is at most $$r^{\sqrt{n}}\max\{c^k d^{kt} n_1 n_2\ldots n_k\} \gamma_{\cp}^n,$$ where the maximum ranges over all partitions $\{n_1,n_2,\ldots,n_k\}$ of $n$ such that every $N$ appears at most $d^t N$ times. As in the proof of \Tr{theta analytic 2D}, it is easy to check that the quantity $\max\{c^k d^{kt} n_1 n_2\ldots n_k\}$ grows subexponentially in $n$. Since each $S\in \MS_n$ occurs with probability at most $(1-p)^n$ by the definitions,  we conclude that 
\labtequ{msn fp}{$\Pr_p(\text{some } S\in\MS_n \text{ occurs})\leq r^{\sqrt{n}} \max\{c^k d^{kt} n_1 n_2\ldots n_k\} \gamma_{\cp}^n (1-p)^n$,} 
and thus $\Pr_p(\text{some } S\in\MS_n \text{ occurs})$ decays exponentially for every $p>1-1/\gamma_{\cp}$.

Finally, combining this exponential decay with \autoref{cor general} and\\ \autoref{port pint} we deduce that $\theta$ is analytic in $(1-1/\gamma_{\cp},1]$, arguing as in the end of the proof of \Tr{theta analytic 2D}.
\end{proof}

\subsection{Extending to site percolation} \label{sec site}

In this section we extend \Tr{thm pc fp s} to site percolation. The proof is essentially the same, all we have to do is to adapt the probability $(1-p)^n$ appearing in \eqref{msn fp}, but we will also adapt \Lr{grow pint} in order to obtain a better bound on $p_\C$.

For a \pint\ $(I_V,I_O)$ of \g we let $V_O$ denote the set of vertices incident with an edge in $I_O$, and we let $V_V$  denote the set of vertices incident with an edge in $I_V$ but with no edge in $I_O$.
We say that a \pint\ $I=(I_V,I_O)$ is a \defi{site-\pint}, if no edge in $I_V$ has both its end-vertices in $V_O$. Note that any site percolation instance $\oo\in \{0,1\}^{V(G)}$ naturally gives rise to a bond percolation instance $\oo'\in \{0,1\}^{E(G)}$, by setting $\oo'(xy)=1$ whenever $\oo(x)=1$ and $\oo(y)=1$. It is obvious from the definitions that if $I$ occurs in such an $\oo'$, then $I$ is a site-\pint. For site-\pint s we can improve \Lr{grow pint} as follows, using the same proof except that we work with \g rather than $L(G)$. The \defi{vertex-boundary size} of $(I_V,I_O)$ is $|V_V|$.

\begin{lemma} \label{grow pint s}
The number of site-\pint s $(I_V,I_O)$ of \g of vertex-boundary size $n$ such that $I_V$ contains a fixed edge of \g is less than $c' \dot{\gamma}_{\cp}^n$, where  $\dot{\gamma}_{\cp}=(d^{m_\cp}-1)e$,  and $d$ is the degree of \G.
\end{lemma}

Using this we can now adapt \Tr{thm pc fp} to site percolation, repeating the proof verbatim, except that we use site-\pint s instead of \pint s.
\begin{corollary} \label{thm pc fp s}
Let \g be an 1-ended  \Cg\ with a finite presentation \cp. Then  $p_\C\leq 1-1/\dot{\gamma}_{\cp}$ for site percolation on \G.
\end{corollary}

This bound on $p_\C$ is far from the conjectured $p_\C=p_c$, but not so far from  $p_\C\leq 1-p_c$, which is the best that our methods can achieve (and possibly the truth) in light of a result of Kesten \& Zhang, saying that for site percolation on $\Z^d, d\geq 3$,  the distribution of the  vertex-boundary size of the site-\pint\ of the cluster of the origin does not have an exponential tail \cite[Theorem~4]{KeZhaPro} (here $\Z^d$ denotes the cubic lattice in $\R^d$, and the basis \cp\ consists of the squares bounding the faces of its cubes). Our next result implies that this `theoretical barrier'  $p_\C\leq 1-p_c$ can in fact be achieved if we are allowed to modify the graph a little by adding some diagonal edges.

\begin{theorem} \label{thm pc triang} 
Let \g be an 1-ended quasi-transitive graph admitting a basis $\cp$ of $\cc(G)$ all cycles of which are triangles. Then $p_\C\leq 1-p_c$ for both site and bond percolation on \G. In particular, we have $p_c\leq 1/2$. 
\end{theorem}

For example, we can obtain such a \g by adding to $\Z^d$ the `monotone' diagonal edges, i.e.\ the edges of the form $xy$ where $y_i-x_i=1$ for exactly two coordinates $i\leq d$, and $y_i=x_i$ for all other coordinates. Then each square gives rise to two triangles, and we can use all these triangles as our basis \cp\ of the cycle space. 

Note that for $d=2$ we obtain the triangular lattice, and so \Tr{thm pc triang} can be thought of as a generalisation of \Cr{triangular}. 

For its proof we will need the following lemma, which is a special case of \cite[Theorem~5.1]{TimCut}, the main idea of which we used in \Prr{prop b con}, as illustrated in (Figure~5). 

\begin{lemma} \label{spint tr}
Let \g be an 1-ended quasi-transitive graph admitting a basis $\cp$ of $\cc(G)$ all cycles of which are triangles. Then \fe\ site-\pint\ $(I_V,I_O)$ of \G, the vertex boundary $V_V$ spans a connected subgraph of $G$.
\end{lemma}

\begin{proof}[Proof of \Tr{thm pc triang}]
We first prove the statement for site percolation.
We follow the lines of the proof of \Tr{theta analytic 2D}, except that we now let $\MS_n$ denote the set of \mpint s all elements of which are site-\pint s. Instead of \Lr{dual connd}, which states that the boundary of a \pint\ spans a connected subgraph of the dual lattice in that setup, we now use \Lr{spint tr}, which is the analogous statement for the boundary $V_V$ of a site-\pint\ under our assumption on \cp\ that all its cycles are triangles. The proof of \Lr{HR bound} can be repeated verbatim, except that we replace the quasi-geodesic $X$ used there with an arbitrary 2-way infinite quasi-geodesic of \G, which exists by a standard compactness argument. In that proof, we used the canonical coupling between bond percolation on a planar lattice and its dual, and applied the \ABP\ to the subcritical clusters of the dual. Here, we instead use the canonical coupling between site percolation with parameter $p$ and with parameter $1-p$ obtained by switching between vacant and occupied vertices. We apply the \ABP\ to the boundaries $V_V$ of our site-\pint s: since they span connected subgraphs of \g by \Lr{spint tr}, each such $V_V$ is contained in a cluster of vacant sites. But as $p>1-p_c$, vacant clusters are subcritical due to that coupling, hence their size distribution has an exponential tail by the \ABP\ (\Tr{AB thm}). The rest of the proof can be repeated as is.

\medskip
To prove the statement for bond percolation, we use the canonical coupling between bond percolation on \G\ and site percolation on its line graph $L(G)$, noting that the cluster of an edge of \g is infinite in the former if and only if the cluster of the corresponding vertex of $L(G)$ is infinite in the latter. Our plan is to apply the statement for site percolation we just proved to  $L(G)$. Note that if \g is quasi-transitive, then so is $L(G)$. Moreover, it is straightforward to check that we can obtain a basis of $\cc(L(G))$ from any basis \cp\ of \g by adding all the triangles of the form $x,y,z$ in $L(G)$ whenever the edges $x,y,z$ of \g are incident with a common vertex. Thus we can reduce to the case of site percolation as desired.
\medskip

For both site and bond percolation, since $p_c<1$, we have $p_c\leq p_\C$, hence we immediately obtain $p_c\leq 1/2$.
\end{proof}

\section{Triangulations}  \label{sec triang}

\subsection{Overview}

In this section we use the techniques we developed to provide upper bounds on $p_c$ and \pcs\ for certain families of triangulations. Although these bounds will apply to $p_\C$, we stress that the results of this section give the best known (or only) such bounds on $p_c, \pcs$ for these triangulations.

\medskip
We will prove that $p_{\C} \leq 1/2$ for Bernoulli bond percolation on triangulations of an open disk that either satisfy a weak expansion property or are transient. Once again the analyticity of $\theta_o$ will follow by showing that the \scv s (\pint s) of $o$ have an exponential tail for every $p>1/2$.  

The interest in the study of percolation on triangulations of an open disk was sparked by the seminal paper \cite{BeSchrPer} of Benjamini \& Schramm. They made a series of conjectures, the strongest one of which is that $\pcs(T)\leq 1/2$ on any bounded degree triangulation $T$ of an open disk that satisfies a weak isoperimetric inequality of the form 
$|\partial_V A| \geq f(|A|)) \log(|A|))$ for some function $f=\oo(1)$, where $S$ is any finite set of vertices. More recently, Benjamini \cite{Benjamini} conjectured that $\pcs(T)\leq 1/2$ on any transient bounded degree triangulation $T$ of an open disk. 

Angel, Benjamini \& Horesh \cite{AnBeHor} proved that for any triangulation $T$ of an open disk with minimum degree $6$, the isoperimetric dimension of $T$ is at least $2$ and thus satisfies the assumption of the conjecture of Benjamini \& Schramm. They also asked whether $p_c(T)\leq 2\sin(\pi/18)$ (and $\pcs\leq 1/2$), the critical value for bond percolation on the triangular lattice, for any such triangulation. 

The main results of this section, which we now state, imply that in all aforementioned conjectures, the bound $p_c\leq 1/2$ is correct if one considers bond instead of site percolation.

\begin{theorem}\label{pC triang}
Let $T$ be a triangulation of an open disc \st\ every vertex has finite degree (not necessarily bounded) and\footnote{The reader will lose nothing by replacing ${\mathrm{diam}}(A)$ by $|A|$ in this statement, which only strengthens our assumptions.}
\labtequ{partial A}{for all but finitely many sets $A$ of vertices we have\\ $|\partial_V A| \geq k \log(\mathrm{diam}(A))$ for some constant $k>0$.}
%
%
Then  there is a constant $\nu_k<1$ that converges to $1/2$ as $k$ goes to infinity, such that $$p_c(T)\leq p_{\C}(T)\leq \nu_k.$$

In particular, if 
\labtequ{partial A'}{for every finite set $A$ of vertices we have $|\partial_V A| \geq f(\mathrm{diam}(A)) \log(\mathrm{diam}(A))$ for some function $f=\oo(1)$,}
then
$$p_c(T)\leq p_{\C}(T) \leq 1/2.$$
(This holds in particular when $h(T)>0$, i.e.\ when $T$ is non-amenable.)
\end{theorem}

\begin{theorem}\label{pC transient}
Let $T$ be a transient triangulation of an open disc with degrees bounded above by $d$. Then $$p_c(T)\leq p_{\C}(T) \leq 1/2.$$ 
\end{theorem}

We will also prove the same bounds for recurrent triangulations $T$ with a uniform upper bound on the radii of the circles in any circle packing of $T$, as well as analogues for site percolation (\Sr{sec sit tri}). 

\subsection{Proofs}

Notice that any bounded degree triangulation $T$ is 1-ended and by definition admits a basis \cp\ of $\cc(G)$ whose elements are cycles of bounded length. Hence the arguments of \Sr{pint props} imply that $p_{\C}(T)<1$ for bond percolation provided we further assume that $T$ contains a 2-way infinite geodesic. However, the latter is a rather strong condition. But we only used the existence of a $2$-way infinite geodesic in the proof of \Lr{fp scs axis}, and it will turn out that a variant of that lemma still holds for transient triangulations and triangulations satisfying the above isoperimetric inequality.

We will first focus on proving \Tr{pC triang}, but many of the following arguments will also be valid for transient triangulations. 

Our proofs will follow the lines of that of \Tr{theta analytic 2D}.
Recall the definitions of \scv\ and \smc\ of \Sr{sec th pl}. Again $\MS$ denotes the set of \smc s of a chosen vertex $o$, while  $\partial M$ denotes the boundary of a \smc\ $M$ and $\MS_n:= \{ M \in \MS \mid |\partial M|=n\}$. 

Let $T$ be a triangulation of an open disk and $o$ a vertex in $T$. Once again we will utilise the inclusion-exclusion principle to express $1-\theta_o$ as an infinite sum 
\labtequ{thetaTriang}{$1-\theta_o(p)= \sum_{M\in \MS} (-1)^{c(M)+1} Q_M(p)$}
\fe\ $p$ large enough, where $c(M)$ denotes the number of \scv s in the \smc\ $M$, and $Q_M(p):= \Pr_p(\text{$M$ occurs})$. 
The validity of the formula will follow as in the proof of \Tr{theta analytic 2D} (recall \eqref{thetaQ}) once we establish an exponential tail  for the corresponding probabilities, which is the purpose of the following lemma.

\begin{lemma} \label{exp dec triang}
There is a constant $\nu_k<1$ that converges to $1/2$ as $k$ goes to infinity, such that 
\fe\ triangulation $T$ of an open disk satisfying condition \eqref{partial A} 
of \Tr{pC triang} and every $p\in (\nu_k,1]$,
\labtequ{dec Q triang}{$\sum_{M\in \MS_n} Q_M(p) \leq c_1 {c_2}^{n}$,} 
where $c_1=c_1(p)>0$ and $c_2=c_2(p)>0$ are some constants with $c_2<1$.
Moreover, if $[a,b]\subset (\nu_k,1]$, then the constants $c_1$ and $c_2$ can be chosen independent of $p$ in such a way that \eqref{dec Q triang} holds for every $p\in [a,b]$.
\end{lemma}

In order to prove the above lemma, we first pick an arbitrary infinite geodesic $R$ starting from $o$. Our goal is to show that the \scv s $M$ of $o$ for which $\partial M$ contains a fixed edge $e\in E(R)$, occur with exponentially decaying probability for every large enough value of $p$. Then we will upper bound the choices for $e\in R$.

In what follows we will be using the standard coupling between percolation on $T$ and its dual $T^*$ as in the proof of \Lr{exp dec Z2}. Since $T$ is a triangulation, the dual of any minimal cut of $T$ is a cycle. The number of cycles in $T^*$ of size $n$ containing a fixed edge is at most $2^{n-1}$, because $T^*$ is a cubic graph. Then the union bound shows that the probability that some minimal cut containing a fixed edge is vacant has an exponential tail for every $p>1/2$. However, the boundary of an \scv\ is not necessarily a minimal cut. Still, the dual of the boundary of any \scv\ in $T$ is a connected subgraph of $T^*$. The desired exponential tail will follow from \Tr{AB thm} once we show that $\sup_{u\in V(T^*)} \chi_u(p)<\infty$ for every $p<1/2$, where as usual $\chi_u(p)$ denotes the expected size of the percolation cluster of $u$. 
The next lemma proves the this statement.

\begin{lemma}\label{chi finite}
Let $T$ be a triangulation of an open disc. Then $$X^*(p):=\sup_{u\in V(T^*)} \chi_u(1-p)<\infty$$ for every $p\in (1/2,1]$. 
\end{lemma}
\begin{proof}
Let $u$ be a vertex of $T^*$. Note that whenever some vertex $v$ belongs to $C_u$ there is a path from $u$ to $v$ with occupied edges. Hence we obtain $\mathbb{E}_{1-p}(|C_u|)\leq \mathbb{E}_{1-p}(P(u))$, where
$P(u)$ is the number of occupied self-avoiding walks starting from $u$ (including the self-avoiding walk with only one vertex). The number $\sigma_k(u)$ of $k$-step self-avoiding walks in $T^*$ starting from $u$ is at most $3\cdot 2^{k-1}$. Consequently, 
\labtequ{chi finite bound}{$\mathbb{E}_{1-p}(P(u))\leq \sum_{k=0}^\infty 3\cdot 2^{k-1} (1-p)^k<\infty$} 
whenever
$p>1/2$. Since this bound does not depend on $u$ the proof is complete.
\end{proof}

Using \Tr{AB thm} we immediately obtain the desired exponential tail.

\begin{corollary}\label{ABP triang}
For every $p>1/2$ there is a constant $0<c=c(p)<1$ such that for any triangulation $T$ of an open disk and any vertex $u\in T^*$, we have $\Pr_{1-p}(|C(u)|\geq n)\leq c^n$.
\end{corollary}

The following lemma converts condition \eqref{partial A} into a statement saying that every \scv\ of $T$ meets a relatively short initial subpath of $R$.

\begin{lemma} \label{geodesic}	
Let $T$ be a triangulation of an open disk satisfying condition \eqref{partial A}. Let $R$ be a geodetic ray in $T$ starting at any $o\in V(G)$, and let $R_n$ be the set of edges of $R$ contained in some \scv\ of $\mathcal S_n$. Then $|R_n| \leq e^{n/k}$ for all but finitely many values of $n$.
\end{lemma}
\begin{proof}
Define a function $g: \N \to \N$  by letting $g(n)$ be the smallest integer $l$ \st\ every \scv\ of $\mathcal S_n$ contains at least one of the first $l$ edges of $R$ if such a $l$ exists, and let $g(n)=\infty$ otherwise, with the convention that $g(n)=1$ if no such edge-separator of size $n$ exists.

We need to show that $g(n) \leq e^{n/k}$ for almost every $n$ (in particular, $g(n)< \infty$). In other words, we need to show that $g(n) > e^{n/k}$ holds for only finitely many values of $n$. To see this, assume $n$ is such a number, which means that some \scv\ $M$ of $\mathcal S_n$ does not contain any of the first $e^{n/k}$ edges of $R$. Let $B$ be the minimal cut of $M$ and $A=A_n$ be the component of $o$ in $G - B$. Our condition \eqref{partial A} says that
$$k \log(\mathrm{diam}(A)) \leq |\partial_V A| \leq |\partial_E A| = |B| \leq n,$$ 
except possibly for finitely many sets $A=A_n$, hence for finitely many values of $n$. 

On the other hand, we have $\mathrm{diam}(A)> e^{n/k}$ since $A$ contains the first $e^{n/k}$ edges of the geodesic $R$. Combining these inequalities yields
the contradiction $k \log e^{n/k} >n$.

\end{proof}

An immediate consequence of \Lr{geodesic} is that $g(n)$ grows subexponentially in $n$, i.e. $\limsup_{n\to\infty} g(n)^{1/n}=1$, whenever the stronger condition \eqref{partial A'} is satisfied. 

Note that the constant $e^{-1/(5{X}^2)}$ involved in the statement of the theorem of Aizenman \& Barksy does not converge to $0$ as $p$ goes to $0$, because $X\geq 1$. Hence
when we combine \Cr{ABP triang} and \Lr{geodesic} with the union bound, we deduce that 
$\Pr_p(\text{some } M\in \MS_n \text{ occurs})$ decays exponentially in $n$ for every large enough value of $p$, only when $T$ satisfies \eqref{partial A} for some large enough value of $k$. In particular, when $T$ satisfies \eqref{partial A'}, then $\Pr_p(\text{some } M\in \MS_n \text{ occurs})$ decays exponentially in $n$ for every $p>1/2$.

To cover the remaining cases, we will prove in the next lemma an exponential upper bound for the number of all possible \smc s of $\MS_n$ and then we will deduce the desired exponential decay using a Peierls type argument. \mymargin{The constant $\nu$ is better than the one of \Cr{lattice animals}.}

A straightforward application of \Cr{lattice animals} yields
\begin{lemma} \label{cubic animals}	
For every graph $G$ with maximum degree $3$, and any vertex $e\in E(G)$, the number of 2-connected subgraphs of $G$ with $m$ edges containing $e$ is at most $\nu^m$ for some constant $\nu$. 
\end{lemma}
\comment{
\begin{proof}

\end{proof}
}

The following lemma is the analogue of \Lr{HR bound}.

\begin{lemma} \label{HR bound triang}	
There is a constant $r\in \R$ \st\ \fe\ triangulation $T$ of an open disk satisfying condition \eqref{partial A} 
of \Tr{pC triang} and every $n\in \N$ at most $t r^{\sqrt{n}} e^{n/k}$ elements of $\MS_n$ can occur simultaneously in any percolation instance \oo, where $t=t(T,k)>0$ is a constant depending on $T$ and $k$.
\end{lemma}
\begin{proof}
Let $S$ be an element of $\MS_n$, comprising the \scv s  $S_1,S_2,\ldots,S_{l}$. Since any two distinct occurring \scv s are vertex disjoint by \Lr{scs disjoint}, the sizes $m_i$ of $\partial S_i$ define a partition of $n$. We call the multiset 
$\{m_1,m_2,\ldots,m_{l}\}$ the \defi{boundary partition} of $S$.
It is possible that more than one occurring \smc s have the same boundary partition. In order to prove the desired assertion we will show that for every partition $\{m_1,m_2,\ldots,m_{l}\}$ of $n$ the number of occurring \smc s with $\{m_1,m_2,\ldots,m_{l}\}$ as their boundary partition is at most $t e^{n/k}$ for some constant $t>0$. Then the assertion follows by the \HRT\ (\Tr{HRthm}).

Since occurring \scv s meet $R$ and they are vertex-disjoint by \Lr{scs disjoint}, $S$ is uniquely determined by the subset of $R$ it meets. We can utilise \Lr{geodesic} to conclude that the number of occurring \scv s with boundary of size $m_i$ is at most $e^{m_i/k}$ for every $m_i\geq N$, where $N$ is a sufficiently large positive integer. It is easy to see that the number of \scv s with boundary of size at most $N$ is bounded from above by some constant $M>0$. Hence the number of occurring \smc s with $\{m_1,m_2,\ldots,m_{l}\}$ as their boundary partition is bounded above by $M e^{n/k}$.
\end{proof}

We are now ready to prove \Lr{exp dec triang}.

\begin{proof}[Proof of \Lr{exp dec triang}]
By \Lr{HR bound triang} we have 
$$\sum_{M\in \MS_n} Q_M(p) \leq t r^{\sqrt{n}} e^{n/k} \Pr_p(\text{some } M\in \MS_n \text{ occurs})$$ for every $k$. 
Let $r_m$ denote the $m$th edge of $R$. We pick one of the two endpoints from every dual edge $r_m^*$ and we denote it $v_m$ (maybe some of these endpoints are the same). Let $D(m)$ denote the event that one of the clusters of $v_1,\ldots, v_{g(m)}$ contains at least $m$ vertices.
Arguing as in the proof of \Lr{exp dec Z2} we can deduce that
$$\Pr_p(\text{some } M\in \MS_n \text{ occurs})\leq \sum_{\{m_1,\ldots,m_k\}\in P_n} \Pr_{1-p}(D(m_1))\cdot\ldots\cdot\Pr_{1-p}(D(m_k)),$$
where $P_n$ is the set of partitions of $n$. By \Cr{ABP triang} and the union bound we obtain 
$$\Pr_p(\text{some } M\in \MS_n \text{ occurs})\leq t r^{\sqrt{n}} e^{n/k} c^n,$$
where $c$ is the constant of \Cr{ABP triang}. 

When $k$ is large enough, there is some constant $\nu_k<1$ such that $e^{2/k}c<1$ for every $p>\nu_k$. This proves the exponential decay of $\sum_{M\in \MS_n} Q_M(p)$ when $k$ is large enough. Moreover, as $k$ goes to infinity, $e^{2/k}$ converges to $1$ and thus it is easy to choose $\nu_k$ so that it converges to $1/2$. 

For small values of $k$ we can argue as in the proof of \Tr{thm pc fp} to conclude that 
$$\Pr_p(\text{some } M\in \MS_n \text{ occurs})\leq t r^{\sqrt{n}} e^{n/k} \nu^n (1-p)^n,$$
where $\nu$ is a constant provided by \Lr{cubic animals}.
Hence $\sum_{M\in \MS_n} Q_M(p)$ decays exponentially in $n$ for all $p> 1-1/{\nu e^{2/k}}$.
\end{proof}

The following is an easy combinatorial exercise.

\begin{lemma} \label{triang boundary}	
For every triangulation of a disk $T$ and every \scv\ $M$ we have $|E(M)|\leq 4|\partial M|$.
\end{lemma} \mymargin{4 instead of 2}
\begin{proof}
Let $H$ be a finite connected graph witnessing the fact that $M$ is an \scv. We claim that every edge $e\in M$ lies in a triangular face $T_e$ of $T$ such that at least one edge of $T_e - e$ lies in $\partial M$. Indeed, $e$ lies in exactly two (triangular) faces of $T$, and we choose $T_e$ to be one of them lying in the unbounded face of $H$; such a $T_e$ exists because by definition the vertices and edges of $M$ are incident with the unbounded face of $H$.
As $T_e$ lies in the unbounded face of $H$, one of the two other edges of $T_e$ lies in $\partial M$.

Since any edge of $\partial M$ lies in at most two of these triangular faces $T_e$, and each such face contains at most two edges of $E(M)$, the result follows. \mymargin{check the proof}
\end{proof} 

We have collected all the ingredients for the main result of this section. 

\begin{proof}[Proof of \Tr{pC triang}]
We first remark that $p_c<1$ by \eqref{thetaTriang} because, easily, $c_2(p)\to 0$ as $p\to 1$.  

To obtain our precise bounds, note that, by definition, every $M\in \MS_n$ has $n$ vacant edges. Moreover, $|E(M)|\leq 2n$ by \Lr{triang boundary}. Hence we can now apply \Cr{cor general} to deduce that $\theta_o(p)$ is analytic for $p>\nu_k$. As usual, we then recall that $\theta_o(p)$ cannot be analytic at $p_c$, and so $p_c\leq p_\C$.
\end{proof}

{\bf Remark:} The above proof uses some complex analysis (needed in \Cr{cor general}) to prove $p_c<1/2$. But the  complex analysis can be avoided by using a refinement of the Peierls argument that can be found in \cite[Theorem~4.1]{Pete}. 

\medskip
For the proof of \Tr{pC transient} we just need to show that the size of the set of edges of a 1-way geodesic $R$ that meets $\bigcup \MS_n$ grows subexponentially in $n$. To this end, we will use the well-known theorem of He \& Schramm stating that every graph as in our statement is the contacts graph of a circle packing whose carrier is the open unit disc $\DD$ in $\R^2$; see \cite{HeSchrHyp}, where the relevant definitions can be found. We say that an edge $e$ \defi{meets} $\MS_n$, if there is $M \in \MS_n$ with $e\in \partial M$.

\begin{lemma} \label{trans R}	
Let $T$ be a triangulation of an open disk which is transient and has bounded vertex degrees. Let $R$ be a geodesic ray in $T$ starting at any $o\in V(G)$, and let $R_n$ be the set of edges of $R$ meeting $\MS_n$. Then $|R_n|= O(n^3)$.
\end{lemma}
\begin{proof}
Let $P$ be a circle packing for $T$ whose carrier is the open unit disk $\DD$, provided by \cite{HeSchrHyp}. The main properties of $P$ used in our proof are 
\begin{enumerate}
\item \label{tangent} two vertices of $T$ are joined with an edge \iff\ the corresponding circles are tangent, and
\item \label{accfree} there are no accumulation points of circles of $P$ inside \DD.
\end{enumerate}

Assume that $|R_n| =\oo(n^3)$ contrary to our claim. Let $R'_n$ be the set of vertices of $R$ incident with an edge in $R_n$. Then $|R'_n| > |R_n| =\oo(n^3)$.  For a vertex $u$ of $T$, let $x_u$ denote the corresponding circle of $P$.

For any $u\in R'_n$ \Lr{triang boundary} yields a connected subgraph $G_u$ of $T$ of at most $4n+1$ edges containing $u$ and surrounding $o$; indeed, $G_u$ can be obtained from any \scv\ $M$ witnessing the fact that $u\in R'_n$ by possibly adding the edge of $u$ lying in $\partial M$ in case $u$ does not lie on $M$. 

Let $P_u$ denote the union of the disks of $P$ corresponding to $G_u$. We claim that the area $\text{area}(P_u)$ covered by $P_u$ is at least $r/n^2$ for some constant $r=r(P)$. Indeed, $P_u$ is the union of at most $4n+2$ disks, and its diameter is greater than the diameter of $x_o$, and so at least one of its circles must have diameter of order at least $1/n$, hence area of order at least $1/n^2$.

For every $n$, pick a subset $R''_n$ of $R'_n$ such that any two vertices of $R''_n$ lie at distance at least $8n+3$ along $R$, and therefore in $T$ since $R$ is a geodesic, and $|R''_n| =\oo(n^2)$. Such a choice is possible because $R'_n=\oo(n^3)$. 

Note that for any two distinct elements $u,v\in R''_n$, the subgraphs $G_u,G_v$ defined above are vertex disjoint: this is because we chose $u,v$ to have distance at least $8n+3$ in $T$, and each of $G_u,G_v$ has at most $4n+1$ edges and is connected. Moreover, recall that each $P_u$ has area of order at least $1/n^2$. Combining these two facts we obtain $\sum_{u\in R''_n} \text{area}(P_u) = \oo(1)$, a contradiction since $\text{area}(\DD)$ is finite.
\end{proof}

\begin{proof}[Proof of \Tr{pC transient}]
We repeat the arguments of the proof of \Tr{pC triang}, replacing \Lr{geodesic} by \Lr{trans R}.
\end{proof}

In the case of recurrent triangulations the theorem of He \& Schramm states that $T$ is the contacts graph of a circle packing whose carrier is the plane $\mathbb{R}^2$ \cite{HeSchrHyp}.
Let $P$ be such a circle packing. We will prove the analogue of \Lr{trans R} for recurrent triangulations of an open disk such that the radii of the circles of $P$ are bounded from above. This in turn implies that $p_{\C}\leq 1/2$ for such triangulations by repeating the proof of \Tr{pC transient}.

\begin{lemma} \label{recur R}	
Let $T$ be a triangulation of an open disk which is recurrent and has bounded vertex degree. Assume that 
\labtequ{radius}{for some (and hence every) circle packing $P$ of $T$,  the radius of every disk in $P$ is bounded from above by some constant $M$.}
Let $R$ be a geodesic ray in $T$ starting at any $o\in V(G)$, and let $R_n$ be the set of edges of $R$ contained in some \scv\ of $\MS_n$. Then $|R_n|= O(n^5)$.
\end{lemma}
\begin{proof}
We will follow the proof of \Lr{trans R}. Assume that $|R_n| =\oo(n^5)$ contrary to our claim. Recall the definitions of $P_u$, $G_u$ and $R'_n$, and let $R''_n$ be defined as in the proof of \Lr{trans R} with the additional property $\infty>|R''_n|=\oo(n^4)$. This is possible because $|R_n| =\oo(n^5)$. In the proof of \Lr{trans R}, we utilised the finite area of $\DD$ to derive a contradiction. However, the area of the plane is infinite. For this reason, we will construct a family of bounded domains $(D_n)$ with the property that $P_u$ is contained in $D_n$ for most $u\in R''_n$.

Let $u_n$ be the vertex of $R''_n$ that attains the greatest graph distance from $o$. We claim that $G_n:=G_{u_n}$ contains a cycle that surrounds $o$. Indeed, assuming that $G_n$ does not contain any such cycle, we obtain that $o$ lies in $G_n$. Consider now some $u\in R''_n$ other than $u_n$. Then $G_u$ is vertex disjoint from $G_n$, as mentioned in the proof of \Lr{trans R}. As $G_u$ separates $o$ from infinity, $G_n$ must lie in a bounded face of $G_u$. This implies that $d(u,o)>d(u_n,o)$, which is a contradiction. Hence $G_n$ contains a cycle $C_n$ that surrounds $o$. 

Let $D_n$ be the domain bounded by $C_n$. Arguing as above, we can immediately see that each $P_u$ for $u\in R''_n\setminus \{u_n\}$ lies in $D_n$. Moreover, $C_n$ contains at most $4n$ edges by \Lr{triang boundary}. Every edge of $T$ has length at most $2M$ in $P$ by our assumption, therefore, the length of $C_n$ (as a curve in $\mathbb{R}^2$) is at most $8Mn$.

As in the proof of \Lr{trans R} if $u\in R''_n$, then some circle of $P_u$ has area of order at least $1/n^2$. Hence we obtain $\sum_{u\in R''_n\setminus \{u_n\}} \text{area}(P_u) = \oo(n^2)$, since $|R''_n\setminus \{u_n\}|=\oo(n^4)$. Using the standard isoperimetric inequality of the plane, we derive $$\sum_{u\in R''_n\setminus \{u_n\}} 4\pi\cdot \text{area}(P_u)\leq 4\pi\cdot \text{area}(D_n)\leq (8Mn)^2.$$ We have obtained a contradiction.
\end{proof}

Using an idea of Grimmett \& Li \cite{GrimLi}, we can slightly improve our results to obtain the strict inequality $p_c\leq p_\C<1/2$ instead of $p_c\leq p_\C\leq 1/2$ in all above results. Indeed, it is not hard to see that for any bounded degree triangulation of an open disk $T$, $\sigma_k(o)\leq 3 \cdot 2^{d-1}(2^d -2)^{\left \lfloor{n/d}\right \rfloor}$, where $d$ is the maximum degree of $T$. This comes from the fact that for every vertex $u$ and any edge $e$ incident to $u$ the number of $d$-step self avoiding walks starting from $u$ that do not traverse $e$ is at most $2^d -2$. Hence $p_c\leq p_{\C}<1/2$ as claimed. 

\subsection{Site percolation} \label{sec sit tri}
A well-known remark of Grimmett \& Stacey \cite[\S 7.4] {LyonsBook} transforms any upper bound on $p_c(G)$ into an upper bound on \pcs\ via the formula $\pcs\leq 1-(1-p_c)^d$ whenever $G$ has maximum degree $d$. But in our case we can do better: for the triangulations for which we proved $p_\C \leq 1/2$ in the previous section we can also prove $\pcs \leq 1-\frac1{d-1}$. For this, instead of working with the dual $T^*$ we work directly with the primal $T$. We adapt \eqref{chi finite bound} into  \\
$\mathbb{E}_{1-p}(P(u))\leq \sum_{k=0}^\infty d\cdot (d-1)^{k-1} (1-p)^k<\infty,$
which yields an analogue of \Cr{ABP triang} for $p>1-1/d$. We then proceed as in the proof of \Tr{pC triang}.

\section{Alternating signs of Taylor coefficients} \label{sec alt}

In \Sr{secPF} we proved that the functions $f_m(t) :=\Pr_t(|C(o)| \geq m)$ and $p_m(t) :=\Pr_t(|C(o)| = m)$ are analytic, and even more, they can be extended into entire functions. Thus $p_m$ is uniquely determined by its Maclaurin coefficients. We remark that most `macroscopic' functions of percolation theory, e.g.\ $\chi$ and $\theta$, are uniquely determined by the sequence $\{p_m\}_{m\in \N}$, and hence by their Maclaurin coefficients. It is rather hopeless to try to determine all these coefficients for any particular percolation model, but perhaps it is less hopeless to e.g.\ compare two models by comparing the corresponding Maclaurin coefficients. 

Motivated by such thoughts we wondered what can be said about those coefficients in general. In this section we determine the signs of the Maclaurin coefficients of $f_m$ and $p_m$, which turn out not to depend on the model, and deduce that they are alternating. In fact this remains valid in any non trivial percolation model and we do need to impose any transitivity assumption. We let $V$ be a countably infinite set, and $\mu$ any function defined on the set $E:= V^2$ of pairs of elements of $V$ such that $\sum_{y \in V} \mu(xy) <\infty $
for every $x\in V$, and use this data to obtain a percolation model as defined in \Sr{sec setup}. However, for ease of notation we will assume that $\sum_{y \in V} \mu(xy) =1 $  for every $x\in V$, as in \Sr{secPF}. (Some readers may prefer to think of $V$ as the vertex set of a countable connected graph, with $\mu$ supported on its edge set $E$.)

\medskip
We call an entire function \defi{alternating}, if its Maclaurin coefficients are all real and their signs are alternating. To be more precise, if the Maclaurin series of $f$ is $\sum c_i x^i$, with $c_i\in \R$, we say that $f$ is alternating if $\text{sgn}(c_i) = (-1)^{i+\eps}$, for some $\eps \in \{0,1\}$. Here, the \defi{sign $\text{sgn}(c)$} of a real number $c\neq 0$ is defined as $c/|c|$. With a slight abuse of notation, we allow $\text{sgn}(0)$ to take any of the values $1$ or $-1$. For example, any constant real function is allowed as an alternating function. 

More generally, we say that $f$ is \defi{alternating at a point $r\in \R$}, if the Taylor coefficients $c_i$ of $f(z)$ at $z=r$  satisfy $\text{sgn}(c_i) = (-1)^{i+\eps}$. 

For an analytic function $f$, we let \defi{$f[k]:= \frac{f^{(k)}(0)}{k!}, k\geq 0$} denote the $k$th Maclaurin coefficient of $f$. More generally, let $f[k](r)$ denote the $k$th Taylor coefficient of $f$ at $r$.

\begin{theorem}\label{alterf}
The (entire extension of the) function  $f_m$ is alternating, with $\text{sgn}(f_m[k])= (-1)^{m+1+k}$. 
\end{theorem}

Since $p_m= f_{m} - f_{m+1}$, this immediately implies
that $p_m$ is alternating too, with $\text{sgn}(p_m[k])= (-1)^{m+1+k}$.

We will prove \Tr{alterf} by induction, and to do so we will prove the following refinement of our statement. 
Let $F$, $A$ be non-empty subsets of $V$, such that $F$ is a finite subset of $A$. Any percolation instance $\oo\in \{0,1\}^{E}$  can be restricted to define a random graph $A_\oo$ on $A$ by only keeping the edges that are occupied and have both end-vertices in $A$.  By a straightforward extension of \Tr{Pm entire}, we can prove that the function
$\Pr_{t}(|\cup_{g\in F} C_{A,g}|\geq m)$, where $C_{A,g}$ denotes the component of vertex $g$ in $A_\oo$, admits an entire extension, which we will denote by $f_m$. Our aim is to prove that $f_m$ is alternating for every $m\geq |F|$, with $\text{sgn}(f_m[k])= (-1)^{m+|F|+k}$. The special case where $F=\{o\}$ and $A= V$ then yields \Tr{alterf}.

We will prove this using the following formula:
\labsplitequ{firstF}{
f_{m}(t)=\Pr_t(|N_{A\setminus F}(g_1)|\geq m-|F|)+\\ \sum_{n=0}^{m-|F|-1}\sum_{L\in B_n} 
\Pr_t(|\cup_{g\in S_L} C_{A\setminus \{g_1\},g}|\geq m-1)\Pr_t(N_{A\setminus F}(g_1)=L),}
where $g_1$ is a fixed but arbitrary element of $F$, and $B_n$ is the set of all possible subsets of size $n$ of the (deterministic) neighbourhood of $g_1$ in $A\setminus F$, and $S_L:=\big(F\setminus \{g_1\}\big)\cup L$. 

The fact that this formula holds (for all $t\in \R_+$) is easy to check: we consider all possible neighbourhoods $L$ of $g_1$ in $A\setminus F$ in our percolation instance, and compute the probability of the event $|\cup_{g\in F} C_{A,g}|\geq m$ defining $f_m$ conditioning on $L$, except that we bulk all $L$ with $|L|\geq m-|F|$ into the first summand of the right hand side.

We claim moreover that the functions involved in the right hand side admit entire extensions, and that these extensions still satisfy \eqref{firstF} \fe\ $z\in \C$.

Indeed, the first summand can expressed as a sum of simpler functions via the formula
\labtequ{AFg1}{$\Pr_t(|N_{A\setminus F}(g_1)|\geq m-|F|)=1-\sum_{n=0}^{m-|F|-1} \sum_{L\in B_n} \Pr_t(N_{A\setminus F}(g_1)=L).$}
By \Cr{neighbourhood} all functions of the form $\Pr_t(N_{A\setminus F}(g_1)=L)$ admit entire extensions and 
\labtequ{sum neighbourhood}{$\sum_{L\in B_n} |\Pr_t(N_{A\setminus F}(g_1)=L)|\leq e^{2M} \sum_{L\in B_n} \Pr_{M}(N_{A\setminus F}(g_1)=L)<\infty $} 
for every $M>0$ and every $z\in D(0,M)$. Applying the Weierstrass M-test and Weierstrass' \Tr{thmWei} as usual we deduce that $\sum_{L\in B_n} \Pr_t(N_{A\setminus F}(g_1)=L)$ admits an entire extension, and hence so does $\Pr_t(|N_{A\setminus F}(g_1)|\geq m-|F|)$ by \eqref{AFg1}.

For the second summand of \eqref{firstF} we observe as above that all functions $\Pr_t$ involved admit entire extensions and thus it suffices to verify once again the assumptions of the Weierstrass M-test for the series taken when summing over $L\in B_n$. To upper bound $\Pr_t(|\cup_{g\in S_L} C_{A\setminus \{g_1\},g}|\geq m-1)$ we will use the identity
$$\Pr_t(|\cup_{g\in S_L} C_{A\setminus \{g_1\},g}|\geq m-1)=1-\sum_{j=1}^{m-2} \Pr_t(|\cup_{g\in S_L} C_{A\setminus \{g_1\},g}|=j).$$ 
Using the estimates of \Lr{C equals S LR} and a simple triangle inequality we deduce, for the corresponding entire extensions, that 
$$|\Pr_z(|\cup_{g\in S_L} C_{A\setminus \{g_1\},g}|\geq m-1)|\leq 1+\sum_{j=1}^{m-2} e^{2Mj}P_{M}(|\cup_{g\in S_L} C_{A\setminus \{g_1\},g}|=j)$$
for every $M>0$ and every $z\in D(0,M)$. We can further upper bound\\ $|\Pr_z(|\cup_{g\in S_L} C_{A\setminus \{g_1\},g}|\geq m-1)|$ by $1+(m-2)e^{2Mm}$, because obviously \\ $P_{M}(|\cup_{g\in S_L} C_{A\setminus \{g_1\},g}|=j)\leq 1$. Combining this with \eqref{sum neighbourhood} we deduce that the assumptions of the M-test are verified. 

This proves that the right hand side of \eqref{firstF} admits an entire extension as claimed. Since this extension  coincides with $f_m$ on $\R_+$ as mentioned above, it must coincide with $f_m(z)$ on all of $\C$ by the uniqueness principle since $f_m(z)$ is entire.

\medskip
Our inductive proof of \Tr{alterf} is based on the observation that all these functions $\Pr_t$ involved in \eqref{firstF} are alternating themselves, and the following basic fact that products of alternating functions are alternating.

\begin{lemma}\label{alterprod}
If $f,g$ are entire alternating functions then $fg$ is also alternating, and $\text{sgn}(fg[k]) = \sgn{k} \text{sgn}(f[k])\text{sgn}(g[k])$.
\end{lemma}
\begin{proof}
This is an easy combinatorial excersise, using the well-known fact that the Taylor series of a product of two analytic functions coincides with the product of the Taylor series of the two functions at any point of the intersection of their domains of definition.
\end{proof}

We now prove that the entire extensions of the functions of the form\\ $\Pr_t(N_{A\setminus F}(g_1)=L)$ appearing in \eqref{firstF} are alternating.

\begin{lemma}\label{NAalter}
Let $L$, $X$ be non-empty subsets of $V$, such that $L$ is a finite subset of $X$. Then for every $o\in V$, the entire extension $f$ of $\Pr_t(N_{X}(o)=L)$ is alternating, with $\text{sgn}(f[k])=\sgn{|L|+k}$. 
\end{lemma}
\begin{proof}
By definition, our function satisfies the following formula:
\labsplitequ{NLformula}{
f(z):= \Pr_{z}(N_{X}(o)=L)= &\prod_{s\in X\setminus L}e^{-z\mu(o s^{-1})}\prod_{s\in L}\big(1-e^{-z\mu(o s^{-1})}\big) =\\
& e^{-z\sum_{s\in X\setminus L}\mu(o s^{-1})}\prod_{s\in L}\big(1-e^{-z\mu(o s^{-1})}\big).}
Since the function $e^{-z \nu}$ is alternating for every real constant $\nu$, the latter expression is a product of $|L|+1$ alternating functions. Thus the result follows from \Lr{alterprod}. Indeed, the leftmost factor has its odd Maclaurin coefficients positive, while each of the $|L|$ other factors has its even  coefficients positive.
\end{proof}

Next, we prove that the first summand of \eqref{firstF} is also alternating.

\begin{lemma}\label{jalter}
Let $F$, $X$ be non-empty subsets of $V$, such that $F$ is a finite subset of $X$. Then \fe\ $o\in V$, the analytic extension $f$ of\\ $P_{z}(|N_{X}(o)|\geq j)$ is alternating \fe\ $j\geq 0$, with $\text{sgn}(f[0])=\sgn{j}$. 
\end{lemma}
\begin{proof}
We can rewrite $f$ as 
\labtequ{fsumsum}{$f(z)= 1- \Pr_{z}(|N_{X}(o)| < j) = 1- \sum_{n=0}^{j-1}\sum_{L\in B_n} \Pr_{z}(N_{X}(o) = L)$.}
Indeed, this formula is easily verified for $z\in \R_+$, and by the arguments used for \eqref{firstF} it holds \fe\ $z\in \C$.

Note that the right hand side involves $j$ sums, each of which is a sum of alternating functions with agreeing signs by \Lr{NAalter}. However, the signs of each of those $j$ sums have alternating parities, and since we do not know anything about the absolute values of their coefficients this formula is not enough to prove our statement. However, it will be useful below on different grounds.

\medskip
We start by proving the statement of the lemma for finite $X$, using a double induction on $|X|$ and $j$. To begin with, for $j=0$, $f$ is alternating, with $\text{sgn}(f[k])=\sgn{k}$ \fe\ finite $X$, as it becomes the constant function $f=1$. Moreover, $f$ is identically 0 and hence alternating for $X=\emptyset$ and every $j\geq 1$, and we can take $\text{sgn}(f[k])=\sgn{j+k}$ in this case. For the inductive step, suppose the statement is proved for $j\leq k$ and every finite $X$. Then for $j=k$, we will prove it by induction on $|X|=1,2,\ldots$ (remember we already know it for $|X|=0)$. For this, we can pick any element $x\in X$, and rewrite $f$ as follows, by distinguishing between the events of the edge $ox$ being absent or present:  
\begin{equation}  \label{sumsum} 
\begin{aligned}
f(z)= P_{z}(|N_{X}(o)|\geq j) = & e^{-z\mu(ox)} P_{z}(|N_{X\sm x}(o)|\geq j) +\\ 
&(1- e^{-z\mu(ox)} )P_{z}(|N_{X\sm x}(o)|\geq j-1).
\end{aligned} 
\end{equation} 
(Again, we repeat the arguments used \eqref{firstF} to establish this in all of $\C$.)
By our induction hypothesis, both $P_z$ functions involved are alternating; the sign of the $k$th  Maclaurin coefficient of the first one is \sgn{j+k}, while for the second one it is \sgn{j-1+k}. By \Lr{alterprod}, each of the two products of \eqref{sumsum} is alternating, with the sign of the $k$th coefficient being \sgn{j+k}.

\medskip
This completes the induction step, establishing that $f$ is alternating for finite $X$. 
For an infinite $X$ we now use an approximation argument. Let\\ $X_1 \subset X_2 \subset \ldots$ be an increasing sequence of finite subsets of $X$ with $\bigcup X_i = X$. We claim that each Maclaurin coefficient of $P_{z}(|N_{X}(o)|\geq j)$ is the limit, as $i\to \infty$, of the corresponding Maclaurin coefficient of $P_{z}(|N_{X_i}(o)|\geq j)$. Since we have already proved the latter functions to be alternating because $X_i$ is finite, this claim implies our statement that $f$ is alternating.

Applying \eqref{fsumsum} with $X$ replaced by $X_i$ \fe\ $i\in \N$, we have 
\labtequ{fsumsumi}{$f_i(z) := 1- P_{z}(|N_{X_i}(o)| < j) = 1- \sum_{n=0}^{j-1}\sum_{\substack{L\in B_n \\ L\subset X_i}} P_{z}(N_{X_i}(o) = L)$}
To prove the aforementioned claim about the convergence of Maclaurin coefficients, it suffices to show that $f_i$ converges to $f$ uniformly on some open disk $D(0,M)$, and we next show that this is the case. 

Using the explicit formula \eqref{NLformula}, we have 
$$P_{z}(N_{X_i}(o) = L) = P_{z}(N_{X}(o) = L) e^{z\sum_{x\in X \sm X_i} \mu(ox)}$$ whenever $L\subset X_i$. Hence we obtain 
$$\sum_{n=0}^{j-1}\sum_{\substack{L\in B_n \\ L\subset X_i}} P_{z}(N_{X_i}(o) = L)=e^{z\sum_{x\in X \sm X_i} \mu(ox)}\sum_{n=0}^{j-1}\sum_{\substack{L\in B_n \\ L\subset X_i}} P_{z}(N_{X}(o) = L).$$
Pick some $M>0$, and note that as $i\to \infty$, the last factor $e^{z\sum_{x\in X \sm X_i} \mu(ox)}$ approaches the constant $1$ function uniformly on $D(0,M)$ because 
$$|e^{z\sum_{x\in X \sm X_i} \mu(ox)}-1|\leq e^{M\sum_{x\in X \sm X_i} \mu(ox)}-1$$ for every $z\in D(0,M)$ by \Lr{absz}, and the latter quantity converges to $0$. Moreover as $i\to \infty$ the sequence $\sum_{n=0}^{j-1}\sum_{\substack{L\in B_n \\ L\subset X_i}} P_{z}(N_{X}(o) = L)$ converges to $\sum_{n=0}^{j-1}\sum_{L\in B_n} P_{z}(N_{X}(o) = L)$ uniformly on $D(0,M)$, since 
$$\sum_{\substack{L\in B_n \\ L\not\subset X_i}} |P_{z}(N_{X}(o) = L)|\leq \sum_{\substack{L\in B_n \\ L\not\subset X_i}} e^{2M} P_{M}(N_{X}(o) = L)$$ for every $z\in D(0,M)$ by \Lr{C equals S LR},  and the latter sum converges to $0$. Therefore $f_i$ converges to $f$ uniformly on $D(0,M)$ as desired.
\end{proof}

We now have all the ingredients needed for \Tr{alterf}:

\begin{proof}[Proof of \Tr{alterf}]

We work with the more general function\\ $f_m(t)= \Pr_{t}(|\cup_{g\in F} C_{A,g}|\geq m)$ as discussed after the statement of \Tr{alterf}, and proceed by induction on $m$. The statement is trivial for $m\leq |F|$, since $f_m$ is the constant function 1 in this case, and we are allowed to consider $\text{sgn}(0)$ to be 1 or $-1$. For the induction step, supposing we have proved the statement for $m<j$, we can obtain it for $m=j$ using \eqref{firstF}; we repeat it here for convenience:
\labsplitequ{repeat}{
f_{m}(z)=P_{z}(|N_{A\setminus F}(g_1)|\geq m-|F|)+\\ \sum_{n=0}^{m-|F|-1}\sum_{L\in B_n} 
P_{z}(|\cup_{g\in S_L} C_{A\setminus \{g_1\},g}|\geq m-1)P_{z}(N_{A\setminus F}(g_1)=L),}
The first summand is alternating by \Lr{jalter}, while we can prove each summand of the form $P_{z}(|\cup_{g\in S_L} C_{A\setminus \{g_1\},g}|\geq m-1)P_{z}(N_{A\setminus F}(g_1)=L)$ appearing in the second summand to be alternating by combining \Lr{alterprod} with our induction hypothesis and  \Lr{NAalter} (here we used the fact that $|S_L| = |F|+|L|-1 <m-1$ since $|L|\leq m-|F|-1$ in order to be allowed to apply the induction hypothesis). Moreover, it is straightforward to check that these results also imply that the sign of the $k$th Maclaurin coefficient of any of those summands is \sgn{m+|F|+k}. Since the $k$th Maclaurin coefficient of $f_m$ is the sum of the corresponding coefficients of these finitely many summands, this completes the proof that $f_m$ is alternating, with $\text{sgn}(f_m[k])= (-1)^{m+|F|+k}$.

\comment{
\medskip
The proof that $f_m(z)$ is almost alternating at any $r\in \R$ follows the same lines: it will suffice to extend \Lr{alterprod} to all of $\R$, since \Lrs{NAalter} and \ref{jalter}, on which the above proof was based, are consequences of that lemma. The extension of  \Lr{alterprod} we need is the following:
\labtequ{alterprodgen}{If $f,g$ are entire functions almost alternating at a point $r\in \R$, then $h=fg$ is also almost alternating at $r$, and\\ $\text{sgn}((h(z)-h(r))[k]) =  \sgn{k}\text{sgn}(f[k])\text{sgn}(g[k])$.}
To prove this, notice first that the (entire) function 
\begin{align*} 
\begin{split}
h^*(z):= &\left( f(z) - f(r) \right) \left( g(z) - g(r) \right) = \\
&f(z)g(z) - f(r)g(z) - g(r)f(z) + f(r)g(r)
\end{split}
\end{align*} 
is alternating at $r$; indeed, each factor is alternating by the hypothesis that $f$ and $g$ are almost alternating at $r$, and hence $h^*$ is alternating by  \Lr{alterprod}.

However, our aim is to prove that $h=fg$ is almost alternating at $r$, that is,  $h^\circ(z):= f(z)g(z) - f(r)g(r)$ is alternating. To prove this, notice that, by elementary manipulations, 
\labsplitequ{threesummands}{
h^\circ(z) = &h^*(z) + f(r)g(z) + g(r)f(z) - 2f(r)g(r) = \\
&h^*(z) + f(r)\left(g(z) - g(r) \right) + g(r)\left(f(z)- f(r) \right).
}

Recall that the first summand $h^*$ is alternating. The same is true for the other two summands $f(r)\left(g(z) - g(r) \right)$ and $g(r)\left(f(z)- f(r) \right)$ by the assumption that $f$ and $g$ are almost alternating at $r$, since the terms $f(r)$ and $g(r)$ are constants. Thus, to prove that $h^\circ(z)$ is alternating, it only remains to check that the signs of corresponding Taylor coefficients of these three summands agree. This is a straightforward check: the sign of the odd coefficients of the last two summands coincides with $\text{sgn}(f(r)g(r))$ (we can assume that $f(r)g(r)\neq 0$ for simplicity, by adding constants to $f$ and $g$ if needed). To determine the signs of $h^*(z)$ we need to calculate the second derivative as $h^*(z)$ has a zero of order 2 at $r$. We have $dh^*(z)/dz = f'(z)\left(g(z) - g(r) \right)  + g'(z) \left(f(z)- f(r) \right)$, and $d^2h^*(z)/dz^2 = f''(z)\left(g(z) - g(r) \right)  + g''(z) \left(f(z)- f(r) \right) + 2f'(z) g'(z)$. Evaluated at $z=r$ this gives $2 f'(r) g'(r)$, which has the same sign as $f(r)g(r)$ by the alternating property of $f,g$. Thus the sings of the Taylor coefficients of $h'$ agree with those of the other two summands of \eqref{threesummands} as claimed. This completes the proof of \eqref{alterprodgen}.

The proofs of \Lrs{NAalter} and \ref{jalter} apply almost verbatim to prove that the corresponding functions are almost alternating at any $r\in \R$. Using them, we can prove  that $f_m(z)$ is almost alternating too at any $r\in \R$ by the same induction as in the case $r=0$ considered above.}
\end{proof}

We just proved that $f_m$ and $p_m$ are alternating at 0. Using this we can prove the same for $z$ on the negative real axis.

\begin{corollary}\label{alterneg}
The functions $f_m$ and $p_m$ are alternating at every $r \in \R_{\leq 0}$, with $\text{sgn}(f_m[k](r))= \text{sgn}(p_m[k](r))= (-1)^{m+k+1}$.
\end{corollary}
\begin{proof}
It suffices to prove the statement for $f_{m}$, since we can then deduce it for $p_{m}$ using again the fact that $p_{m}= f_{m} - f_{m+1}$.

Since $f_{m}$ is an entire function, so is its $n$th derivative $f_m^{(n)}(z)$, and therefore the radius of convergence of the Maclaurin expansion of $f_m^{(n)}(z)$ is infinite. Thus we can determine the sign of $\text{sgn}(f_m^{(n)}[k](r))$ by using the Maclaurin expansion of $f_m^{(n)}(z)$. The latter can be immediately obtained using the Maclaurin expansion of $f_m$, and we have $\text{sgn}(f_m^{(n)}[k]) = \text{sgn}(f_m[k+n])$, which by \Tr{alterf} equals $(-1)^{m+1+k+n}$. Evaluating the Maclaurin expansion of $f_m^{(n)}$ at $r<0$ we see that all terms of that expansion have sign $(-1)^{m+n+1}$, and so $\text{sgn}(f_m^{(n)}(r)) = (-1)^{m+n+1}$. Since $\text{sgn}(f_m[k](r))= \text{sgn}(f_m^{(k)}(r))$ by the definition of the Taylor expansion, our claim follows.
\end{proof}

We finish this section with a related fact about the zeros of our functions.
\begin{theorem}\label{zerosf}
The functions $p_m$ and $f_m$ have a zero of order at least $m-1$ at $z=0$ \fe\ $m>1$.
\end{theorem}
\begin{proof}
We first prove the statement for $p_m$. Note that any connected graph with $m$ vertices has at least $m-1$ edges. Hence using the explicit formulas 
\labtequ{sumpm}{$p_m(z)=\sum_{S\in G_m} P_{z}(C(o)=S)$,}
where $G_m$ denotes the set of connected graphs on $m$ vertices in $V$, and
$$P_{z}(C(o)=S)=\prod_{e\in \partial S}e^{-z\mu(e)}\prod_{e\in E(S)}\big(1-e^{-z\mu(e)}\big),$$
we see that the summands of $p_m$ have a zero of order at least $m-1$ at $z=0$, because each factor of the form $1-e^{-z\mu(e)}$ contributes a zero of order $1$ and $|E(S)|\geq m-1$. By \Tr{Pm entire} the partial sums in \eqref{sumpm} converge uniformly on an open neighbourhood of 0 to $p_m$, which implies that $p_m$ satisfies the desired property.

Combining this with the formula  $p_m= f_{m} - f_{m+1}$, we can now easily deduce that  $f_m$ too has a zero of order at least $m-1$ at $z=0$. Indeed, by \Cr{alterneg} the $k$th Maclaurin coefficient of $f_{m}$ and $-f_{m+1}$ have the same sign $(-1)^{m+1+k}$. Hence if any of the first $m-1$ Maclaurin coefficients of  $f_{m}$ or $-f_{m+1}$ is non-zero then so is the corresponding coefficient of $p_m$, contradicting what we just proved.
\end{proof}

\section{The negative percolation threshold} \label{sec neg}

In \Sr{secChi} we proved that the susceptibility $\chi$ is an analytic function of the parameter below the percolation threshold $p_c$ or $t_c$ for all transitive models. This means that $\chi(t)$ admits an extension into a holomorphic function in some domain $D$ of $\C$ containing the interval $[0,p_c)$ or $[0,t_c)$. It would be interesting to come up with a definition that determines this $D$ uniquely, and makes it maximal in some sense. Motivated by this quest, we introduce in this section a `negative threshold' $t_c^- \in \R_{<0}$, at which the boundary of such a $D$ would have to cross the negative real axis.
From now on we will be working with a transitive \lrm\ as defined in \Sr{sec setup}, but the discussion can be repeated for \nnm s as well.



The standard percolation threshold $t_c$ is typically defined as $\sup\{t \mid \theta(t)=0\}$. Natural alternative definitions of $t_c$ can be given by considering the finiteness of the susceptibility $\chi$, i.e.\ as $\sup\{t \mid \chi(t)< \infty\}$, or in terms of the exponential decay of the cluster size as $\sup\{t \mid \exists c<1 : p_m(t) \leq c^m\ \forall m\in \N \}$. 
For a while it was an open problem whether these three thresholds coinside, which was settled by the papers \cite{AizNewTre,AizBar,DCTa} (we discussed in \Sr{sec ABP} about how these results generalise to \lrm s).

When trying to define the negative threshold $t_c^-$ we are faced with similar difficulties, some of which we are able to overcome below. Perhaps the most natural definition is the following. Since we know (\Tr{chi LR}) that $\chi(t)$ admits an analytic extension into a domain containing the real interval $[0,t_c)$, we can let $I$ be the largest real interval that contains $[0,t_c)$ and is contained in the domain of an analytic extension of $\chi(t): \R_+ \to \R_+$, and let $t_c^-=t_A$, where $t_A\in \R_- \cup \{-\infty\}$ is defined as the leftmost point of $I$. 

Different definitions for $t_c^-$ can be given based on the concrete analytic extension of $\chi$ that we constructed with \Tr{chi LR}: recall that we used the fact that, for $t\in \R_+$, we have $\chi(t) = \sum_m m p_m$. Alternatively, we could have used the formula $\chi(t) = \sum_m f_m$. This motivates the following definitions.

\begin{definition}
We define $t_1\coloneqq \inf\{r<0 \mid \lim_{m\rightarrow \infty} p_m(r)=0\}$,
$t_2\coloneqq \inf\{r<0 \mid \sum_{m=1}^\infty m|p_m(r)| < \infty\}$, 
$t_3\coloneqq \inf\{r<0 \mid \lim_{m\rightarrow \infty} f_m(r)=0\}$ and $t_4\coloneqq \inf\{r<0 \mid \sum_{m=1}^\infty |f_m(r)| < \infty\}$.
\end{definition}

Moreover, given the important role of the exponential decay of $p_m$ in this paper, it is also natural to define 
$$t_5:= \inf\{r<0 \mid \exists c<1 : |p_m(r)| \leq c^m\ \forall m\in \N\}.$$

We remark that since $\text{sgn}(f_m[k](r)))=(-1)^{m+k+1}$ and $\text{sgn}(p_m[k](r)))=(-1)^{m+k+1}$ when $r<0$ by the results of \Sr{sec alt}, we see that $|f_m(r)|$ and $|p_m(r)|$ are decreasing
functions of $r$ for every $m\geq 1$. 

We will show that all these values  $t_i$ coincide (\Tr{coincide}).
A key role in our proof will be played by the Hadamard three circles theorem (\Tr{HT}). 
In order to use it, we first prove that
the supremum of both $|f_m|$ and $|p_m|$ over the closed disk $D(0,M)$ is attained at
$z=-M$.

\begin{lemma} \label{maxx}
Let $f$ be an alternating function. Then for every $M>0$
$$\sup_{z\in D(0,M)} |f(z)|=|f(-M)|.$$
\end{lemma}
\begin{proof}
Let $f(z)=\sum_{k=0}^\infty c_kz^k$ be the Taylor expansion of $f$.
Then 
$$\abs*{\sum_{k=0}^\infty c_k z^k}=\abs*{\sum_{k=0}^\infty (-1)^k c_k (-
z)^k}\leq \sum_{k=0}^\infty|(-1)^k c_k|M^k.$$
Note that the sign of $(-1)^k c_k$ is the same 
for every $k$, since $\text{sgn}(c_k)=(-1)^{k+\varepsilon}$ for some $\varepsilon\in \{0,1\}$. Hence,
$$\sum_{k=0}^\infty |(-1)^k c_k|M^k=\abs*{\sum_{k=0}^\infty (-1)^k c_kM^k}=
\abs*{\sum_{k=0}^\infty c_k(-M)^k}=|f(-M)|.$$ Thus, $f$ is maximised at 
$z=-M$.
\end{proof}

\comment{
As both $f_m$ and $p_m$ are alternating, we have proved the following.

\begin{corollary}\label{maxx}
For every $M>0$ we have  
$$\sup_{z\in D(0,M)} |f_m(z)|=|f_m(-M)|$$
and 
$$\sup_{z\in D(0,M)} |p_m(z)|=|p_m(-M)|$$
\end{corollary}
}

\begin{theorem}\label{coincide}
With the above notation we have $t_1=t_2=t_3=t_4=t_5$. 
\end{theorem}
\begin{proof}
We will show that $t_1=t_5$, from which the remaining equalities follow easily. It is immediate from the definitions that $t_1\leq t_5$. To show that $t_1\geq t_5$, pick $r,r_2\in \R$ with $t_1<r_2<r<0$. By Theorem~\ref{Pm entire} we have
$|P_m(z)|\leq e^{2mM} P_m(M)$ for every $M>0$ and $z\in D(0,M)$. Therefore, since
$P_m(M)$ decays exponentially in $m$ for every $0<M<t_c$ \cite{AizNewTre, AntVes}, we can choose  $r_1<0$ with
$|P_m(r_1)|\leq ke^{-lm}$ for some $k,l>0$ (for this argument we can do without the results of \cite{AizNewTre, AntVes}; instead, we can use the fact that $P_m(M)$ decays exponentially in $m$ for $0<M<1$, which can be proved by comparison with a subcritical Galton-Watson tree). Pick such an $r_1<0$ with $r_1>r$. 
Using the Hadamard three circles theorem (\Tr{HT}) and Lemma~\ref{maxx}, we have 
$$|P_m(r)|\leq |P_m(r_1)|^{c_1}|P_m(r_2)|^{c_2},$$
where $c_1=\dfrac{\log|r_2|-\log|r|}{\log|r_2|-\log|r_1|}$ and $c_2=\dfrac{\log|r|-\log|r_1|}{\log|r_2|-\log|r_1|}$. Note that both 
$c_1,c_2$ are positive. Since $|P_m(r_2)|$ converges to
$0$ by the definition of $t_1$, it follows that  $|P_m(r_2)|^{c_2}$ is bounded above by some constant $c>0$. Moreover, by the choice of $r_1$, $|P_m(r_1)|^{c_1}$ decays exponentially in $m$. Hence so does $|P_m(r)|$. This proves that $t_1=t_5$.

Obviously $t_1\leq t_2$ and $t_3\leq t_4$.
Using the identity $P_m=f_m-f_{m+1}$ we see that $P_m$ converges to $0$ whenever $f_m$ does.
This shows that $t_1\leq t_3$. Also, assuming that $|P_m(r)|$ decays exponentially in $m$ for $t_1<r<0$ we obtain that $\sum_{m=1}^\infty m|P_m(r)| < \infty$. Hence $t_1\geq t_2$. Moreover, we have $f_m(r)=\sum_{i=m}^\infty P_m(r)$: to see this, note that the functions $f_m$ and $\sum_{i=m}^\infty P_m(z)$ coincide on the positive real line. Besides, the exponential decay of $|P_m(r)|$ combined with Lemma~\ref{maxx} and the fact that $P_m$ is alternating by the results of \Sr{sec alt}, implies that $\sum_{i=m}^\infty P_m(z)$ is continuous on $D(0,M)$ and analytic on its interior. Since $f_m$ is entire, the two functions coincide on $D(0,M)$. Therefore, $|f_m(r)|$ decays exponentially in $m$ for $t_1<r<0$ and the series $\sum_{m=1}^\infty |f_m(r)|$ converges, which implies that $t_4\leq t_1$.
\end{proof}

We thus let $t_\chi:= t_i$ be our second candidate for the definition of $t_c^-$. There is one case where we can actually compute $t_\chi$: for the Poisson branching process  (which is not one of our percolation models, but our definitions extend to it  canonically), we have $t_\chi= W(1/e)$, where $W$ denotes the Lambert function. This implies that for appropriately parametrised percolation on the $d$-regular tree $T_d$, we have $\lim_{d\to \infty} t_\chi(T_d)= W(1/e)$. 
\medskip

Since $\chi$ is analytic in $(t_\chi,0]$ and $t_A$ is defined as the infimum over those $t$ such that $\chi$ is analytic in $(t,0]$, it is natural to ask
whether $t_\chi=t_A$, but it turns out that this is not the case: for percolation on the 1-way infinite path, as well as for the Poisson branching process, we have found out that $t_A=-\infty$ although $t_\chi$ is finite. Since these two models are the least and the most percolative examples, it might be that $t_A=-\infty$ always holds, and $t_\chi$ is the `right' definition of the negative threshold.

\section{Appendix: On the number of lattice animals of a given size} \label{sec App Enum}

Let $T_d$ denote the infinite $d$-regular tree, and let $\mathcal S_n$ denote the number of subtrees of $T_d$ with $n$ vertices containing a fixed vertex $o\in V(T_d)$. We claim that
\labtequ{Sn bound}{$\mathcal S_n < c_d  \left( \frac{(d-1)^{(d-1)}}{(d-2)^{(d-2)}} \right)^n$,}
where $c_d$ is a constant depending on $d$ but not on $n$.

This can be proved using the following idea due to Kesten \cite[Lemma~5.1]{KestenBook}. Consider bond percolation on $T_d$ with parameter $p=1/d-1$ (the critical value). The probability that the cluster $C$ of the root has exactly $n$ vertices is of course at most 1. This probability can be explicitly computed as 
$$\Pr( |C|=n) = \mathcal S_n p^{n-1} (1-p)^{(d-2)n+2},$$
 since if $|C|=n$ then $|E(C)|=n-1$ and $|\partial C| = (d-2)n+2$ (the latter can be proved by induction on $n$). Substituting $p$ by $1/d-1$ we arrive at \eqref{Sn bound} by elementary manipulations.

\medskip
Using \eqref{Sn bound} 
we can also upper bound the number of subtrees of any $d$-regular graph: 

\begin{corollary} \label{lattice animals}	
For every graph $G$ with maximum degree $d$, and any vertex $o\in V(G)$, the number of subtrees of $G$ with $n$ vertices containing $o$ is at most
$$c_d  \left( \frac{(d-1)^{(d-1)}}{(d-2)^{(d-2)}} \right)^n< c_d ((d-1)e)^n$$
where $c_d$ is a universal constant depending on $d$ only. 
\end{corollary}
\begin{proof}
We may assume without loss of generality that $G$ is $d$-regular, for otherwise we can attach an appropriate infinite tree to each vertex of degree less than $d$ to raise all degrees to exactly $d$. 

Since \g is $d$-regular, its universal cover is (isomorphic to) $T_d$, so let $p: T_d \to \G$ be a covering map. Fix a preimage $o'$ of $o$ under $p$. Then every subtree of \g containing $o$ lifts uniquely to a subtree of $T_d$ containing $o'$, and distinct subtrees of \g lift to distinct subtrees of $T_d$. This means that the number of subtrees of $G$ containing  $o$ is at most the corresponding number for $T_d$, which is less than $c_d  \left( \frac{(d-1)^{(d-1)}}{(d-2)^{(d-2)}} \right)^n$ by \eqref{Sn bound}. 
We can rewrite the fraction in the parenthesis as 
$$(d-1) (\frac{d-1}{d-2})^{(d-2)} = (d-1) (1+\frac{1}{d-2})^{(d-2)}< (d-1)e$$
to complete our proof.

\end{proof} 

{\bf Remark 1:} \Cr{lattice animals} implies that the number of $n$-vertex induced connected subgraphs of $G$ containing a fixed vertex, called (site) lattice animals in the statistical mechanics literature, or polyominoes in combinatorics, is upper-bounded by the same expression, since every such graph has at least one spanning tree, and no two distinct induced subgraphs share a spanning tree. In particular, we deduce that the growth rate of the number of site lattice animals of any graph of maximum degree $d$ is at most $(d-1)e$. 
In the special case where $G$ is the $\Z^d$ lattice this upper bound was proved in \cite{BaBaRoFor} with different arguments.

\medskip
{\bf Remark 2:} The number  $\mathcal S_n$ is known exactly: it is $\frac{d((d-1)n)!}{(n-1)!((d-2)n+2)!}$.\footnote{We thank Stephan Wagner for acquainting us with this formula.} This can be proved using analytic combinatorics. One can also arrive at \eqref{Sn bound} using Stirling's formula to approximate the factorials in the latter expression.

\comment{

\begin{proposition} \label{thanks Stephan}
$$\mathcal S_n = \frac{d((d-1)n)!}{(n-1)!((d-2)n+2)!}$$
\fe\ $d>2$.
\end{proposition}

	Using this, we calculate
$$\mathcal S_n= \frac{d((d-1)n)!}{(n-1)!((d-2)n+2)!} = d \frac{((d-1)n)!}{(n)!((d-2)n)!} \frac{n}{((d-2)n+1)((d-2)n+2)}.$$ 
Since $d>2$, the last term $\frac{n}{((d-2)n+1)((d-2)n+2)}$ is less than 1. 
Using this fact, and Stirling's formula $n! \sim \sqrt{2\pi n} (n/e)^n$ to approximate the above factorials, we obtain
$$\mathcal S_n < d \frac{((d-1)n/e)^{(d-1)n}}{(n/e)^{n} (((d-2)n)/e)^{(d-2)n}}.$$
Eliminating the factor $(n/e)^{(d-1)n}$ from both the enumerator and the denominator, the last expression simplifies to
\labtequ{Sn bound}{$\mathcal S_n < c_d  \left( \frac{(d-1)^{(d-1)}}{(d-2)^{(d-2)}} \right)^n$,}
}

\section{Appendix: complex analysis basics} \label{sec App CA}

In this appendix we list some classical facts in complex analysis used throughout the paper. They can be found in standard textbooks like \cite{Ahlfors}. 
The first two provide the standard technique for showing that a sum of analytic functions is analytic, a technique we employ many times throughout the paper.

\begin{theorem}\textbf{(Weierstrass Theorem)}  \label{thmWei}
Let $f_n$ be a sequence of analytic functions defined on an open subset $\Omega$ of the plane, which converges uniformly on the compact subsets of $\Omega$ to a function $f$. Then $f$ is analytic on $\Omega$. Moreover, $f'_n$ converges uniformly on the compact subsets of $\Omega$ to $f'$.
\end{theorem}

\begin{theorem}\textbf{(Weierstrass M-test)} \label{MT} 
Let $f_n$ be a sequence of complex-valued functions defined on a subset $\Omega$ of the plane and assume that there exist positive numbers $M_n$ with $|f_n(z)|\leq M_n$ for every $z\in\Omega$, and $\sum_{n} M_n<\infty$. Then $\sum_n f_n$ converges uniformly on $\Omega$.
\end{theorem}

The following is only used in \Sr{sec neg}, when we discuss the negative percolation threshold.

\begin{theorem}\textbf{(Hadamard's three circles theorem)} \label{HT} 
Let $f(z)$ be an analytic function on the annulus $r_1\leq |z| \leq r_2$. Let $M(r)=\sup\{|f(re^{it})|, t\in\mathbb{R}\}$ be the supremum of $|f(z)|$ over the circle of radius $r$. Then for every $r\in (r_1,r_2)$
$$M(r)\leq M(r_1)^{R}M(r_2)^{R'},$$
where $R=R(r_1,r,r_2)=\dfrac{\log r_2-\log r}{\log r_2-\log r_1}$ and $R'=R'(r_1,r,r_2)=\dfrac{\log r-\log r_1}{\log r_2-\log r_1}.$
\end{theorem}

\bibliographystyle{plain}
\bibliography{collective}
\end{document}